\documentclass[11pt,letterpaper]{article}

 \newcommand{\sv}[1]{}
 \newcommand{\lv}[1]{#1}
\usepackage{etoolbox}
\usepackage{tablefootnote}
\newcommand{\appendixText}{}
\newcommand{\toappendix}[1]{\gappto{\appendixText}{{#1}}}

\usepackage[lmargin=1.0in,rmargin=1.0in,bottom=1.0in,top=1.0in,twoside=False]{geometry}

\pdfoutput=1
\usepackage[utf8]{inputenc}
\usepackage[T1]{fontenc}
\usepackage[hypertexnames=false]{hyperref} %
\usepackage[svgnames]{xcolor}
\usepackage{lmodern}
\usepackage{sidecap}
\usepackage{fullpage}  %
\usepackage{color}
\usepackage[protrusion=true,expansion=true]{microtype}
\usepackage{amsfonts}
\usepackage{amsopn, amsmath}
\usepackage{amsthm}
\usepackage{mathtools}
\usepackage{bm}
\usepackage{thm-restate}
 \usepackage{authblk}
\usepackage{graphicx}
\usepackage{caption}
\usepackage{subcaption}
\usepackage{xspace}	
\usepackage{paralist}
\usepackage{enumerate}%
\usepackage{tabularx}
\usepackage[capitalize,noabbrev]{cleveref}
\crefname{Theorem}{Theorem}{Theorems}
\usepackage[makeroom]{cancel}
\newcommand{\df}{\coloneqq}

\newcommand{\PF}{P\!F}

\DeclarePairedDelimiterXPP\Aver[1]{\mathbb{E}}{[}{]}{}{%

#1
}

\DeclareFontFamily{U}{rcjhbltx}{}
\DeclareFontShape{U}{rcjhbltx}{m}{n}{<->rcjhbltx}{}
\DeclareSymbolFont{hebrewletters}{U}{rcjhbltx}{m}{n}

\let\aleph\relax\let\beth\relax

\DeclareMathSymbol{\aleph}{\mathord}{hebrewletters}{39}
\DeclareMathSymbol{\beth}{\mathord}{hebrewletters}{98}

\DeclareMathOperator{\Emb}{Emb}  %
\DeclareMathOperator{\emb}{emb}

\newcommand\E{\mathbb{E}}

\newcommand\floor[1]{\lfloor #1\rfloor}

\newcommand{\Vup}{V^\uparrow}
\newcommand{\Next}{R_i}

\newtheorem{theorem}{Theorem}[section]
\newtheorem{lemma}[theorem]{Lemma}
\newtheorem{corollary}[theorem]{Corollary}

\newtheorem{Definition}[theorem]{Definition}
\newtheorem{Conjecture}[theorem]{Conjecture}

\newtheorem{observation}[theorem]{Observation}
\newtheorem{claim}[theorem]{Claim}

\newcommand{\cev}[1]{\reflectbox{\ensuremath{\vec{\reflectbox{\ensuremath{#1}}}}}}%
\def\Uforw{\vec{U}}%
\def\Uback{\cev{U}}%

\begin{document}

\title{Random Embeddings of Graphs:\\ The Expected Number of Faces in Most Graphs is Logarithmic\thanks{An extended abstract of this manuscript appeared at ACM-SIAM Symposium on Discrete Algorithms (SODA 2024)~\cite{ourSODA}.}}
\author[1]{Jesse Campion Loth}
\author[1]{Kevin Halasz}
\author[1,2]{Tomáš Masařík\thanks{T.M.~was supported by a postdoctoral fellowship at the Simon Fraser University through NSERC grants R611450 and R611368.}}
\author[1,3]{\\ Bojan Mohar\thanks{B.M.~was supported in part by the NSERC Discovery Grant R611450 (Canada), by the Research Project J1-8130 of ARRS (Slovenia), and by the ERC Synergy grant KARST (European Union, ERC, KARST, project number 101071836).}}
\author[4]{Robert Šámal\thanks{
R.S.~was partially supported by grant 19-21082S of the Czech Science Foundation.
This project has received funding from the European Research Council (ERC) under the European Union's Horizon 2020 research and innovation programme (grant agreement No 810115). This project has received funding from the European Union's Horizon 2020 research and innovation programme under the Marie Skłodowska-Curie grant agreement No 823748.}}
\affil[1]{Department of Mathematics, Simon Fraser University, Burnaby, BC, V5A 1S6, Canada\\
\texttt{\{jcampion, khalasz, mohar\}@sfu.ca}}
\affil[2]{Institute of Informatics, Faculty of Mathematics, Informatics and Mechanics, University of Warsaw, Warszawa, 02-097, Poland\\
\texttt{masarik@mimuw.edu.pl}}
\affil[3]{Faculty of Mathematics and Physics, University of Ljubljana, Ljubljana, Slovenia}%
\affil[4]{Computer Science Institute, Faculty of Mathematics and Physics, Charles University, Praha, 118 00, Czech Republic\\
\texttt{samal@iuuk.mff.cuni.cz}}
\date{}

\maketitle

\vspace{-3em}
\begin{abstract}
A random 2-cell embedding of a connected graph $G$ in some orientable
surface is obtained by choosing a random local rotation around each
vertex. Under this setup, the number of faces or the genus of the corresponding 2-cell embedding becomes a random variable. Random embeddings of two particular graph classes -- those of a bouquet of $n$ loops and those of $n$ parallel edges connecting two vertices -- have been extensively studied and are well-understood. However, little is known about more general graphs despite their important connections with central problems in mainstream mathematics and in theoretical physics (see [Lando \& Zvonkin, Graphs on surfaces and their applications, Springer 2004]). There are also tight connections with problems in computing (random generation, approximation algorithms). The results of this paper, in particular, explain why Monte Carlo methods (see, e.g.,  [Gross \& Tucker, Local maxima in graded graphs of imbeddings, Ann. NY Acad. Sci 1979] and [Gross \& Rieper, Local extrema in genus stratified graphs, JGT 1991]) cannot work for approximating the minimum genus of graphs. 

In his breakthrough work ([Stahl, Permutation-partition pairs, JCTB
1991] and a series of other papers), Stahl developed the foundation of
``random topological graph theory''. Most of his results have been
unsurpassed until today. In our work, we analyze the expected number of
faces of random embeddings (equivalently, the average genus)
of a graph $G$.
It was very recently shown [Campion Loth \& Mohar, Expected number of faces in a random embedding of any graph is at most linear, CPC 2023] that for any graph $G$, the expected number of faces is at most linear.
We show that the actual expected number of faces $F(G)$ is almost always much smaller.
In particular, we prove the following results:
\begin{itemize}
  \item[(1)]  $\frac{1}{2}\ln n - 2 < \E[F(K_n)] \le 3.65 \ln n +o(1)$. This substantially improves Stahl's $n+\ln n$ upper bound for this case.
  \item[(2)] For random graphs $G(n,p)$ ($p=p(n)$), we have
    $\E[F(G(n,p))] \le \ln^2 n+\frac{1}{p}$.
\item[(3)] For random models $B(n,\Delta)$ containing only graphs, whose
maximum degree is at most $\Delta$, we obtain stronger bounds by showing that the expected number of faces is $\Theta(\log n)$.
\end{itemize}

\end{abstract}

\section{Introduction}

\subsection{Random embeddings of graphs in surfaces}
\label{sec:REintro}

Every 2-cell embedding of a graph $G$ in an (orientable) surface can be described combinatorially up to homeomorphic equivalence by using a \emph{rotation system}. This is a set of cyclic permutations $\{R_v \mid v\in V(G)\}$, where $R_v$ describes the clockwise cyclic order of edges incident with $v$ in an embedding of $G$ in an oriented surface. We refer to \cite{MT01} for further details. In this way, a connected graph $G$, whose vertices have degrees $d(v)$ ($v\in V(G)$), admits precisely
$\prod_{v \in V(G)} (d(v)-1)!$
nonequivalent 2-cell embeddings. 

Graph embeddings are of interest not only in topological graph theory but also within several areas of pure mathematics, physics and computing.  They are a fundamental concept in combinatorics (products of permutations, Hopf algebra, chord diagrams), algebraic number theory (algebraic curves, Galois theory, Grothendiek's ``dessins d'enfants'', moduli spaces of curves and surfaces), knot theory (Vassiliev knot invariants) and theoretical physics (quantum field theory, string theory, Feynmann diagrams, Korteweg and de Vries equation), we refer to \cite{LandoZvonkin} for details.  Every embedding of a graph can be described by a combinatorial map.  Random maps with a given number of vertices have been the subject of much recent study.  They have links with representation theory (conjugacy class products \cite{ChPi16,pippenger2006topological}) and probability theory (the Brownian map, see \cite{le2013uniqueness} and the references therein).  They also have applications in theoretical physics, via quantum gravity and matrix integrals, see \cite{gwynne2020random, zvonkin1997matrix} for introductions to these fields.  We will study the random maps obtained by randomly embedding a fixed graph or random graph.  Despite these being natural models in random graph theory and probability theory, they have received less attention.

Existing work on random embeddings of graphs in surfaces is mostly concentrated on the notion of the \emph{random genus} of a graph. By considering the uniform probability distribution on the set $\Emb(G)$ of all (equivalence classes of) 2-cell embeddings of a graph in (orientable) closed surfaces, we can speak of a random embedding and ask what is the expected value of its genus. The initial hope of using Monte Carlo methods on the configuration space of all 2-cell embeddings to compute the minimum genus of graphs \cite{GrRi91,GrTu79} quickly vanished as empirical simulations showed that, in many interesting cases, the average genus is very close to the maximum possible genus in $\Emb(G)$. The work of Gross and Rieper \cite{GrRi91} also showed that there can be arbitrarily deep local minima for the genus that are not globally minimum.
That result rules out traditional local-search algorithms. However it does not exclude search methods that have more significant random component, like the popular simulated annealing heuristic~\cite{simAnn}. Our results show that for almost all graphs, starting with a random embedding we would be very far from a minimum with extremely high probability. Therefore, any heuristic with strong randomness will with high probability lead toward an embedding with only a few faces (and so of large genus). Hence, our work gives strong theoretical evidence that such methods are very unlikely to be successful. Of course, if we restrict inputs to a particular graph class such algorithms may still work.
We conclude this paragraph with phrasing one of the main outcomes of our work; This paper provides a formal evidence that the Monte Carlo approach cannot work for approximating the minimum genus of graphs.

Unlike most previous works, we will not discuss the (average) genus but instead the (average) number of faces in random embeddings. Although the two variables are related linearly through Euler's formula, it turns out that the study of the number of faces yields a more appreciative view of certain phenomena that occur in this area.

\subsection{State-of-the-art}

Random embeddings of two special families of graph are well understood. The first one is a bouquet of $n$ loops (also called a \emph{monopole}), which is the graph with a single vertex and $n$ loops incident with the vertex. This family was first considered in a celebrated paper by Harer and Zagier \cite{HZ86} using representation theory. Several combinatorial proofs appeared later \cite{Ch11,Gross89,Ja87,Jackson1994_integral,Za95,Zagier2004}. By duality, the maps of the monopole with $n$ loops correspond to unicellular maps \cite{Ch11} with $n$ edges. The second well-studied case is the \emph{$n$-dipole}, a two-vertex graph with $n$ edges joining the two vertices; see \cite{Andrews1994,Chen2020,CR16,CMS12,Jackson1994_algebraic,Jackson1994_integral,KL93,RiThesis}. A more recent case gives an extension to the ``multipoles'' \cite{CHMMS22} using a result of Stanley \cite{St11}. Random embeddings in all these cases are in bijective correspondence with products of permutations in two conjugacy classes. A notable generalization of these cases appears in a paper by Chmutov and Pittel \cite{ChPi16}.   
Another well-studied case includes ``linear'' graph families, obtained from a fixed small graph $H$ by joining $n$ copies of $H$ in a path-like way, see \cite{GKMT18,St91JCTB} and references therein.

\begin{table}
  \small
\begin{subtable}{0.44\textwidth}
    \begin{tabular}{c||r|r|r|r}
        $n$ & 3&4&5&6\\\hline\hline
        $\emb(K_n) $ &$1$ & $2^4$& $6^5$&  $24^6$\\\hline
        $g=0$ &  1 & 2 & 0 & 0\\ 
        $g=1$ & 0 & 14 & 462 & 1,800 \\
        $g=2$ & 0 & 0 & 4,974 & 654,576\\
        $g=3$ &  0&0& 2,340 & 24,613,800 \\
        $g=4$ &  0& 0& 0 & 124,250,208 \\
        $g=5$ & 0 & 0 & 0 & 41,582,592\\\hline
        $\E(g)$ & 0 & 0.875 & 2.24 & 4.082
    \end{tabular}
    \caption{Genus distribution}
    \label{tab:genus distro}
\end{subtable}
~
\begin{subtable}{0.52\textwidth}
    \begin{tabular}{c||r|r|r|r|r}
      $n$ & 3&4&5&6&7\\\hline\hline
        $\emb(K_n) $ &$1$ & $2^4$& $6^5$&  $24^6$\\\hline
        $F=1$ &  0 & 0 & 2,340 & 41,582,592 \\ 
        $F=2$ & 1 & 14 & 0 & 0 \\
        $F=3$ & 0 & 0 & 4,974 & 124,250,208\\
        $F=4$ &  0 & 2 & 0 & 0 \\
        $F=5$ &  0 & 0 & 462 & 24,613,800  \\
        $F=6$ & 0 & 0 & 0 & 0 \\
        $F=7$ & 0 & 0 & 0 & 654,576\\
        $F=8$ & 0 & 0 & 0 & 0\\
        $F=9$ & 0 & 0 & 0 & 1,800 \\\hline
        $\E(F)$ & 2 & 2.25 & 2.517 & 2.836 & 3.1265\tablefootnote{This value was computed explicitly in~\cite[Table 3.1]{BCthesis7}.}
        \\\hline
        $\approx 2\ln{n}$ &2.2 & 2.77 & 3.22 & 3.58 & 3.89
    \end{tabular}
    \caption{Face distribution}
    \label{tab:face distro}
\end{subtable}
\caption{Data obtained by exhaustive computation concerning $K_n$ for $n \le 6$}
\label{tab:distros}
\end{table}

Here we discuss random graphs, including dense cases. One special case, which is of particular importance, is that of complete graphs. 
Looking at the small values of $n$, $K_3$ has only one embedding, which has two faces. It is easy to see that $K_4$ has two embeddings of genus 0 (with four faces) and all other embeddings have genus 1 and two faces. A brute force calculation using a computer gives the numbers for $K_5$ and $K_6$. 
They are collected in Table~\ref{tab:distros}.
The genus distribution of $K_7$ has been computed only recently \cite{BCHK16,BCthesis7} and there is no data for larger number of vertices.  The computed numbers for $K_n$ show that for $n \leq 7$ most embeddings have a small number of faces. The results of this paper show that, similarly to the small cases, most embeddings of any $K_n$ will have large genus and the average number of faces is not only subquadratic but it is actually proportional to $\ln n$.\footnote{We use $\ln n$ to denote the natural logarithm.} This is a somewhat surprising outcome, because the complete graph $K_n$ has many embeddings with $\Theta(n^2)$ faces. In fact, it was proved by Grannell and Knor \cite{GrKn10} (see also \cite{GrGr08} and \cite{GrGr09}) that for infinitely many values of $n$ there is a constant $c>0$ such that the number of embeddings with precisely $\tfrac{1}{3}n(n-1)$ faces is at least $n^{cn^2}$. All these embeddings are triangular (all faces are triangles) and thus of minimum possible genus. When we compare this result with the fact that 
\[
  |\Emb(K_n)| = ((n-2)!)^n = n^{\Theta(n^2)},
\]
we see that there is huge abundance of embeddings of $K_n$ with many more than logarithmically many faces.

Stahl \cite{Stahl1980} introduced the notion of \emph{permutation-partition pairs} with which he was able to describe partially fixed rotation systems.  Through the linearity of expectation these became a powerful tool to analyze what happens in average. In particular, he was able to prove that the expected number of faces in embeddings of complete graphs is much lower than quadratic.

\begin{theorem}[Stahl {\cite[Corollary 2.3]{Stahl1995}}]
\label{thm:Stahl_K_n}
    The expected number of faces in a random embedding of the complete graph $K_n$ is at most $n + \ln n$.
\end{theorem}

Computer simulations show that even the bound given in \cref{thm:Stahl_K_n} is too high. In fact, Mauk and Stahl conjectured the following.

\begin{Conjecture}[Mauk and Stahl {\cite[page~289]{Mauk1996}}]\label{con:Mauk}
    The expected number of faces in a random embedding of the complete graph $K_n$ is at most $2 \ln n+O(1)$.
\end{Conjecture}

For general graphs, a slightly weaker bound than that of \cref{thm:Stahl_K_n} was derived by Stahl using the same approach as in~\cite{Stahl1995}; it had appeared in~\cite{St91JCTB} a couple of years earlier. 

\begin{theorem}[Stahl {\cite[Theorem~1]{St91JCTB}}]
\label{thm:Stahl_any}
    The expected number of faces in a random embedding of any $n$-vertex graph is at most $n\ln n$.
\end{theorem}

The $n\ln n$ bound of Stahl was improved only recently. 
Campion Loth, Halasz, Masařík, Mohar, and Šámal \cite{CHMMS22} conjectured that the bound should be linear, which was then proved in \cite{loth2022expected}. 

\begin{theorem}[Campion Loth and Mohar~{\cite[Theorem 3]{loth2022expected}}]
\label{thm:CampionLothMohar_any}
    The expected number of faces in a random embedding of any graph is at most $\frac{\pi^2}{6} n$.
\end{theorem}

The bound of \cref{thm:CampionLothMohar_any} is essentially best possible as there are $n$-vertex graphs whose expected number of faces is $\frac{1}{3}n+1$, see \cite{loth2022expected}.

\subsection{Our results}\label{sub:ourRes}

The first main contribution of this paper is the proof of~\cref{con:Mauk} with a slightly worse multiplicative factor.

\begin{theorem}\label{thm:completeGraphsAllValues}
  Let $n\geq 1$ be an integer and let $F(n)$ be the random variable whose value is the number of faces in a random embedding of the complete graph $K_n$. 
  The expected value of $F(n)$ is at most $10\ln n+2$.
  For $n$ sufficiently large ($n\ge e^{e^{16}}$) the multiplicative constant can be improved to give a better bound:
  \[
    \E[F(n)]\le 3.65\ln n.
    \]
\end{theorem}

We complement our upper bound with a lower bound showing that our result is tight up to the multiplicative factor.

\begin{restatable}{theorem}{loglb}
  \label[Theorem]{thm:lb}
  For all positive integers $n$, we have 
  \[
    \E[F(n)]> \tfrac{1}{2}\ln(n)-2.
    \]
\end{restatable}

In order to prove~\cref{thm:completeGraphsAllValues}, we split the proof into ranges based on the value of $n$ and use a different approach for each range.
In fact, we provide two theoretical upper bounds using a close examination of slightly different random processes.
The first one is easier to prove, but it gives an asymptotically inferior bound.
However, it is useful for small values of $n$. 
In the bound, we use the harmonic numbers $H_k\df \sum_{j=1}^k \frac{1}{k}$, whose value is approximately equal to $\ln n$.

\begin{restatable}{theorem}{logsqbnd}
  \label[Theorem]{thm:logsqBound}
  Let $n\geq 10$ be an integer.
  Then \[\E[F(n)] < H_{n-3}H_{n-2}.\]
\end{restatable}

Note that proof of \cref{thm:logsqBound} also works for $n\ge 4$, but yields a slightly worse bound (see \cref{eq:logsqForComputation}), which we have not stated above.
Moreover, we used \cref{eq:logsqForComputation} to estimate values for $n\le 242$ using computer and this implies \cref{thm:completeGraphsAllValues} ($\E[F(n)]\le5\ln n +5$) for this range; see \sv{\cite[Section 5]{ourArxiv}}\lv{\cref{sec:smallValues}} for the details.

\def\ComputThr{40748}
The next theorem is our core result that implies \cref{thm:completeGraphsAllValues} for $n> \ComputThr$. 

\begin{restatable}{theorem}{logbnd}
  \label{thm:logBound}
  For $n \geq e^{e^{16}}$, $\E[F(n)] \leq 3.65\ln(n)$.
  For $n \geq e^{30}$, $\E[F(n)] \leq 5\ln(n)$.
  For $e^{10.6}\approx\ComputThr \leq n < e^{30}$, $\E[F(n)] \leq 10\ln(n)+2$.%
\end{restatable}

For small values of $243\le n\le \ComputThr$ we used a computer-assisted proof which is based on our general estimates given in the proof of~\cref{thm:logBound} combined with pre-computed bounds for smaller values of $n$ and Markov inequality. 
We will give more details on our computation in~\sv{\cite[Section 5]{ourArxiv}}\lv{\cref{sec:smallValues}}.
We summarize the results of computer-calculated upper bounds in the following proposition.
Note that having a small additive constant for small values of $n$ helps us to keep smaller additive constants for middle values of $n$ as our proof is inductive.

\begin{restatable}{proposition}{propComput}
  \label[Theorem]{prop:smallValues}
 For $1\leq n \leq \ComputThr$, $\E[F(n)] \leq 5\ln(n)+5$.
\end{restatable}

In summary, the proofs of the above results for complete graphs are relatively long. A ``log-square'' improvement of Stahl's linear bound is not that hard, but the $O(\log n)$ bound appears challenging and shows all difficulties that arise for more general dense graph classes. 

\bigskip

In the second part of the paper, we turn to more general random graph families. 
Let $F(n,p)$ be the random variable for the number of faces in a random embedding of a random graph in $G(n,p)$.
We will first show a bound on the expectation of this variable which holds for any value of $p$.

\begin{restatable}[]{theorem}{randomGraphs}
  \label[Theorem]{thm:randomgraphs}
  Let $n$ be a positive integer and $p \in (0,1]$ $(p = p(n))$.  Then we have 
  \[
    \E[F(n,p)] \leq H_n^2 + \frac{1}{p}.
  \]
\end{restatable}

\cref{thm:randomgraphs} gives a ``log-square'' general bound which can be improved in the sparse regime as well as in the dense regime (for multigraphs).
First, we state a general result for random embedding of random maps with fixed degree sequence.
In other words, we will investigate random embeddings of random multigraphs possibly with loops sampled uniformly out of multigraphs with the same fixed degree sequence.
Some results of this flavor have been obtained earlier in the setup of ``random chord diagrams'', see \cite{ChPi13,LiNo11}.

\begin{restatable}[]{theorem}{thmmulti}
\label{thm:multigraphs}
Let ${\mathbf d} = (t_{1},t_{2},\dots,t_{n})$ be a degree sequence for an $n$-vertex multigraph (possibly with loops) where $t_i \ge 2$ for all $i$.
Let $\E[F_{\mathbf d}]$ be the average number of faces in a random embedding of a random multigraph with degree sequence~${\mathbf d}$.
Then $\E[F_{\mathbf d}]=\Theta(\log n)$.
\end{restatable}

However, we are mostly interested in simple graphs.
For larger degree sequences, the majority of random embeddings generated in the model of Chmutov and Pittel \cite{ChPi13,LiNo11} will not be simple. 
Therefore, we will be focusing on degree sequences with bounded parts while we allow $n$ to grow to infinity. 
Given a degree sequence ${\mathbf d} = (t_{1},t_{2}, \dots, t_{n})$, let
\[
  m_{\mathbf d} = \frac 12 \sum_i t_{i} \quad \textrm{and} \quad \lambda_{\mathbf d} := \frac{1}{2m_{\mathbf d}}\sum_{i=1}^n \binom{t_{i}}{2}.
\]
Janson \cite{Janson2009} showed that a random multigraph with degree sequence ${\mathbf d}$ is asymptotically almost surely \emph{not} simple unless $\lambda_{\mathbf d} = O(1)$. This means, for example, that the probability of a $d$-regular multigraph on $n$ vertices being simple is bounded away from 0 only if $d$ is constant (while $n$ grows arbitrarily). 
Restricting our attention to the case where vertex degrees are bounded by an absolute constant, Janson's result tells us that simple graphs make up a nontrivial fraction of all multigraphs with a given degree sequence. In fact, this special case of Janson's result was obtained over 30 years earlier by Bender and Canfield \cite{BC78}. 
We prove that, in the case of random \emph{simple} graphs with constant vertex degrees, we preserve logarithmic bounds on the expected number of faces.
\begin{restatable}[]{theorem}{thmsimple}
  \label{thm:randomgraphsmallp}
  Let $d\ge 2$ be a constant, $\varepsilon>0$, and 
  let ${\mathbf d} = (t_{1},t_{2},\dots,t_{n})$ be a degree sequence for some $n$-vertex simple graph with $2 \le t_i \le d$ for all $i$, and such that $m_{\mathbf d}\ge (1+\varepsilon)n$.
Let $\E[F_{\mathbf d}^s]$ be the average number of faces in a random embedding of a random simple graph with degree sequence~${\mathbf d}$.
Then $\E[F_{\mathbf d}^s] = \Theta_{\varepsilon}(\log n)$ (constants within $\Theta$ depend on $\varepsilon$).
\end{restatable}

In the light of the above theorems and our Monte Carlo experiments, we conjecture that a logarithmic upper bound should be achievable for any usual model of random graphs.
However, extending our proof of~\cref{thm:completeGraphsAllValues} to arbitrary random graphs seems to require further ideas. 

\begin{Conjecture}\label{con:random}
  Let $p=p(n)$ be the probabillity of edges in $G(n,p)$.
  The expected number of faces in a random embedding of a random graph $G\in G(n,p)$ is \[(1+ o(1))\ln(pn^2).\]
\end{Conjecture}

We refer to \cref{sec:open} for further discussion on conjectures and open problems that are motivated by our results.

\bigskip

\paragraph{Structure of the paper.}
Before we dive into proofs we will present our common strategy and formalization used in Theorems~\ref{thm:logsqBound} and \ref{thm:logBound} in~\cref{sec:strategy}.
First, we present the easier proof of \cref{thm:logsqBound} in \cref{sec:logsqBound}.
Our main result (\cref{thm:logBound}) on complete graphs can be found in \cref{sec:logBound}.
{In \cref{sec:smallValues}, we describe how the estimates presented in~\cref{sec:logBound} were used to compute the bounds for small values of $n$ using computer evaluation.
We conclude the complete graph sections with a short proof of our lower bound (\cref{thm:lb}) in \cref{sec:lb}.  
The proof of~\cref{thm:randomgraphs} is given in \cref{sec:random}. 
The paper closes with proofs of~Theorems~\ref{thm:multigraphs} and~\ref{thm:randomgraphsmallp} in \cref{sec:randomBnd}.
In \cref{sec:open}, we discuss conjectures and open problems.

\subsection{Preliminaries}\label{sub:prelim}

\medskip
\paragraph{Combinatorial maps.~}
To describe 2-cell embeddings of graphs we need a formal definition of a map.  A \emph{combinatorial map} (as introduced in \cite{jacques1970constellations,ringel2012map}) is a triple $M=(D,R,L)$ where
\begin{itemize}
    \item $D$ is an abstract set of \emph{darts};
    \item $R$ is a permutation on the symbols in $D$;
    \item $L$ is a fixed point free involution on the symbols in $D$.
\end{itemize}
Combinatorial maps are in bijective correspondence with $2$-cell embeddings of graphs on oriented surfaces, up to orientation-preserving homeomorphisms. 
See \cite[Theorem 3.2.4]{MT01} for a proof.  
We give details of this correspondence.  Let $G=(V,E)$ be a graph on $n$ vertices, where $V=\{v_1, \dots, v_n\}$. 
\begin{itemize}
    \item For $i \in [n]$, let $D_i$ be the set of all pairs $(v_i,e)$ where $e$ is an edge incident with $v_i$.  Note that $|D_i|=t_i$ is the degree of $v_i$.  Let $D = D_1 \cup \dots \cup D_n$ be the set of all darts.
    \item For each $i\in [n]$, we let $R_i$ be a unicyclic permutation of darts in~$D_i$, in clockwise order as they emanate from $v_i$ on the surface.  So, $R_i(d)$~is the dart following $d$ in the clockwise order given by~$R_i$, and conversely $R_i^{-1}(d)$ is the dart preceding $d$ in this cyclic order.   We let $R=R_1 R_2 \cdots R_n$, and call $R$ a \emph{rotation system}.
    \item We let $L$ be a permutation of~$D$ consisting of 2-cycles swapping $(v_i,e)$ with $(v_j,e)$ for each edge $e=v_iv_j$.  We call $L$ an \emph{edge scheme}.
    \item The cycles of the permutation $R \circ L$ give the faces of the embedding.
\end{itemize}

Conversely, starting with a combinatorial map $M=(D,R,L)$, we define the graph whose vertices are the cycles of $R$, and whose edges are the 2-cycles of $L$.

\medskip

\paragraph{Random embeddings.~}

Fix an arbitrary edge scheme $L$ of a graph $G$.  It is well known that all 2-cell embeddings of $G$, up to homeomorphism, are given by the set of all $(D,R,L)$ over all rotation systems $R$.  We call an embedding chosen uniformly at random from the set of these maps a \emph{random embedding} of $G$.

Now fix some rotation system $R$. 
Intuitively, given $G$ we know what vertices are connected by an edge, say $uv\in E(G)$, but within the dart model, we do not know what particular dart incident with $u$ connects to a particular dart of $v$.
Hence, we argue that we can model a random embedding of $G$ just by picking what darts form the edges uniformly at random.
Indeed, a simple counting argument shows that for $G$ with degree sequence $t_1,\ldots,t_n$, there are $t_1!t_2!\dots t_n!$ possible edge schemes.  Moreover, each embedding of $G$ is given by $t_1t_2 \cdots t_n$ different edge schemes.  In particular, each embedding is given by the same number of edge schemes.  Therefore, we may also obtain a random embedding of $G$ by fixing some rotation system $R$ and picking a uniform at random edge scheme.  This is the model, which we will use in \cref{sec:logsqBound}. 

Thirdly, we may vary both the local rotation and the edge scheme.  Picking a uniform at random rotation system and edge scheme also gives a random embedding of $G$.  This is the model, which we will use in \cref{sec:logBound}.

\medskip

\paragraph{Partial maps and temporary faces.~}

Our proofs will involve building up a map step by step.  Therefore we will need a notion of a partially constructed map.  A \emph{partial map} is defined in the same way as a map $(D,R,L)$, except $L$ need not be fixed point free.  We define the darts that are in 2-cycles in $L$ as \emph{paired darts} and the darts that are fixed points in $L$ as \emph{unpaired darts}\phantomsection\label{def:unpaired}.  

The \emph{faces} of the implied embedding of a map $M=(D,R,L)$ are given by the orbits of $R \circ L$. One of our main interests in this paper will be the number of faces.  In a partial map, each cycle in $R \circ L$ may contain some number of unpaired darts and/or paired darts.  For a partial map $(D,R,L)$, a cycle of $R \circ L$ is a \emph{completed face} if it contains only paired darts, and a \emph{temporary face} if it contains at least one unpaired dart.  In particular, we say a temporary face is \emph{$k$-open} if it contains precisely $k$ unpaired darts.  We say that a temporary face $f$ is \emph{strongly $2$-open} if $f$ is $2$-open and the two unpaired darts in $f$ are incident with different vertices.

Our proofs are often stated in terms of \emph{facial walks}.  For a completed face, this is simply a walk around the boundary of the face.  For temporary face, this is a walk where we travel along the paired darts which make up edges, but walk through any unpaired dart.  Let $f$ be a $k$-open face and let $d_1, d_2, \ldots, d_k$ be the unpaired darts that belongs to $f$ in their anti-clockwise order of appearance on a facial walk around $f$. For each $i$ ($1\le i\le k$), we call the segment of a facial walk around $f$ from $d_i$ to $d_{i+1}$ the \emph{partial facial walk} (\emph{partial face}) with initial dart $d_i$ and ending dart $d_{i+1}$. (We also say that this partial face \emph{leads from $d_i$ to $d_{i+1}$}. Note that each unpaired dart is the initial dart for precisely one partial face and is also the ending dart of precisely one partial face.

 We will use the following precise estimate for the \emph{harmonic numbers} $H_n \df 1 + \frac12 + \cdots + \frac{1}{n}$.
\begin{theorem}[Estimate of $H_n$~\cite{DeT93}]\phantomsection\label{thm:Hnconv}
  Let $n\ge 1$.
  \[
    H_n = \ln\left(n+\tfrac{1}{2}\right) + \gamma + \varepsilon_n,
  \]
where $\frac{1}{24(n+1)^2} \le \varepsilon_n \le \frac{1}{24n^2}$ and $\gamma\approx 0.57721$ is the Euler–Mascheroni constant.
\end{theorem}

The above lower bound works also for $n=0$ since $H_0 = 0 \ge \ln\left(\tfrac{1}{2}\right) + \gamma + \tfrac{1}{24}$.

\newcommand{\ConstC}[0]{\ensuremath{\sqrt{12}}}
\newcommand{\ConstA}[0]{\ensuremath{\sqrt{\frac{4}{3}}}}
\newcommand{\ConstAnew}[0]{\ensuremath{\sqrt{\frac{8}{3}}}}
\newcommand{\ConstInduction}[0]{\aleph}
\newcommand{\ConstInductionLarge}[0]{\beth}
\newcommand{\ConstB}[0]{\ConstInduction}
\newcommand\comSplitP{22456}

\section{Our proof strategy for complete graphs}\label{sec:strategy}

We give two proofs of a bound on $\E[F(n)]$.
One gives an asymptotically worse bound, but will be useful to give the best estimates for small values of $n$.
The other one is more involved and requires rather tedious computation.  Here we present an intuition on both proofs and introduce a bit more terminology.

\medskip

\paragraph{Log-square bound.}
For this proof, we will fix an arbitrary rotation system and pick a uniform at random edge scheme.  We will work with a random process that builds a random edge scheme step by step.  First, we order the vertices of a graph $G$ arbitrarily.  We represent the ordering as $v_1,v_2,\dots,v_n$, and we process vertices one by one, starting with $v_n$.  When processing a vertex $v_k$, since we fixed a rotation system, the cyclic rotation of darts in $D_k$ is fixed.  We process the darts incident with $v_k$ in this fixed order.  At each step we either keep this dart unpaired or pick another random dart to pair this dart with to make a $2$-cycle in $L$, as defined precisely in \hyperref[def:rpA]{Random Process~A}.  An analysis of this process in Section~\ref{sec:logsqBound} will give Theorem~\ref{thm:logsqBound}.

\medskip

\paragraph{Logarithmic bound.}
\lv{
Our main tool is a simplified and refined approach of Stahl~\cite{Stahl1995}, which gives the best previously-known upper bound of $n+\ln(n)$ stated in \cref{thm:Stahl_K_n}.  Stahl's strategy can be summarized as follows.
First, we order the vertices of a graph $G$ arbitrarily.
We represent the ordering as $v_1,v_2,\dots,v_n$, and we process vertices one by one, starting with $v_n$.
For each $i\in [n]$, we process all darts in $D_i$, one-by-one.
In this way, we construct a decision tree of all possible embeddings.
In each node $t$ of the decision tree, one of yet unprocessed darts $d$ is processed, which means we enumerate all available choices for $\Next(d)$.
For each such choice, we obtain a son of the node $t$ in the decision tree.
It is straightforward to verify that at most one such choice actually \emph{completes a face}, i.e., both $d$ and $\Next(d)$ are part of a face in $R \circ L$ which was completed once $\Next(d)$ was determined. 
We formalize this fact as the following observation, which can be attributed to Stahl~\cite{Stahl1995}. 

\begin{observation}[\cite{Stahl1995} (reformulated)]\label{o:at-most-one-face}
  Let $d$ be a dart of $G$ such that $\Next(d)$ is undefined.
  Among all valid choices for $\Next(d)$ at most one completes a face containing $d$. 
\end{observation}

Observe that the above-described procedure generates a decision tree where the leaves are uniformly random embeddings of~$K_n$. 
}

In this proof, we use a more refined random process to generate a random rotation system, and a random edge scheme.  We then conclude by rather complicated computation.  In a similar manner to the previous description, we process vertices one at at time, and process darts one at a time at each vertex.
 
When we process $v_k$, we refer to it as \emph{step $k$}.  For each $k\in[n-2]$, we define the following terminology.  Let $\Vup$ be vertices $v_n,\ldots,v_{k+1}$ and $D_{\Vup}$ be set of their darts.  Recall that dart $d$ is \hyperref[def:unpaired]{unpaired} if $L(d)$ is undefined.
Now, we make the following random choice.
For each $i>k$, we choose uniformly at random an unpaired dart $d_i\in D_i$ and we define $L(d_i):=d$ for some unpaired dart $d\in D_k$.  We call all such newly paired darts \emph{active} for this step.  Observe that $k-1$ darts remain unpaired at vertex $v_k$ in this step.

We then study how many of various types of active darts we expect to obtain from this random choice.  Based on this, we randomly build a rotation system at $v_k$.  We do this step by step: we fix some processing order of the darts in $D_k$.  Then for each dart $d$ in this order, we randomly choose a value of $R(d)$.  This will be defined precisely as \hyperref[rpb]{Random Process B}.  Analysing the probability of adding a completed face to the embedding when assigning each value of $R(d)$ will give the proof of Theorem \ref{thm:logBound}.

\section{Log-square bound---proof of \cref{thm:logsqBound}}\label{sec:logsqBound}

We start by proving \cref{thm:logsqBound}.

\logsqbnd*

We will use a similar approach for the proof of \cref{thm:randomgraphs} \sv{in~\cite[Section 7]{ourArxiv}}\lv{ later in \cref{sec:random}}. 
Refer to \cref{fig:logsqprocess} for an example of this random process.

\bigskip

\pagebreak[3]
\noindent
{\bf Random process A.}\phantomsection\label{def:rpA}

\begin{figure}
    \centering
    \includegraphics[scale = 0.8]{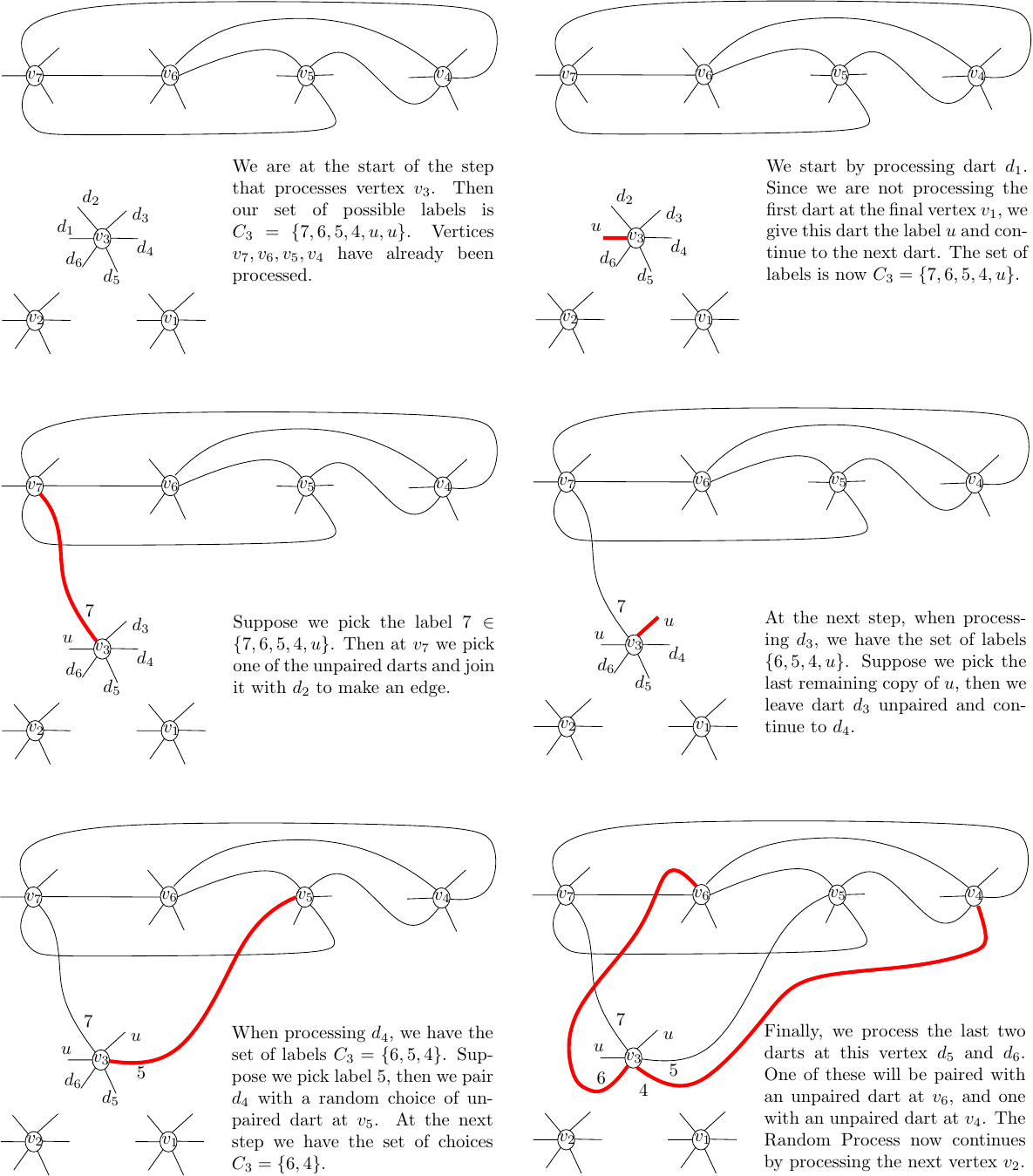}
    \caption{An example of \hyperref[def:rpA]{Random Process A}, processing vertex $v_3$.}
    \label{fig:logsqprocess}
\end{figure}

\begin{enumerate}
\item Order the vertices of the graph $v_{n}, \dots, v_1$ arbitrarily and process the vertices in this order. 
\item Start with vertices $v_n$ and $v_{n-1}$.
  They belong to one temporary face and no face has been closed so far.\label{step:ra:special1}
\item Consider vertex $v_k$ for $k\in[n-2]$.
  \label{step:ra:vertex1}
  Label the darts of~$D_k$ as $\{d_1, \dots, d_{n-1}\}$ arbitrarily. We define $R_k$ as this cyclic order, that is 
    $R_k(d_i) = d_{i+1}$ (except $R_k(d_{n-1}) = d_{1}$).
        Let $C_k \df \{n,n-1,\dots,k+1, u,u,\dots,u\}$ where there are $k-1$ copies of the symbol $u$ representing that the dart choosing $u$ remains unpaired.
        This is the multi-set of choices of where the darts may lead at the end of this step.
    \begin{enumerate}
      \item Process darts in $D_k$ in order $d_1, d_2, \dots, d_{n-1}$.  If $k>1$, give $d_1$ the label $u$, remove one copy of $u$ from $C_k$, and proceed processing $d_2$.  If $k=1$, start by processing $d_1$.
                \item Consider the dart $d_\ell$ which is next in the order.
                {\bf Random choice 1a:}\phantomsection\label{def:rpA:rca}
                   Pick a symbol from the set $C_k$ uniformly at random, then remove this choice from $C_k$. 
                   \begin{itemize}
                     \item Case 1: \emph{The choice was some  $i \geq k+1$}. {\bf Random choice 1b:}\phantomsection\label{def:rpA:rcb} Then pick an unpaired dart $d'$ uniformly at random from those at $v_i$.
                       Then add the transposition $(d',d_\ell)$ to the permutation $L$.
                     \item Case 2: \emph{The choice was some $u$.}  Then leave dart $d_\ell$ unpaired.
                   \end{itemize}
                   Continue to the next dart in the order.
    \end{enumerate}
                   Continue to the next vertex in the order. \hfill $\lrcorner$
\end{enumerate}

For each value of $k\le n-2$, let $F_k$ ($F_k=F_k(n)$) be the number of faces completed at step $k$.
    By this, we mean the facial walks that contain $v_k$ and
    no vertex $v_j$ with $j<k$.  They were completed at step $k$ and have stayed unchanged until the end of the process. 
    We need an upper bound on $\E[F_k]$. By linearity of expectation, we have that $\E[F(n)] = \sum_{k=1}^{n-2} \E[F_k(n)]$.
    
    Suppose we are processing the dart $d_\ell$ at step $k$. Recall that $d_\ell$ is contained in two partial faces: one starting at some dart $d$ and ending at $d_\ell$, and one starting at $d_\ell$ and ending at some dart $d'$.  We complete a face at this step if and only if we pair $d_\ell$ with dart $d$ or $d'$.  The dart $d'$ is an unpaired dart incident with $v_k$ with a single exception when $k=1$ and $\ell=n-1$.  So pairing $d_\ell$ with $d'$ cannot have completed a face unless we have this exception.  We have two cases:

    \textbf{Case 1:}\phantomsection\label{s3:c1} \emph{$\ell=1$, or the previously processed dart $d_{\ell-1}$ was chosen to be unpaired:} Then both darts $d$ and $d'$ are incident with vertex $v_k$, so we cannot pair with them.  Therefore, we cannot have completed a face when processing $d_\ell$.

    \textbf{Case 2:}\phantomsection\label{s3:c2} \emph{$d_{\ell-1}$ is paired:} See Figure \ref{fig:knrandomprocessstep} for an example of this analysis.  We complete a face at this step if and only if we pair $d_\ell$ with $d$, where $d$ is the dart at the start of the partial face ending at $d_\ell$. %
    The probability we choose~$d_\ell$ to lead to vertex~$v$ incident with $d$ is at most $\frac{1}{n-\ell}$ as we have already chosen at most $\ell-1$ vertices in \hyperref[def:rpA:rca]{Random choice 1a}.
    The probability that we choose dart~$d$ (and not another unused dart at~$v$) to connect with~$d_{\ell}$ is $\frac{1}{k}$ as there are $k$ unpaired darts incident vertex $v$ to choose from in \hyperref[def:rpA:rcb]{Random choice 1b}. 
    Therefore, the probability that we complete the face is at most $\frac{1}{k(n-\ell)}$.

    \textbf{Case 3:}\phantomsection\label{s3:c3} \emph{$k=1$:} When processing $d_{n-1}$, the dart $d'$ at the end of the partial face starting at $d_{n-1}$ is not at $v_1$.  Therefore, we can close two faces at this step. %

    \smallskip

    Assume now $k>1$. 
    Each dart (except for $d_1$) has probability $\frac{n-k}{n-2}$ of being paired (as $d_1$ is unpaired). 
    Thus a dart $d_{\ell}$ ($\ell \ge 3$) has the same probability $\frac{n-k}{n-2}$ of being \hyperref[s3:c2]{Case~2}. Therefore, the probability that 
    we close a face by pairing up $d_{\ell}$ is at most $\frac{n-k}{n-2} \cdot \frac{1}{k(n-\ell)}$.

    For $k=1$, all edges are connected to~$\Vup$, thus the probability of closing a face by $d_\ell$ (for $\ell \ge 2$ now) 
    is $\frac{1}{n-\ell}$.
    Moreover, the last dart $d_{n-1}$ can close two faces as described in \hyperref[s3:c3]{Case 3}. 

    Summing over all values of $\ell$ we get for $k \ge 2$ and $n\geq 4$
    \[
      \E[F_k] \le \sum_{\ell=3}^{n-1} \frac{n-k}{n-2} \cdot \frac{1}{k(n-\ell)} = \frac{n-k}{k(n-2)} \cdot H_{n-3}.
    \]
    Also, 
    \[
      \E[F_1] \le 1 + \sum_{\ell=2}^{n-1} \frac{1}{n-\ell} = 1 + H_{n-2}. 
    \]

    Summing over all steps $k$ and assuming $n\ge 4$ (apart from the last line where we assume $n\ge 10$), we obtain:

    \begin{align}
      \E[F] &=     \E[F_1] + \sum_{k=2}^{n-2} \E[F_k] \nonumber\\ 
            &\leq  1 + H_{n-2} +  \sum_{k=2}^{n-2} \frac{n-k}{k(n-2)}\, H_{n-3} \nonumber\\
            &=     1 + H_{n-2} +  \frac{n}{n-2} H_{n-3} ( H_{n-2} - 1) - \frac{n-3}{n-2} H_{n-3} \label{eq:logsqForComputation} \\
            &<     H_{n-3}H_{n-2}.  &(\text{for } n \geq 10) \qedhere\nonumber
    \end{align}

\begin{figure}
        \centering
        \includegraphics[scale=0.75]{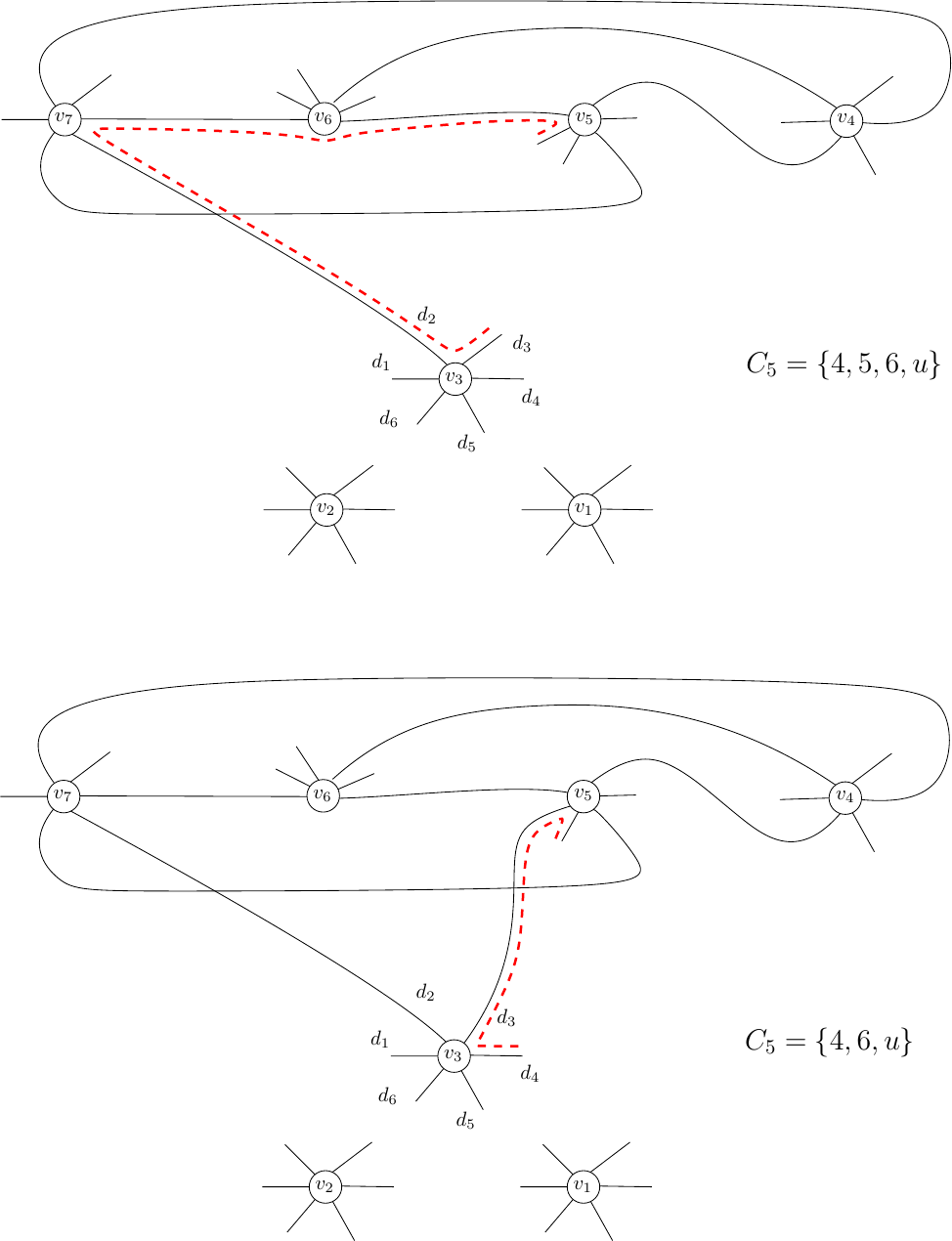}
        \caption{The upper diagram shows the step of Random Process $A$ where we are processing dart $d_3$ at vertex $v_3$.  The partial facial walk is traced in dotted red line, showing the only dart for which pairing with makes a completed face.  At the next step, the only dart for which pairing with makes a completed face is at vertex $v_5$.  However, we have already added the edge from $v_3$ to $v_5$, so $3$ is not a valid choice of a label at this step.  Therefore, we cannot add a completed face at this step.}
        \label{fig:knrandomprocessstep}
    \end{figure}

    \section{Logarithmic bound---proof of \cref{thm:logBound}}\label{sec:logBound}

\logbnd*

    We first introduce more notation that will be needed in the proof.
We look more carefully at step $k$.  At this step the walks in $R\circ L$ can be split into two categories building on notation defined in \cref{sub:prelim}:
    \begin{enumerate}
    \item \emph{Completed faces:} cycles of $R \circ L$. Those are closed walks that corresponds to 0-open faces which will not change any more, and
    \item \emph{Candidate walks:} those are partial faces that originates at an unpaired dart $d_s$ and lead to an unpaired dart $d_e$ (possibly $d_s=d_e$). 
    \end{enumerate}
    For each vertex in $\Vup$, we will pick an \emph{active dart} randomly from the set of all unpaired darts incident with this vertex.
Observe that if a partial face starts with a dart~$d_s$ and ends with~$d_e$, then it can complete a face in step~$k$ only if 
both $d_s$ and~$d_e$ become active.
We call such walks \emph{active} in step $k$.
We further partition the active walks into 
\begin{enumerate}
  \item[(1)] Those for which $d_s=d_e$. Observe that such are necessarily 1-open faces and so we refer to them as \emph{1-open active faces}, and 
  \item[(2)] All other active walks (i.e., $d_s\neq d_e$), which we refer to as \emph{potential faces}.
  \end{enumerate}
An active dart $d\in D_k$ is called \emph{1-open} if $L(d)$ is the dart incident with some 1-open face.
An active dart $d\in D_k$ is called \emph{potential} if $L(d)$ is incident with some potential face.
We will give more intuition on our terminology.
We will show that under certain circumstances, only potential faces may complete a face.
Therefore, we call unpaired darts in $D_k$ together with darts that do not take part in any active walk \emph{non-contributing}.
Let $\PF_k$ be a random variable representing the number of potential faces and $O_k$ be a random variable representing the number of 1-open active faces
in step~$k$, after active darts were chosen.  Let $F_k$ denote the total number of completed faces added during step $k$.

We now describe our random procedure in detail.
We refer to Figure \ref{fig:rpb} for an example of this random process.

\bigskip

\noindent
{\bf Random process B.}\phantomsection\label{rpb}

\begin{figure}
    \centering
    \includegraphics[scale = 0.8]{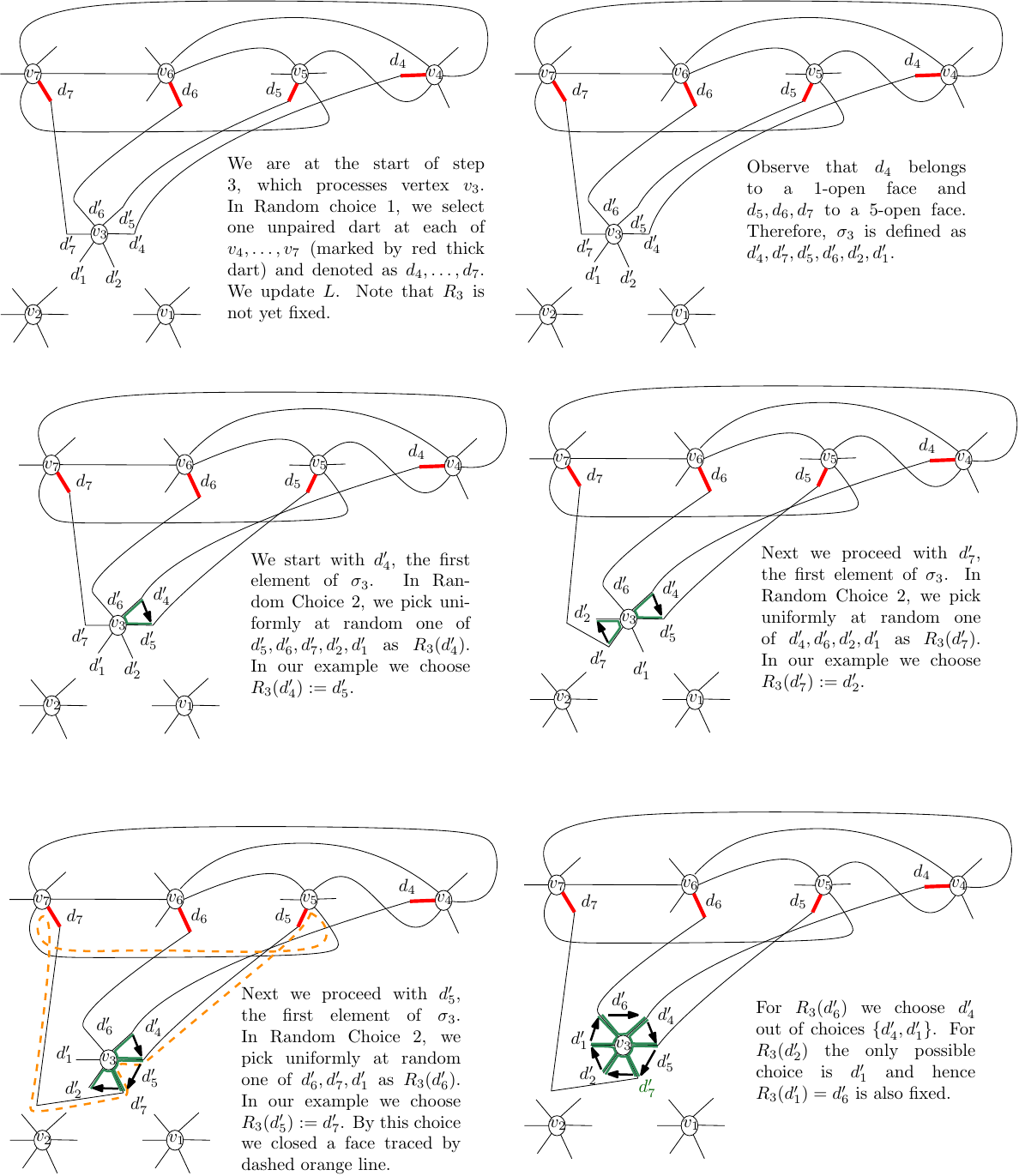}
    \caption{An example of \hyperref[rpb]{Random Process B}, processing vertex $v_3$ to obtain $R_3$.  At the end of this step, the darts $d_1', d_2'$ remain unpaired.  It is not decided which one will go to $v_1$ and which one will go to $v_2$.}
    \label{fig:rpb}
\end{figure}

\begin{enumerate}
\item Label the vertices arbitrarily as $v_n,\dots,v_1$ and process them in that order. 
\item Start with vertex $v_n$, and fix a uniform at random full cycle $R_n$.  This vertex is incident with $n-1$ unpaired darts.
\item Consider vertex $v_k$ for $k\in[n-1]$, starting with $n-1$.\label{step:vertex2}
    \begin{enumerate}
      \item\label{step:3a} \textbf{Random choice 1:}\phantomsection\label{rpb:rc1} For each vertex in $\Vup$ we choose uniformly at random one out of $k$ unpaired darts to lead to~$v_k$ and update $L$ appropriately.  The chosen darts are said to be the \emph{active darts} at step $k$.
      \item We treat $D_k$ as an unordered set, and build a local rotation $R_k$ by processing the darts in a special order $\sigma_k$ given by the type of walk the dart describes.  Each time we fix $R_k(d)$ for the processed dart $d$.
        We define $\sigma_k$ as follows:
              \begin{enumerate}
                \item 
                  First, process 1-open darts in arbitrary order.
                \item 
                  Next, potential darts follow in arbitrary order.
                \item 
                  Last, non-contributing darts are processed, again in arbitrary order.
              \end{enumerate}            
            \item \textbf{Random choice 2:}\phantomsection\label{rpb:rc2} For each $d\in D_k$ in order $\sigma_k$ we choose uniformly at random one dart $d'$ among all possible options (those
                   that do not violate the property that $R_k$ will define a single cycle eventually) and we set $R_k(d)\df d'$. \hfill $\lrcorner$
    \end{enumerate}
\end{enumerate}

Now, we define a function $q$, which will form an upper bound for the contribution of vertex $v_k$ to the expected number of faces. The function is defined as follows.
(Note that $H_0 = 0$.) 

\begin{Definition}\phantomsection\label{def:q}
If\/ $1\le t < n$ and\/ $0\le\xi<n-1-t$, then
 \begin{align}
   q(\xi,t)\df \sum_{i=1}^{t} \frac{1}{n-\xi-i-1} = H_{n-\xi-2} - H_{n-\xi-t-2}. \label{eq:defQ}
\end{align}
If $\xi+t=n-1$ then 
 \begin{align}
   q(\xi,t)\df \sum_{i=1}^{t-1} \frac{1}{n-\xi-i-1} + 1 = H_{n-\xi-2} + 1. \label{eq:defQ-border}
\end{align}
\end{Definition}

It is easy to observe the following fact about the function $q$:
\begin{observation}\label{obs:q}
  Let $a\ge 1$, $1\le t + a < n$, and $0\le\xi-a<n-1-t-a$.
  Then 
  \[
    q(\xi,t)\le q(\xi-a,t+a).
    \]
\end{observation}

Now, we state the crucial lemma that is the starting point of the upper bound computation.

\begin{lemma}\label{lem:VertexContribution}
Given $\PF_k=t$ and $O_k=\xi$, the average number of faces completed at vertex $v_k$ is at most $q(\xi,t)$.  In other words, $\E[F_k \mid \PF_k=t, O_k=\xi] \leq q(\xi,t)$.
\end{lemma}

Note that $O_k+\PF_k$ is never larger than $n-1$ and therefore the value $q(\xi,t)$ is well-defined.
Observe that $O_k+\PF_k=n-1$ if and only if $k=1$ as there are exactly $n-k$ edges between $v_k$ and $V^\uparrow$.

\begin{proof}[Proof of \cref{lem:VertexContribution}]
  Recall that first, we determine which unpaired darts of $\Vup$ lead to $v_k$ in \hyperref[rpb:rc1]{Random choice 1}.
  This corresponds to determining $L$ for $n-k$ darts incident with $v_k$.
  Then we create an auxiliary order $\sigma_k$ of darts in $D_k$, and process the edges according to $\sigma_k$.
When we process dart $d$, we determine what will be $R_k(d)$.
As mentioned above, in order to do that, we will be constructing the cyclic permutation $R_k: D_k\to D_k$ step-by-step.
We start with $R_k$ being undefined.
We define a \emph{forefather} of a dart $d\in D_k$ which is the furthest possible predecesor of $d$ in partially constructed $R_k$.
If no predecessor of $d$ exists, then $d$ is its own forefather.

We label the darts of~$D_k$ in order $\sigma_k$ as $d_1$, \dots, $d_{n-1}$. 
Now, suppose we are about to process $d_i$ where $i\neq n-1$.
We pick uniformly at random the next dart in rotation $R_{k}$, i.e., we choose $R_k(d_i)$.
We are allowed to use any dart which does not have a predecessor (this rules out $i-1$ choices) as well as the forefather of $d_i$ is disallowed (as such a choice would close the cycle $R_k$ prematurely).
Observe that as there are $n-1$ darts around $v_k$, for the $i$-th dart we have $n-1-(i-1)-1$ valid choices.
In case $i=n-1$, we do not have any choice and $R_k(d_{n-1})$ must be equal to the forefather of~$d_{n-1}$. 
Observe that this process produces a uniformly random embedding.

We continue by calculating the probability that a face is formed by fixing some $R_k(d_i)$ for $i < n-1$.
If $d_i$ is of the first category, choosing its successor never completes a new face as, so far, we only determined $R_k$ for 1-open darts.
If $d_i$ is of the second category, we argue we can complete at most one face by determining $R_k(d_i)$. 
We follow $R\circ L$, and it leaves only one choice for the successor, which completes the face.
Therefore, for each $d_i$ of the second category the probability that we complete a face is at most $\frac{1}{n-i-1}$.
Here, $i$ goes from $\xi+1$ to $\xi + t$, where $\xi$ is the number of darts of the first category at $v_k$ and $t$ is the number of darts of the second category.
It is easy to see that for any $d_i$ in the last category, there is no choice $R_k(d_i)$ which completes a face.
Therefore, if $k>1$, then the third category is not empty and $R_k(d_{n-1})$ never completes a face.
We conclude that we arrive at equation~(\ref{eq:defQ}).
If $k=1$ fixing $R_1(d_{n-1})$ might complete an additional face and this accounts for the additional $+1$ in equation~(\ref{eq:defQ-border}).
\end{proof}

We define one more random variable.
Let $T_{n-k}$ represent the number of temporary faces in $G[\Vup]$ in step $k$ (before vertex $v_k$ is added).
Note that $E[T_n]$ is, in other words, an average number of faces of $K_n$.
Hence, the following lemma is the first step in the proof of the main theorem.
The rest of the proof will provide an involved analysis of the right-hand side of Inequality~(\ref{eq:target}).

\begin{lemma}\label{lem:estimate} 
  Let $n\ge 3$ and let $F, \PF_k, O_k$ be random variables as defined earlier.  Then we have
\begin{align}
  \E[F]= \E[T_n] \le \sum_{k=1}^{n-2}  \E[q(O_k,\PF_k)] = \sum_{k=1}^{n-2} \sum_{i=1}^{n-k} \sum_{j=0}^{n-k-i} q(j,i) \cdot \Pr[O_k=j \wedge \PF_k=i]. \label{eq:target}
\end{align}
\end{lemma}

\begin{proof}
The equalities in (\ref{eq:target}) are clear, so we will only argue about the inequality.
We execute \hyperref[rpb]{Random process B} as defined above. 
For the first two vertices $v_n$ and $v_{n-1}$ in the order, all choices are isomorphic. %
We process each other vertex as described in part~\ref{step:vertex2} of the process description.
Hence, the contribution of a single vertex is upper-bounded by \cref{lem:VertexContribution}
\end{proof}

\newcommand{\Oconst}{\nu}
\newcommand{\oconst}{\overline\nu}
\let\nubar\oconst %
Let $1/2 < \nu < 1$ be a constant and $\nubar \df 1 - \nu$.  We will fix this value later on for different ranges of $n$ in order to optimise our bound.  We split the above triple sum (\cref{eq:target} in \cref{lem:estimate}) into several parts: 
\begin{itemize}
  \item $S_1$ will contain the terms where $k=1$.
  \item $S_2$ will contain the terms where $j < \nubar n$ and $i < \frac{n-k}{k}$.
  \item $S_3$ will contain the terms where $j < \nubar n$ and $i \ge \frac{n-k}{k}$.
  \item $S_4$ will contain the terms where $j \ge \nubar n$.
\end{itemize}

Recall that we use $\gamma$ to denote the Euler-Mascheroni constant, as defined in \cref{thm:Hnconv}.
We now define $S_1$, $S_2$, $S_3$, and $S_4$.
We will also state the bounds which we derive for each portion of the sum in the forthcoming subsections.
\begin{align}
  S_1 &\df \sum_{i=1}^{n-1} \sum_{j=0}^{n-1-i} q(j,i) \cdot \Pr[O_1=j \wedge \PF_1=i] \le H_{n-2} + 1 \le \ln(n) + \gamma + 1. \label{eq:fixed}
\end{align}
For the rest, we first take the terms for which $O_k < \oconst{}n$.
Let $b= b(n,k,i) \df\min(n-k-i,\lceil\oconst{}n\rceil-1)$.
When writing down the terms for $S_2$, we used the fact that these terms do not occur if $\tfrac{n-k}{k} \le 1$. Thus we have the summation range for $k$ only between $2$ and $n/2$.
\begin{align}
  S_2 \df& \sum_{k=2}^{n/2} \sum_{i=1}^{\lceil\frac{n-k}{k}\rceil-1}  \sum_{j=0}^{b} q(j,i) \cdot \Pr[O_k=j \wedge \PF_k=i] \label{eq:case1-better}\\
  \le&\  \frac{1}{\Oconst{}}\ln(n) +\ln{\left(\frac{\Oconst{}n-3/2}{\Oconst{}n-1/2-\frac{n}{2}}\right)}  +\frac{1}{\Oconst}\left( \ln(\Oconst/2) - \ln(5\Oconst/2-1) \right)  \, . \nonumber\\
  S_3 \df& \sum_{k=2}^{n-2} \sum_{i=\lceil\frac{n-k}{k}\rceil}^{n-k} \sum_{j=0}^{b} q(j,i) \cdot \Pr[O_k=j \wedge \PF_k=i] \label{eq:defS3}\\
    &\le
\ln(2 \Oconst n)   \frac{\frac{\pi^2}{6}-1}{\Oconst^2}  \left(1 + \frac{4}{\Oconst n - 2}\right)  +
 1.67\ln n +5
 +  \frac{2n}{\Oconst n - 5/2}. \nonumber
\end{align}
In case $n\ge e^{e^{16}}$ and $\Oconst\ge\tfrac{999}{1000}$, we have a stronger estimate:
\begin{align}
  S_3\le 1.6474 \ln n-9. \label{eq:defS3:ass}
\end{align}

Finally, we take the remaining case where $O_k\ge \oconst{}n$. 
The corresponding inequality involves an auxiliary (real) parameter $\mu\in[1,3]$, and an integer $\aleph_m\in\mathbb{Z}$ such that
  $\E[F(m)]\le 5\ln(m)+\aleph_m$ for all $2\le m<n$.
  We denote $\aleph^a_b\df \max_{b<i<a} \aleph_i$ for $0<b<a$.
\begin{align}
  S_4 \df& \ \sum_{k=2}^{n-2} \sum_{i=1}^{n-k} \sum_{j=\lceil\oconst n\rceil}^{n-k-i} q(j,i) \cdot \Pr[O_k=j \wedge \PF_k=i] \label{eq:defS4}\\
    <  & ~~ 
    \Oconst n \ln(\Oconst n) e^{\frac{-n\oconst^2}{2}}   
     +\frac{\nu\ln(\Oconst n)\left(5\ln n+\aleph_{\lceil\oconst n\rceil}^{n-\left\lceil \frac{2}{\oconst}ln^\mu(n)\right\rceil}\right)}{\ln^{\mu}(n)} \, + \nonumber\\
   & ~~ \frac{2\ln^\mu(n) \ln(\Oconst n)}{\oconst^2 n} \cdot \left(5\ln n+\aleph_{n-\left\lceil \frac{2}{\oconst}ln^\mu(n)\right\rceil +1}^{n-2}\right).
  \label{eq:caseOk}
\end{align}

\cref{lem:estimate} together with the above analysis reformulates \cref{thm:logBound} as the following inductive theorem.
The base case of the induction is computed using the computer analysis formulated as \cref{prop:smallValues}.
Note that it is sufficient to assume $n\ge 243$ for the next theorem as the smaller values follow from \cref{thm:logsqBound} via computer-evaluation which is described later in \sv{\cite[Section 5]{ourArxiv}}\lv{\cref{sec:smallValues}}.
Observe that if we do not aim for the best multiplicative constant we can use our $\ln^2 n$ upper bound (\Cref{thm:logsqBound}) in the place of the inductive argument. 
However, it would not be sufficient to use there, for example, the previously
known linear bound.

\begin{theorem} \label{thm:Completegraphslargevalues}
  Let $n\ge 243$ be an integer.
  For $3 \leq m < n$, suppose that $\E[F(m)] \leq 5\ln(m) + \aleph_m$.
  Then we have
    \[
      \E[F(n)] \le S_1+S_2+S_3+S_4 
      \]
      where $S_1,S_2,S_3,S_4$ are defined above in Equations (\ref{eq:fixed}), (\ref{eq:case1-better}), (\ref{eq:defS3}), and (\ref{eq:defS4}).
\end{theorem}

\paragraph{Organization of the remainder of the section.}
First, we carefully compute our estimates and therefore we prove our main result.
It remains to prove the bounds (\ref{eq:fixed})--(\ref{eq:caseOk}) on $S_1$, $S_2$, $S_3$, and $S_4$. 
In order to do that, we show estimates on first and second moment of random variable $\PF_k$.
We follow by Subsections~\ref{sub:S1}, \ref{sub:S2}, \ref{sub:S3}, and \ref{sub:S4}, where the bounds (\ref{eq:fixed})--(\ref{eq:caseOk}) are proven.
Using the formulation of \cref{thm:Completegraphslargevalues} and the estimates in Subsections~\ref{sub:S1}, \ref{sub:S2}, \ref{sub:S3}, and \ref{sub:S4}, we conclude the proof of \cref{thm:logBound} by analysis on different values of $n$.
  We postpone the detailed case analysis to \cref{sec:A}.

  \toappendix{
\begin{proof}[Case analysis of values in \cref{thm:logBound}]
  Using the formulation of \cref{thm:Completegraphslargevalues} and the estimates in Subsections~\ref{sub:S1}, \ref{sub:S2}, \ref{sub:S3}, and \ref{sub:S4}, we conclude by analysis on different values of $n$.
  We adjust the parameters based on $n$.
\begin{itemize}
  \item[Case 1]
  For $\ComputThr< n< e^{20}$ we set $\oconst\df \frac{5}{13}$ ($\Oconst=\frac{8}{13}$), and  $\mu\df 1.5$.
\item[Case 2]
  For $ e^{20}\le n \le e^{e^{16}}$ we set $\oconst\df \frac{1}{25}$ ($\Oconst=\frac{24}{25}$), and $\mu\df 2.1$.
\item[Case 3]  For $n> e^{e^{16}}$ we set $\oconst\df \frac{1}{1000}$ ($\Oconst=\frac{999}{1000}$), and $\mu\df 3$.
\end{itemize}
The initial estimate follows by \cref{thm:Completegraphslargevalues} together with estimates proved in \cref{sub:S1} (\cref{eq:fixed}), \ref{sub:S2} (\cref{eq:case1-better}), \cref{sub:S3} (Equations (\ref{eq:S3:base}) and (\ref{eq:S3:assymptotic})), and \cref{sub:S4} (\cref{eq:caseOk}). 

Suppose that $n>\ComputThr$ and $\Oconst\ge\tfrac{8}{13}$. Then we have
\begin{enumerate}
  \item $\ln{\left(\frac{\Oconst{}n-3/2}{\Oconst{}n-1/2-\frac{n}{2}}\right)}< 1.675$ %
  \item If $\Oconst=\tfrac{8}{13}$, $\Oconst=\tfrac{24}{25}$, or $\Oconst=\tfrac{999}{1000}$ then $\frac{1}{\Oconst}\left( \ln(\Oconst/2) - \ln(5\Oconst/2-1) \right)<-0.9$. %
  \item $\frac{2n}{\Oconst n - 5/2}<3.251$ %
  \item $\ln(2 \Oconst)   \frac{\frac{\pi^2}{6}-1}{\Oconst^2}  \left(1 + \frac{4}{\Oconst n - 2}\right)<1.18$ %
  \item In either case listed above, $\Oconst n \ln(\Oconst n) e^{\frac{-n\oconst^2}{2}}<0.001$ %
  \item 
    $\frac{\nu\ln(\Oconst n)\left(5\ln n+\aleph_{\lceil\oconst n\rceil}^{n-\left\lceil \frac{2}{\oconst}ln^\mu(n)\right\rceil}\right)}{\ln^{\mu}(n)} <
    \begin{cases}
   19.51 & \text{for $\ComputThr\le n\le e^{11}$.} \\ %
  19.774 & \text{for $ e^{11}< n \le  e^{12} $.} \\ %
  19.569 & \text{for $ e^{12}< n \le  e^{13}$.} \\ %
  20.95 & \text{for $e^{13}< n \le  e^{20}$.} \\  %
  5.51 & \text{for $e^{20}< n \le  e^{e^{16}}$.} \\ %
  0.001 & \text{for $e^{e^{16}}<n$.}  \\ %
    \end{cases}$
\item
   $\frac{2\ln^\mu(n) \ln(\Oconst n)\left(5\ln n+\aleph_{n-\left\lceil \frac{2}{\oconst}ln^\mu(n)\right\rceil + 1}^{n-2}\right)}{\oconst^2 n} <
   \begin{cases}
  12.519 & \text{for $\ComputThr\le n\le e^{11}$.} \\  %
  9.53 & \text{for $ e^{11}< n \le  e^{12} $.} \\  %
 4.5 & \text{for $ e^{12}< n \le  e^{13} $.} \\  %
  2.05 & \text{for $ e^{13}< n \le  e^{20} $.} \\  %
  4.4 & \text{for $ e^{20}< n \le  e^{e^{16}} $.} \\  %
  0.001 & \text{for $  e^{e^{16}}<n $.} \\  %
  \end{cases}$
\item Finally, we get:\\
 $\E[F(n)] <
   \begin{cases}
   6 \ln n + 44 \le 5 \ln n + 55,& \text{for $\ComputThr\le n\le e^{11}$,} \\ 
   6 \ln n + 41.1 \le 5 \ln n + 53.1,& \text{for $e^{11}\le n\le e^{12}$,} \\ 
   6 \ln n + 36 \le 5 \ln n + 49,& \text{for $e^{12}\le n\le e^{13}$,} \\ 
   6 \ln n + 35 \le 5 \ln n + 55,& \text{for $e^{13}\le n\le e^{20}$,} \\ 
   4.42 \ln n + 22 \le 5 \ln n + 11,& \text{for $e^{20}\le n\le e^{e^{16}}$,} \\ 
   3.6484 \ln n -2 \le 3.65 \ln n,& \text{for $n\ge e^{e^{16}}$.} \\ 
   \end{cases} $\hfill\qedhere
\end{enumerate}
\end{proof}
}

\lv{
Before we dive into case analysis of the upper bound of Inequality~(\ref{eq:target}), we show an important lemma that bounds the
first two moments of the random variable $\PF_k$. This will be used in the final computations.  Recall that $T_{n-k}$ represents the expected number of temporary faces in the random embedding of $\Vup$.

\begin{lemma}\label{lem:power_of_expectation}
Let $n\ge3$ and $k\le n-2$ be natural numbers. Then
\[
  \E[\PF_k]\le \frac{n-k}{k}
\]
and
\[
    \E[\PF_k^2]\le \frac{(n-k) \left( n+2-\tfrac{3}{k}\right) + 2\E(T_{n-k})}{k^2}.
\]
\end{lemma}

\begin{proof}
There are precisely $k(n-k)$ candidate walks $W_1,W_2,\dots, W_{k(n-k)}$. 
We decompose $\PF_k$ into a sum of $k(n-k)$ indicator random variables $X_i$, where each $X_i$ corresponds to the candidate walk $W_i$ in $V^\uparrow$:
\[ \PF_k =\sum_{i=1}^{k(n-k)} X_i.\]
More precisely, $X_i=1$ if $W_i$ is a potential face, and $0$, otherwise. 
To determine that, we choose, for each vertex $v_t \in \Vup$, one of its unpaired darts and pair it with one of the darts incident with $v_k$.
Each possible dart at $v_t$ is selected uniformly at random (\hyperref[rpb:rc1]{Random choice 1})  with probability $\frac{1}{k}$.
This corresponds to Step \ref{step:3a} in the description of \hyperref[rpb]{Random process B}. 

Now, we use the linearity of expectation to bound $\E[\PF_k]$ and $\E[\PF_k^2]$.
For that we need to determine the values of $\E[X_i]$, $\E[X_i^2]$ and $\E[X_iX_j]$ where $i\neq j$. 

  \begin{claim}\label{c:xi}
    For each $i\in [k(n-k)]$, we have
    \[\E[X_i^2] = \E[X_i] \le \frac{1}{k^2}.\]
  \end{claim}

\noindent\textit{Proof of claim.~}
We just observe that $X_i^2 = X_i$, and that $X_i=1$ if and only if the first and the last darts of the candidate walk $W_i$ are different and both active. If they are different and incident with the same vertex in $\Vup$, then they cannot be both active; otherwise, each of them is active with probability $\frac{1}{k}$. This implies the claim. 
\hfill$\diamondsuit$

\bigskip 
Claim~\ref{c:xi} gives an immediate conclusion about $\E[\PF_k]$.

Two distinct candidate walks are \emph{consecutive} if one originates with the dart that the other leads to.
In other words, last dart of one candidate walk is the first dart of the other candidate walk.

  \begin{claim}\label{c:xixj}
    Let $W_i$ and $W_j$ be candidate walks, where $i\neq j$.
    Then
    \begin{itemize}
      \item    $\Pr[X_i=X_j=1]=0$ if $W_i,W_j$ are the two candidate walks on a 2-open face which is not strongly 2-open face.
      \item    $\Pr[X_i=X_j=1]\le\frac{1}{k^2}$ if $W_i,W_j$ are the two candidate walks on a strongly 2-open face.
      \item    $\Pr[X_i=X_j=1]\le\frac{1}{k^3}$ if $W_i,W_j$ are consecutive candidate walks not on a strongly 2-open face.
      \item    $\Pr[X_i=X_j=1]\le\frac{1}{k^4}$ otherwise.
    \end{itemize}
  \end{claim}

\noindent\textit{Proof of claim.~}
If darts of two candidate walks on a 2-open face which is not strongly 2-open cannot both be active, so $X_iX_j=0$.
Suppose that $W_i$ and $W_j$ are the two candidate walks on a strongly 2-open face $f$ and let $d_1,d_2$ be the corresponding downward darts. As $f$ is strongly 2-open $d_1$ and $d_2$ cannot be incident with the same vertex of $\Vup$.  Hence, each of them is active with probability $\tfrac{1}{k}$, so $\Pr[X_i=X_j=1]=\frac{1}{k^2}$.

Suppose now that $W_i$ and $W_j$ are consecutive candidate walks not on a 2-open face.
Then $X_i=X_j=1$ if and only if all three corresponding darts are active. Since each is active with probability $\tfrac{1}{k}$, we conclude that $\Pr[X_i=X_j=1]\le\frac{1}{k^3}$.

In the remaining possibility, $W_i$ and $W_j$ are distinct candidate walks that are not consecutive.
If they together involve fewer than 4 downward darts and are not in the cases treated above, then one of them (say $W_i$) involves just one downward dart, in which case $X_i=0$ and the considered probability is 0. 
Otherwise, they involve four distinct downward darts, each of which is active with probability $\tfrac{1}{k}$. This implies that $\Pr[X_i=X_j=1]\le\frac{1}{k^4}$.
\hfill$\diamondsuit$

\bigskip

Since the bounds in the claim are dependent on whether the walks are in strongly 2-open faces or not, we continue by estimating the number of strongly 2-open faces at a given step.  Let $L_k$ denote the number of strongly 2-open faces at the start of step $k$. Let us first consider an upper bound of the expectation of $\PF_k^2$ conditional on $L_k=\ell$.

Suppose that there are $\ell$ strongly 2-open faces. There are $k(n-k)$ candidate walks and there are at most $k(n-k)$ pairs of consecutive candidate walks since each pair has a unique downward
dart in common. 
Therefore, there are at most $k(n-k)-2\ell$ consecutive candidate walks that are not in strongly 2-open faces.
Putting these facts together in combination with Claims~\ref{c:xi} and \ref{c:xixj} gives the following:

 \begin{align*}
     &2\sum_{i=1}^{k(n-k)}\sum_{j=i+1}^{k(n-k)}Pr[X_iX_j = 1\mid L_k = \ell]\\
     &\leq 2\ell\,\frac{1}{k^2} + 2 (k(n-k)-2\ell)\,\frac{1}{k^3} + \bigl( k^2(n-k)^2 - k(n-k) - 2\ell-2(k(n-k)-2\ell) \bigr)\,\frac{1}{k^4} \\
     &= \frac{2\ell}{k^2} + \frac{2(n-k)}{k^2} - \frac{4\ell}{k^3} +  \frac{(n-k)^2}{k^2} + \frac{2\ell}{k^4} - \frac{3(n-k)}{k^3} \\
     &\le \frac{2\ell}{k^2} + \frac{2(n-k)}{k^2} + \frac{(n-k)^2}{k^2} - \frac{3(n-k)}{k^3}.
 \end{align*}
 
To conclude the proof, we will use linearity of expectation. %
We will also use $T_{n-k}$ as an upper bound on the number of strongly 2-open faces.
\begin{align}
  \E[\PF_k^2] &=
  \sum_{i=1}^{k(n-k)}\sum_{j=1}^{k(n-k)}\E[X_iX_j]= \sum_{i=1}^{k(n-k)}\E[X_i^2] + 2 \sum_{i=1}^{k(n-k)}\sum_{j=i+1}^{k(n-k)}\E[X_iX_j] \nonumber\\
  & = \sum_{i=1}^{k(n-k)}\E[X_i^2] + 2 \sum_{\ell} \sum_{i=1}^{k(n-k)}\sum_{j=i+1}^{k(n-k)}Pr[X_iX_j=1 \mid L_k = \ell] Pr[L_k = \ell] \nonumber\\
  &\le k(n-k)\,\frac{1}{k^2} + \sum_{\ell} \left( \frac{2\ell}{k^2} + \frac{2(n-k)}{k^2} + \frac{(n-k)^2}{k^2} - \frac{3(n-k)}{k^3} \right) Pr[L_k = \ell]  \nonumber\\
  &= \frac{(n+2-\tfrac{3}{k})(n-k) + 2\E[L_k]}{k^2} \nonumber\\
  &\le \frac{(n+2-\tfrac{3}{k})(n-k) + 2\E[T_{n-k}]}{k^2}. \label{eq:Eksq}
\end{align}
\end{proof}

The proof of \cref{thm:logBound} follows by estimates on parts $S_1$, $S_2$, $S_3$, and $S_4$ which are given in the following subsections.

\subsection{Estimate on $S_1$ (\cref{eq:fixed})}\label{sub:S1}

We estimate the worst-case scenario for function $q$ in the case when $k=1$; see (\cref{eq:defQ-border}) of \cref{def:q}.
Note that for the case $k=1$, $O_k+\PF_k=n-1$.
\begin{align}
  S_1 = \sum_{i=1}^{n-1} \sum_{j=0}^{n-1-i} q(j,i) \cdot \Pr[O_1=j \wedge \PF_1=i] \le q(0,n-1)&\le H_{n-2}+1 \label{eq:S1} \\ 
  &< \ln(n)+1+\gamma. \nonumber
\end{align}
The last inequality follows from \cref{thm:Hnconv} (assuming $n\ge3$).

\subsection{Estimate on $S_2$ (\cref{eq:case1-better})}\label{sub:S2}
\newcommand\nfl{\lfloor n/2 \rfloor}

First, we show a lemma we will be using in our estimates.
\begin{lemma}\label{lem:harmonic}
Let $n\geq 3$, $t\geq 1$, and $\xi$ be integers such that $t+\xi\le n-2$. %
Then
\begin{equation}
  q(\xi,t)= H_{n-\xi -2}-H_{n-\xi-t-2} 
    \le \ln{\left(\frac{n-3/2-\xi}{n-3/2-\xi-t}\right)}
    < \frac{t}{n-3/2-\xi-t} \, .  \label{eq:Lemma2.10}
\end{equation}
\end{lemma}

\begin{proof}
  As $t+\xi<n-2$ we use definition of function $q$ in Equation~(\ref{eq:defQ}).
  Note that the same together with $t\ge 1$ implies that $H_{n-\xi-2}>H_0$ and $H_{n-\xi-t-2}\ge H_0$.
  Hence, \cref{thm:Hnconv} yields the following estimate:
\begin{align*}
    H_{n-\xi -2}-H_{n-\xi-t-2} ~\le~
    & \ln{\left(n-3/2-\xi\right)}+\gamma+\frac{1}{24(n-\xi-2)^2} \\ 
    & -\ln{\left(n-3/2-\xi-t\right)}-\gamma -\frac{1}{24(n-\xi-t-2+1)^2} \\
    ~\le~ & \ln{\left(\frac{n-3/2-\xi}{n-3/2-\xi-t}\right)} %
    ~=~  \ln{\left(1+\frac{t}{n-3/2-\xi-t}\right)} \\
    ~<~ & \frac{t}{n-3/2-\xi-t}.
\end{align*}
In the second inequality we used the fact that $t\ge 1$ and in the last one that $\ln(1+x)\le x$.
\end{proof}

Now, we continue by showing that Inequality~(\ref{eq:case1-better}) holds.
As \hyperref[def:q]{$q(j,i)$} is an increasing function in both $i$ and $j$, we can upper-bound it by the value for the largest $i$ and largest $j$.
Therefore, we can factor it out of the sum and upper-bound the disjoint probabilities by $1$.
Recall that $b=\min(n-k-i,\lceil\oconst{}n\rceil-1)$. 
The first inequality follows by \cref{obs:q}.

\begingroup
\allowdisplaybreaks
\begin{align}
  S_2 &= \sum_{k=2}^{\nfl} \sum_{i=1}^{\lceil\frac{n-k}{k}\rceil-1} \sum_{j=0}^{b} q(j,i) \cdot \Pr[O_k=j \wedge \PF_k=i] \nonumber\\
      &\le \sum_{k=2}^{\nfl} q\left(\min\left(n-k-\lfloor\tfrac{n-k}{k}\rfloor,\lceil\oconst{}n\rceil-1\right),\lfloor\tfrac{n-k}{k}\rfloor\right)\label{eq:S2}\\ 
  &\le \sum_{k=2}^{\nfl} q\left(\lceil\oconst{}n\rceil-1,\lfloor\tfrac{n-k}{k}\rfloor\right)\nonumber\\
  &\le \sum_{k=2}^{\nfl} q\left(\oconst{}n,\tfrac{n-k}{k}\right)\nonumber\\
  &\le \sum_{k=2}^{\nfl} \ln\left(\frac{\Oconst{}n-3/2}{\Oconst{}n-1/2-\tfrac{n}{k}}\right).\nonumber
\end{align}
\endgroup
Recall $\oconst\le 5/11$, $k\ge 2$, and $n\ge 22$.
For the last inequality we used \cref{lem:harmonic} as \[\oconst{}n+\tfrac{n-k}{k}\le \tfrac{5}{11}n+\frac{n-2}{2}\le n-2.\]

Note that 
\[
   \frac{\Oconst{}n-3/2}{\Oconst{}n-1/2-\frac{n}{k}} \le \frac{\Oconst}{\Oconst{}-1/k}
 \]
 when $k\ge 3$ as $\Oconst>\tfrac{1}{2}$.
 Letting $a \df \ln{\left(\frac{\Oconst{}n-3/2}{\Oconst{}n-1/2-\frac{n}{2}}\right)}$, we have
\begin{align*}
    S_2 & \le 
    \sum_{k=2}^{n/2} \ln{\left(\frac{\Oconst{}n-3/2}{\Oconst{}n-1/2-\frac{n}{k}}\right)} \\
     & \le
  \ln{\left(\frac{\Oconst{}n-3/2}{\Oconst{}n-1/2-\frac{n}{2}}\right)} +
  \sum_{k=3}^{n/2} \ln \left(\frac{\Oconst}{\Oconst{}-1/k}\right) \\
    &\le 
 a +
    \sum_{k=3}^{n/2} \frac{1}{\Oconst{} k - 1} =
    a+ \sum_{k=3}^{n/2} \frac{1}{\Oconst{} k - 1}\\
    & \le 
    a+
    \int_{5/2}^{n/2+1/2} \frac{1}{\Oconst{}x-1}dx \\
    & \le a+ \frac{1}{\Oconst{}}
    \int_{5\Oconst/2-1}^{\Oconst{}n/2} \frac{1}{z}\,dz \\
    & = a+  \frac{1}{\Oconst{}} (\ln(\Oconst{}n/2) - \ln(5\Oconst/2-1)) \\
    & = \frac{1}{\Oconst{}}\ln(n) +\ln{\left(\frac{\Oconst{}n-3/2}{\Oconst{}n-1/2-\frac{n}{2}}\right)}  +\frac{1}{\Oconst}\left( \ln(\Oconst/2) - \ln(5\Oconst/2-1)\right)
      .
\end{align*}

\subsection{Estimate on $S_3$ (\cref{eq:defS3} and \cref{eq:defS3:ass})}\label{sub:S3}

In what follows, we will use an auxiliary function $\widehat{q}$ with only one parameter $1\le t \le n-k$ ($\le n-2$) which will be a worst-case upper-bound on the two-parameter function \hyperref[def:q]{$q$}:
\begin{equation}
  \widehat{q}(t)\df
\begin{cases}
\ln{\left(\frac{\Oconst{}n-3/2}{\Oconst{}n-3/2-t}\right)}, & \text{if}\ t\le \Oconst{}n-2, \\
\ln(2t+1), &  \text{if}\ t\ge \Oconst{}n-2. %
\end{cases}\phantomsection\label{eq:qprime}
\end{equation}

\begin{claim}\label{cl:qprime}
Suppose that $2\le k\le n-2$ and $1\le i \le n-k$.
Let $b=b(n,k,i)=\min(n-k-i,\lceil\oconst{}n\rceil-1)$. 
Then $q(b(n,k,i),i) \le \widehat{q}(i)$.
\end{claim}

\noindent\textit{Proof of claim.~}
If $n-k-i > \lceil\oconst{}n\rceil-1$, then $b=\lceil\oconst{}n\rceil-1$ and $i < n-k-\lceil\oconst{}n\rceil+1 < \lfloor\Oconst{}n\rfloor-1$ (since $k\ge2$). Thus, $i+b\le \lfloor\Oconst{}n\rfloor-2+\lceil\oconst{}n\rceil-1<n-2$ and \cref{lem:harmonic} implies that:
\[
  q(b,i) \le \ln{\left(\frac{n-3/2 - \lceil\oconst{}n\rceil+1}{n-3/2-\lceil\oconst{}n\rceil+1-i}\right)} \le \ln{\left(\frac{\Oconst{}n-3/2}{\Oconst{}n-3/2-i}\right)} = \widehat{q}(i).
\]

Suppose now that $b = n-k-i$ ($\le \lceil\oconst{}n\rceil-1$). Again, $i+b\le n-k\le n-2$ as $k\ge 2$. 
Therefore, \cref{lem:harmonic} applies:
\[
q(b,i) \le \ln{\left(\frac{n-3/2 - b}{n-3/2-b-i}\right)} = \ln{\left(1 + \frac{i}{k-3/2}\right)} \le \ln(2i+1) = \widehat{q}(i).
        \tag*{$\diamondsuit$}
\]

\medskip

Note that in the second case, we use quite a loose upper-bound because we want to keep function $\widehat{q}$ continuous.
It is easy to verify that for $t=\Oconst n-2$ both expressions used in the definition of $\widehat{q}$ coincide, so that $\widehat{q}(\Oconst n-2) = \ln(2\Oconst n-3)$.

\smallskip 
Using the \hyperref[eq:qprime]{function $\widehat{q}$}, we define new function for $1\le t \le n-k$ ($\le n-2$):
\[
  f(t)\df \frac{\widehat{q}(t)}{t^2}.\phantomsection\label{def:f}
\]
First, we show that the function $f$ is convex and few other properties.

  \begin{lemma}\phantomsection\label{lem:convex}
    Let $n\ge 3$ and $2\le k\le n-2$. The \hyperref[def:f]{function $f(t)$} is continuous.
    It is convex on the intervals $[1,\Oconst n-2)$ and $[\Oconst n-2,n-k]$.
    Moreover, if\/ $1\le t < \frac{z_0}{z_0+1}(\Oconst n-3/2)$, where $z_0 \approx 2.51286$ is the non-zero solution of the equation $z+1=e^{z/2}$, or if $t>\Oconst n-2$, then $f(t)$ is decreasing, while for $\frac{z_0}{z_0+1}(\Oconst n-3/2) < t < \Oconst n-2$ it is increasing.
  \end{lemma}

\begin{proof}
For what follows, we observe that the function is continuous as it is continuous on two given intervals and the value at $t = \Oconst n-2$ coincides in both expressions.
It is also clear that $f(t)$ is decreasing for $t > \Oconst n-2$.

Suppose now that $t<\Oconst{}n-2$.
For simplicity, we make a linear substitution: $x\df\frac{t}{\Oconst{}n-3/2}$. Now $f$ is convex if and only the function $g(x) = x^{-2} \ln(\frac{1}{1-x})$ is convex for $x\in [\frac{1}{\Oconst{}n-3/2}, \frac{\Oconst{}n-2}{\Oconst{}n-3/2})$.
The result follows by examining Taylor series of $\ln\left(\frac{1}{1-x}\right)=\sum_{j=1}^{\infty} \frac{x^j}{j}$.
Now, the result follows since the (infinite) sum of convex functions is convex.

To see where $f(t)$ is decreasing or increasing, we just need to see where its first derivative is negative. This is a routine task and is left to the reader.

For $t\ge\Oconst{}n-2$, we examine the second derivative $f''(t)$ and prove that it is positive. Again, this is a routine task and is left to the reader.
\end{proof}

\begingroup
\allowdisplaybreaks
Recall that $b=\min(n-k-i,\lceil\oconst{}n\rceil-1)$.
We start with Equation (\ref{eq:defS3}):
\begin{eqnarray*} 
  S_3 &=& \sum_{k=2}^{n-2} \sum_{i=\lceil\frac{n-k}{k}\rceil}^{n-k} \sum_{j=0}^{b} q(j,i) \cdot \Pr[O_k=j \wedge \PF_k=i] \\ %
&\le& \sum_{k=2}^{n-2} \sum_{i=\lceil\frac{n-k}{k}\rceil}^{n-k} q(b,i) \cdot \Pr[O_k\le b \wedge \PF_k=i] \\ 
&\le& \sum_{k=2}^{n-2} \sum_{i=\lceil\frac{n-k}{k}\rceil}^{n-k} q(b,i) \cdot \Pr[\PF_k=i] \\ 
&\le& \sum_{k=2}^{n-2} \sum_{i=\lceil\frac{n-k}{k}\rceil}^{n-k} \widehat{q}(i) \cdot \Pr[\PF_k=i]. \end{eqnarray*}
The last inequality uses the \hyperref[eq:qprime]{function $\widehat{q}$} defined in (\ref{eq:qprime}) that upper-bounds $q(b,i)$ by \cref{cl:qprime}.
We transform it to an equivalent formulation:
\begin{equation} 
 S_3 \le \sum_{k=2}^{n-2} \sum_{i=\lceil\frac{n-k}{k}\rceil}^{n-k} f(i) \cdot i^2 \Pr[\PF_k=i].
 \label{eq:S3 double sum}
\end{equation}
\endgroup

By \cref{lem:convex}, the function $f$ is convex on the interval $[1,\Oconst n-2)$ and is decreasing on the interval $[\Oconst n - 2,n-k]$. This implies that
\begin{equation}
    M_k := \max \{ f(i) \mid \left\lceil\tfrac{n-k}{k}\right\rceil \le i \le n-k\}
    \le \max \left\{ f\left(\left\lceil\tfrac{n-k}{k}\right\rceil\right), f(\Oconst n-2) \right\}.
  \phantomsection  \label{eq:f(i) bounded by maximum}
\end{equation}

\begin{lemma}\label{lem:M_k bounds}
  Let $k\in\mathbb{N}$, $2\le k < \lceil \tfrac{n}{2}\rceil$, $n\ge 243$, and $\tfrac{8}{13}\le \Oconst<1$. 
\\  
If $k\le \frac{\ln(2\Oconst n-3)}{17.65 \Oconst^2}$, then
\begin{align}
     M_k \le 
     \ln(2\Oconst n) \cdot \Oconst^{-2} n^{-2} (1+\tfrac{4}{\Oconst n-4}).\label{est:1}
   \end{align}
  If $k\ge \ln(2\Oconst n) \Oconst^{-1} (1+\tfrac{4}{\Oconst n-4})$, then
  \begin{align}
     M_k \le  
     \frac{k}{n(n-k)} \cdot \frac{1}{\Oconst-1/k - 1/(2n)}.\label{est:2}
    \end{align}
   If $k\ge \lceil \tfrac{n}{2}\rceil$, then \[M_k=f(1)<\frac{1}{\Oconst n -5/2}\]
   If  $\frac{\ln(2\Oconst n-3)}{17.65 \Oconst^2} < k < \min\left(\lceil \tfrac{n}{2}\rceil, \ln(2\Oconst n) \Oconst^{-1} (1+\tfrac{4}{\Oconst n-4})\right)$ $M_k$ is at most sum of Equations~(\ref{est:1}) and (\ref{est:2}).
\end{lemma}

\begin{proof} %
  Consider first $k\ge \lceil \tfrac{n}{2}\rceil$ (as $n\ge 243$ and $\Oconst\ge \tfrac{8}{13}$ we have the first inequality) then 
  \[
  \frac{\ln(2\Oconst n-3)}{(\Oconst n -2)^2} 
  <  \frac{1}{\Oconst n- 3/2}
  = \frac{\tfrac{1}{\Oconst n -5/2}}{1+\tfrac{1}{\Oconst n -5/2}} 
  \le \ln\left(1+\frac{1}{\Oconst n -5/2}\right)=f(1)\le\frac{1}{\Oconst n -5/2}.
\]

Therefore, $k< \lceil \tfrac{n}{2}\rceil$.
We first precompute some estimates on the function \hyperref[def:f]{$f$}.

 \begin{claim} \label{claim:Mk:1}
  If $n\ge 8$, then
   \[ f(\Oconst n-2) \le \ln(2\Oconst n) \cdot \Oconst^{-2} n^{-2} (1+\tfrac{4}{\Oconst n-4}). \]
\end{claim}
\noindent\textit{Proof of claim.~}
\begin{align*}
  f(\Oconst n-2)& = \frac{\ln(2\Oconst n -3)}{(\Oconst n-2)^2}
\le \frac{\ln(2\Oconst n )}{\Oconst^2 n^2} \frac{\Oconst^2 n^2}{(\Oconst n-2)^2}%
= \frac{\ln(2\Oconst n )}{\Oconst^2 n^2} \left(1+\frac{4\Oconst n -4}{\Oconst^2 n^2-4\Oconst n +4} \right)\\
                  &\le \frac{\ln(2\Oconst n )}{\Oconst^2 n^2} \left(1+\frac{4 }{\Oconst n-4 } \right).
        \tag*{$\diamondsuit$}
      \end{align*}

 \medskip

 \begin{claim} \label{claim:Mk:2}
   Suppose that $2\le k < \lceil n/2\rceil$ and $\Oconst\ge\tfrac{24}{25}$.
  \[
    f \left( \tfrac{n-k}{k} \right) \le 
      \frac{k}{n(n-k)} \cdot \frac{1}{\Oconst-1/k - 1/(2n)}.
   \]
\end{claim}

\noindent\textit{Proof of claim.~}
  \begin{align*}
                                                                &  f\left( \tfrac{n-k}{k} \right) =
      \frac{k^2}{(n-k)^2} \cdot \ln\left(1 + \tfrac{(n-k)/k}{\Oconst n-1/2-n/k} \right) \\
      & \le
      \frac{k^2}{(n-k)^2} \cdot \frac{(n-k)/k}{\Oconst n-1/2-n/k} =
      \frac{k}{n(n-k)} \cdot \frac{1}{\Oconst-1/k - 1/(2n)}.
      \tag*{{$\diamondsuit$}}
      \end{align*}
 
 \medskip 

 Now we will consider two situations:

 \smallskip

 \noindent\textbf{Case 1.~~} $2\le k\le \frac{\ln(2\Oconst n-3)}{17.65 \Oconst^2}$.
 Then we will verify that $f\left(\tfrac{n-k}{k}\right) \le f(\Oconst n -2)$ and by the convexity of $f$ (\cref{lem:convex}) we conclude that \hyperref[eq:f(i) bounded by maximum]{$M_k$} is upper-bounded by \cref{claim:Mk:1}.
We conclude by the following computation where, for the first inequality, we use \cref{claim:Mk:2}.

\begin{align*}
  f \left( \tfrac{n-k}{k} \right) &
 \le     \frac{k}{n(n-k)} \cdot \frac{1}{\Oconst-1/k - 1/(2n)}
 \le     \frac{2k}{n^2 \left( \Oconst-1/k - 1/(2n) \right) }  \\
    & \le  \frac{17.65k}{n^2}  \tag*{(as $k\ge2$ and $\Oconst\ge \tfrac{8}{13}$ and $n\ge 243$)} %
  \\
&\le \frac{\ln(2\Oconst n-3)}{\Oconst^2 n^2} \le  \frac{\ln(2\Oconst n -3)}{(\Oconst n-2)^2}
  = f(\Oconst n-2).  
\end{align*}

 \noindent\textbf{Case 2.~~} $k\ge \ln(2\Oconst n) \Oconst^{-1} (1+\tfrac{4}{\Oconst n-4})$.
 Then we will verify that $f(\Oconst n -2)\le f\left(\tfrac{n-k}{k}\right)$ and by the convexity of $f$ (\cref{lem:convex}) we conclude that \hyperref[eq:f(i) bounded by maximum]{$M_k$} is upper-bounded by \cref{claim:Mk:2}.
We conclude by the following computation where, for the first inequality, we use \cref{claim:Mk:1}.
 \begin{align*}
   f(\Oconst n -2) & \le 
\ln(2\Oconst n) \cdot \Oconst^{-2} n^{-2} (1+\tfrac{4}{\Oconst n-4})
\le  \frac{k}{\Oconst n^2}  %
     \le  \frac{k^2}{(n-k)^2} \cdot \left( \tfrac{(n-k)/k}{(n-k)/k + \Oconst n-1/2-n/k} \right) \\
 &    \le  \frac{k^2}{(n-k)^2} \cdot \ln\left(1 + \tfrac{(n-k)/k}{\Oconst n-1/2-n/k} \right) 
=
f\left(\tfrac{n-k}{k}\right).
 \end{align*}
 \bigskip
 For the range between bounds, $\frac{\ln(2\Oconst n-3)}{17.65 \Oconst^2}  < k < \min(\ln(2\Oconst n) \Oconst^{-1} (1+\tfrac{4}{\Oconst n-4}),n/2)$, we will use the sum of the two bounds of \cref{eq:f(i) bounded by maximum}  (implied by Claims \ref{claim:Mk:1} and \ref{claim:Mk:2}) as an upper bound on the maximum.
  \end{proof}

 \medskip

Below we will use the expectations $E_k^2 := \E(\PF_k^2) = \sum_{i=1}^{n-k} i^2\Pr[\PF_k=i]$ and their upper bound established in \cref{lem:power_of_expectation}.
Now we split the summation in (\ref{eq:S3 double sum}), and use above inequalities from \cref{lem:M_k bounds} to obtain the following:
\begin{align} 
  S_3 &\le \sum_{k=2}^{n-2} M_k \sum_{i=\lceil\frac{n-k}{k}\rceil}^{n-k} i^2 \Pr[\PF_k=i]
  \nonumber \\
  &\le \sum_{k=2}^{n-2} M_k\, E_k^2 \label{eq:small} \\
  & \le 
     \ln(2\Oconst n) \cdot \Oconst^{-2} n^{-2} (1+\tfrac{4}{\Oconst n-4})
  \sum_{k=2}^{\left\lfloor \ln(2\Oconst n) \Oconst^{-1} (1+\tfrac{4}{\Oconst n-4})\right\rfloor} E_k^2 ~ + \nonumber \\
  & ~~~~~
  \sum_{k=\left\lfloor\tfrac{\ln(2\Oconst n-3)}{17.65 \Oconst^2}\right\rfloor}^{\lfloor (n-1)/2 \rfloor} 
      \frac{k}{n(n-k)} \cdot \frac{1}{\Oconst-1/k - 1/(2n)}
  E_k^2 + 
  f(1) \sum_{k=\lceil n/2 \rceil}^{n-2} E_k^2. \label{eq:S3 main inequality}
\end{align}

It remains to estimate the following sums in the above estimate:
\begingroup
\allowdisplaybreaks
\begin{align*} 
    A &\df \sum_{k=2}^{\left\lfloor \ln(2\Oconst n) \Oconst^{-1} \left(1+\tfrac{4}{\Oconst n-4}\right)\right\rfloor} E_k^2,\\
    B &\df \sum_{k=\left\lfloor\tfrac{\ln(2\Oconst n-3)}{17.65 \Oconst^2}\right\rfloor}^{\lfloor (n-1)/2 \rfloor} 
      \frac{k}{n(n-k)} \cdot \frac{1}{\Oconst-1/k - 1/(2n)} \, E_k^2, \quad \textrm{and}\\
    C &\df \sum_{k=\lceil n/2 \rceil}^{n-2} E_k^2.
\end{align*}
\endgroup

Let us first recall from \cref{lem:power_of_expectation} that $k^2 E_k^2 \le (n-k)(n+2-\frac{3}{k})+2\E(T_{n-k})$. Moreover, by \cref{thm:logsqBound} (as $n\ge 22$ and $k\ge 2$), we have that
\[2\E(T_{n-k}) \le 2\E(F(n)) \le H_{n-3} H_{n-2} \le\frac{3n}{2} \le \frac{3}{k}(n-k)+k(n+2) -2n.\] 
Therefore, \begin{align}E_k^2\le \frac{n^2}{k^2}. \label{eq:k2Ek2} \end{align}
Using \cref{eq:k2Ek2}, we get:
\begin{align}
   A = \sum_{k=2}^{\left\lfloor \ln(2\Oconst n) \Oconst^{-1} \left(1+\tfrac{4}{\Oconst n-4}\right)\right\rfloor} E_k^2 \le n^2 \sum_{k\ge2} \frac{1}{k^2} = n^2 \left(\frac{\pi^2}{6}-1\right).\label{eq:A}
 \end{align}
Similarly,
\begin{align}
   C = \sum_{k=\lceil n/2 \rceil}^{n-2} E_k^2 \le n^2\sum_{k=\lceil n/2 \rceil}^{n-2} \frac{1}{k^2} \le
   2n. \label{eq:C}
 \end{align}

 Next, we estimate $B$.
 We now give two separate estimates of $ \frac{1}{\Oconst-1/k - 1/(2n)}$:
 First, suppose that $n\ge 243$, $\Oconst\ge \tfrac{8}{13}$, and $k\ge 69$: 
 \begin{align}
   \frac{1}{\Oconst-1/k - 1/(2n)} \leq \frac{1}{\frac{2}{3}-\frac{1}{69}-\frac{1}{486}}\le 1.67. %
   \label{eq:Bfraction:worse}
 \end{align}
 For an assymptotic estimate, we suppose that $n\ge e^{e^{16}}$ and $\Oconst\ge \tfrac{999}{1000}$, so $k\ge\left\lfloor\tfrac{\ln(2\Oconst n-3)}{17.65 \Oconst^2}\right\rfloor=504470$: %
 \begin{align}
   \frac{1}{\Oconst-1/k - 1/(2n)} < 1.0011. %
   \label{eq:Bfraction:assymptotic}
 \end{align}

  Using $n\ge 243$, $\Oconst\ge\tfrac{8}{13}$, \cref{eq:k2Ek2}, and \cref{eq:Bfraction:worse} we have
\begin{align}
     B &\le  
     \sum_{k=2}^{68}  
\frac{n}{k(n-k)} \cdot \frac{1}{\Oconst-1/k - 1/(2n)} 
+
     \sum_{k=69}^{\lfloor (n-1)/2 \rfloor}  
     \frac{n}{k(n-k)} \cdot \frac{1}{\Oconst-1/k - 1/(2n)} %
         \nonumber\\
      &\le 
      \sum_{k=2}^{68} \frac{1}{k} \frac{1}{\Oconst-1/k - 1/(2n)} +
      \sum_{k=2}^{68} \frac{1}{n-k} \frac{1}{\Oconst-1/k - 1/(2n)} +
      1.67 \sum_{k=68}^{\lfloor (n-1)/2 \rfloor} 
      \left(  \frac{1}{k} + \frac{1}{n-k}\right)\nonumber\\
      &\le 11.42+ %
      0.62+ %
      1.67 \sum_{k=69}^{n-69}
       \frac{1}{k} \nonumber\\
      &= 12.04+  1.67 \left( H_{n-69} - H_{68}\right)   \tag*{(by \cref{thm:Hnconv})}\nonumber\\ 
      &\le 12.04 +  1.67\ln n + 1.67\cdot 0.57722 + \frac{1.67}{24 \cdot 234^2}  - 8.02 \nonumber  \\ %
      &= 1.67\ln n +5.\label{eq:B:worse}
\end{align}

When $n\ge e^{e^{16}}$, $\Oconst\ge\tfrac{999}{1000}$, \cref{eq:k2Ek2}, and \cref{eq:Bfraction:assymptotic} we have
\begin{align}
     B &\le 1.0011 \sum_{k=\left\lfloor\tfrac{\ln(2\Oconst n-3)}{17.65 \Oconst^2}\right\rfloor}^{\lfloor (n-1)/2 \rfloor} 
         \frac{n}{k(n-k)}\nonumber\\
      &= 1.0011 \sum_{k=\left\lfloor\tfrac{\ln(2\Oconst n-3)}{17.65 \Oconst^2}\right\rfloor}^{\lfloor (n-1)/2 \rfloor} 
      \left(  \frac{1}{k} + \frac{1}{n-k}\right)\nonumber\\
      &= 1.0011 \sum_{k=\left\lfloor\tfrac{\ln(2\Oconst n-3)}{17.65 \Oconst^2}\right\rfloor}^{n-\left\lfloor\tfrac{\ln(2\Oconst n-3)}{17.65 \Oconst^2}\right\rfloor} 
       \frac{1}{k} \nonumber\\
      &\le  1.0011 \left( H_{n-1} - H_{\left\lfloor\tfrac{\ln(2\Oconst n-3)}{17.65 \Oconst^2}\right\rfloor-1}\right) \tag*{(by \cref{thm:Hnconv})}\nonumber\\ %
      &\le 1.0011\ln n + 1.0011\cdot 0.57723 - 12.11 \nonumber\\
      &\le 1.0011\ln n -11.5. \label{eq:B:assymptotic}
\end{align}

Combining all the obtained estimates (Equations (\ref{eq:A}), (\ref{eq:C}), and (\ref{eq:B:worse})), we get (when $n\ge 243$ and $\Oconst\ge\tfrac{8}{13}$):
\begin{align}
   S_3 &\le \ln(2\Oconst n) \cdot \Oconst^{-2} n^{-2} (1+\tfrac{4}{\Oconst n-4}) \cdot A +
  B + f(1)\cdot C \nonumber\\
    &\le
 \ln(2 \Oconst n)   \frac{\frac{\pi^2}{6}-1}{\Oconst^2}  \left(1 + \frac{4}{\Oconst n - 2}\right)  +
 1.67\ln n +5+
   \frac{2n}{\Oconst n - 5/2}.
   \label{eq:S3:base} %
\end{align}

For the asymptotic case we combine estimates (Equations (\ref{eq:A}), (\ref{eq:C}), and (\ref{eq:B:assymptotic})) to get the following (assuming $n\ge e^{e^{16}}$ and $\Oconst\ge\tfrac{999}{1000}$):
\begin{align}
   S_3 &\le \ln(2 \Oconst n)   \frac{\frac{\pi^2}{6}-1}{\Oconst^2}  \left(1 + \frac{4}{\Oconst n - 2}\right)  +
 1.0011\ln n -11.5 +
   \frac{2n}{\Oconst n - 5/2}\nonumber\\
       &\le (\ln n + 0.693)\cdot 0.6463 %
       + 1.0011\ln n -11.5 + 2.003 
       \le 1.6474 \ln n - 9.\label{eq:S3:assymptotic}
\end{align}

\subsection{Estimate on $S_4$ (\cref{eq:caseOk})}\label{sub:S4}

The last estimate we need is for the value $S_4$, which counts what happens when $O_k$ is large. Let us first recall that:
\begin{equation*}
  S_4 = \sum_{k=2}^{n-2} \sum_{i=1}^{n-k} \sum_{j=\lceil\oconst n\rceil}^{n-k-i} q(j,i) \cdot \Pr[O_k=j \wedge \PF_k=i]
  = \sum_{k=2}^{n-2} \sum_{i=1}^{\lfloor\Oconst{}n\rfloor-k} \sum_{j=\lceil\oconst n\rceil}^{n-k-i} q(j,i) \cdot \Pr[O_k=j \wedge \PF_k=i].
\end{equation*}

To show the next lemma we will make use of Hoeffding's Inequality.

\begin{theorem}[Hoeffding's Inequality~(\cite{Hoeffding63}, Theorem 1)]\label{thm:Hoff}
Let $X_1,\ldots,X_d$ be independent random variables such that $0\le X_i \le 1$ for each $i$ and let $t>0$.
Then the following holds:
\[
  \Pr\Biggl[\sum_{i=1}^{n} X_i - \E[\sum_{i=1}^{n} X_i]\ge nt\Biggr]\le e^{-2nt^2}.
  \]
\end{theorem}

\begin{lemma}\label{lem:Ok}
  Let $n\ge 4$, $\mu\in[1,3]$, and let $\aleph_m\in\mathbb{Z}$ such that
  $\E[F(m)]\le 5\ln(m)+\aleph_m$ for all $2\le m<n$.
  For $n>k\ge \frac{2}{\oconst}\ln^\mu(n)$, we have
  \begin{align}
    \Pr\biggl[O_k>\oconst n\biggr]\le e^{\frac{-n\oconst^2}{2}}+ \frac{5+\frac{\aleph_{n-k}}{\ln(n)}}{n\ln^{\mu-1}(n)}. \label{eq:OkSum}
  \end{align}

    For $2\le k<\frac{2}{\oconst}\ln^\mu(n)$
  \[
    \Pr\biggl[O_k>\oconst n\biggr]\le \frac{5\ln(n)+\aleph_{n-k}}{\oconst n}.
    \]
\end{lemma}
\begin{proof}
  Now suppose that $k\ge \frac{2}{\oconst}\ln^\mu(n)$.
  Let $W_1,\ldots,W_{n-k}$ be indicator random variables where $W_i$ describes whether vertex $v_i$ is the first (and also the last) vertex in $V^\uparrow$ forming a 1-open walk for vertex $v_{k+i}$, for $1\le i\le n-k$.
  It is readily seen that $O_k=\sum_{i=1}^{n-k} W_i$.
  Let $w_i$ be the number of downward darts incident to vertex $v_{k+i}$ which 
  form a 1-open face.
  Since each such dart must be in a different temporary face, we have $\sum_{i=1}^{n-k} w_i\le T_{n-k}$.
  By linearity of expectation 
  \begin{eqnarray}
    \E[O_k]= \E\left[\sum_{i=1}^{n-k} W_i\right]=\sum_{i=1}^{n-k} \E[W_i]= \sum_{i=1}^{n-k} \frac{w_i}{k}\le \frac{T_{n-k}}{k}. \label{eq:exp}
  \end{eqnarray}

The proof follows as the sum of two conditional probabilities.
First, we show that \[\Pr\bigl[O_k>\oconst n\mid T_{n-k}\le n\ln^\mu(n)\bigr]\le \left( e^{\frac{\oconst^2}{2}}\right)^{-n}.\] 
By Inequality~(\ref{eq:exp}), $\E[O_k]\le\frac{T_{n-k}}{k}\le \frac{n\ln^\mu(n)}{k}$.
We apply Hoeffding's Inequality (Theorem~\ref{thm:Hoff}) on $ W_i$'s:
\begin{align}
    \Pr\biggl[O_k>\oconst n\biggm| T_{n-k}\le n\ln^\mu(n)\biggr]
   & \le \Pr\biggl[O_k - \E[O_k]\ge (n-k)\frac{\oconst nk-n\ln^\mu(n)}{k (n-k)}\biggm| T_{n-k}\le n\ln^\mu(n)\biggr] \nonumber\\
   & \le e^{-2(n-k)\frac{n^2(\oconst k-\ln^\mu(n))^2}{(n-k)^2 k^2}} \label{eq:Hoeff}\\
    &  \le e^{-\frac{2n(\oconst k-\ln^\mu(n))^2}{ k^2}}.\nonumber
    \end{align}
    As $k\ge \frac{2}{\oconst}\ln^\mu(n)$,
    \[
      \le e^{-\frac{2n\left(\frac{\oconst k}{2}\right)^2}{ k^2}}
      \le e^{-\frac{\oconst^2 n}{2}}
      \le \left( e^{\frac{\oconst^2}{2}}\right)^{-n}.
    \]

    Second, we use Markov's inequality with induction to conclude 
    \begin{equation}
      \Pr[T_{n-k}>n\ln^\mu(n)] \le\frac{5\ln(n)+\aleph_{n-k}}{n\ln^\mu(n)}
    = \frac{5+\frac{\aleph_{n-k}}{\ln(n)}}{n\ln^{\mu-1}(n)}.\label{eq:Markov}
  \end{equation}

    The proof of the first part follows by trivial estimates as a sum of both cases.

    For $k\le \frac{2}{\oconst}\ln^\mu(n)$ we use Markov's inequality with induction to conclude 
    \[\Pr[O_k>\oconst n]\le\frac{5\ln(n)+\aleph_{n-k}}{\oconst n}.\qedhere\]
\end{proof}

We show one more usefull lema before concluding the proof.
\begin{lemma}\label{lem:qnuBound}
  Let $k$ be an integer satisfying $\lfloor\Oconst n\rfloor> k\geq 2$ and let \hyperref[def:q]{$q$} be the function defined in \cref{def:q}.
  Then
  \[
    q(\lceil\oconst n\rceil,\lfloor\Oconst n \rfloor-k)\le \ln(\Oconst n)  - [k\ge 3] \ln(k-1.5) + [k=2], 
    \]
    (where the indicator function $[\mathcal P(k)]$ has value $1$ if the property $\mathcal P(k)$ holds, and is $0$ otherwise).
\end{lemma}

\begin{proof}
  As $k\ge 2$ we have $\lceil\oconst n\rceil +\lfloor\Oconst n \rfloor-k<n-1$, hence, we use Equation~(\ref{eq:defQ}):
  \begin{align*}
    q(\lceil\oconst n\rceil,\lfloor\Oconst n \rfloor-k) &= H_{\lfloor\Oconst n \rfloor-2}-H_{k-2}. %
\end{align*}

If $k=2$ then $q(\lceil\oconst n\rceil,\lfloor\Oconst n \rfloor-2)\le \ln(\Oconst n)+1$ by a trivial estimate.
Otherwise, using \cref{thm:Hnconv} we conclude:
  \begin{align*}
    H_{\lfloor\Oconst n \rfloor-2}-H_{k-2}   &\le \ln(\lfloor\Oconst n \rfloor-1.5)+\frac{1}{24 (\lfloor\Oconst n \rfloor-2)^2}  -\ln(k-1.5) -\frac{1}{24 (k-1)^2} \\
      &\le \ln(\lfloor\Oconst n \rfloor-1.5)  -\ln(k-1.5) \qedhere \\
\end{align*}
\end{proof}

We estimate sum \hyperref[eq:defS4]{$S_4$} as follows.
For the second equality we use the fact that for any $k \ge \lfloor \Oconst n \rfloor$ we must have $j\le n-\lfloor \Oconst n \rfloor-1<\lceil \oconst n \rceil$.
The third inequality holds as $q(\lceil \oconst n \rceil,n-k-\lceil \oconst n \rceil)$ is the maximum possible value function \hyperref[def:q]{$q$} attains the given range of $i,j$.
Indeed, $1+\lceil \oconst n\rceil\le i+j\le n-k$, hence by \cref{eq:defQ}, the function $q(j,i)$ is the largest when $j$ is the smallest possible and then $i+j$ is the largest possible.

\begin{align}
  S_4 &= \sum_{k=2}^{n-2} \sum_{i=1}^{n-k} \sum_{j=\lceil\oconst n\rceil}^{n-k-i} q(j,i) \cdot \Pr[O_k=j \wedge \PF_k=i]  \nonumber\\ 
      &= \sum_{k=2}^{\lfloor \Oconst n \rfloor-1} \sum_{i=1}^{n-k} \sum_{j=\lceil\oconst n\rceil}^{n-k-i} q(j,i) \cdot \Pr[O_k=j \wedge \PF_k=i]  \nonumber\\ 
      &\le \sum_{k=2}^{\lfloor \Oconst n \rfloor-1} q\left(\lceil \oconst n\rceil,\lfloor \Oconst n \rfloor-k\right) \Pr[O_k\ge\oconst n] \label{eq:S4}\\ 
      &\le  q(\lceil \oconst n\rceil,\lfloor \Oconst n \rfloor-k)\Pr[O_2\ge\oconst n]+ \sum_{k=3}^{\left\lceil \frac{2}{\oconst}ln^\mu(n)\right\rceil -1}  q(\lceil \oconst n\rceil,\lfloor \Oconst n \rfloor-k)\Pr[O_k\ge\oconst n]\nonumber\\
      &\,\quad +\sum_{k=\left\lceil\frac{2}{\oconst}ln^\mu(n)\right\rceil}^{\lfloor \Oconst n \rfloor-1}  q(\lceil \oconst n\rceil,\lfloor \Oconst n \rfloor-k)\Pr[O_k\ge\oconst n] \nonumber%
\end{align}

For $0<b<a$, let $\aleph^a_b\df \max_{b<i<a} \aleph_i$.
Using %
\cref{lem:Ok} and \cref{lem:qnuBound} (assuming $n\ge 243$) we conclude:
\begin{align}
  S_4 &\le (\ln(\Oconst n)+1)\frac{5\ln(n)+\aleph_{n-2}}{\oconst n}
  \nonumber \\&\,\,+ \left(  \left\lceil \frac{2}{\oconst}ln^\mu(n)\right\rceil -3\right) \cdot  \ln(\Oconst n)  \cdot \frac{5\ln(n)+\aleph_{n-\left\lceil \frac{2}{\oconst}ln^\mu(n)\right\rceil +1}^{n-3}}{\oconst n}
  \nonumber  \\&\,\,+ \left( \left\lfloor \Oconst n \right\rfloor -1 -  \left\lceil \frac{2}{\oconst}ln^\mu(n)\right\rceil +1\right) \cdot\left( \ln(\Oconst n)-\ln\left(\left\lceil\frac{2}{\oconst}ln^\mu(n)\right\rceil -1.5\right) \right) \cdot \left( e^{\frac{-n\oconst^2}{2}}+ \frac{5\ln n+\aleph_{\lceil\oconst n\rceil}^{n-\left\lceil \frac{2}{\oconst}ln^\mu(n)\right\rceil}}{n\ln^{\mu}(n)} \right) 
   \nonumber \\
     &<  \Oconst n \ln(\Oconst n) e^{\frac{-n\oconst^2}{2}}   
     +\frac{\nu\ln(\Oconst n)\left(5\ln n+\aleph_{\lceil\oconst n\rceil}^{n-\left\lceil \frac{2}{\oconst}ln^\mu(n)\right\rceil}\right)}{\ln^{\mu}(n)} %
   +\frac{2\ln^\mu(n) \ln(\Oconst n)\left(5\ln n+\aleph_{n-\left\lceil \frac{2}{\oconst}ln^\mu(n)\right\rceil +1}^{n-2}\right)}{\oconst^2 n}. \label{eq:S4final}
   \end{align}

 }

 \lv{
\section{Computer-evaluated estimates for small values of $n$}\label{sec:smallValues}

\propComput*

Up to $n\le7$ the exact values are known, see also \cref{tab:face distro} in the introduction.
They can be computed by exhaustive enumeration of all possible embeddings.
Computing higher values might require additional insight to cut down the size of the search space.
Therefore for the computation of small values of $n>7$, we used a different approach. 
\lv{We provide a simple program in Sage\footnote{Available in the sources of our arXiv submission~\cite{ourArxiv}; files \texttt{Num\_bounds.sage} and \texttt{Num\_bounds\_faster.sage}. We also provide there the computed data using this program for $7\le n \le \ComputThr$ in file \texttt{data.txt}.} that is used only to numerically compute the exact upper bounds as derived in the preceding proofs, using previously computed values for smaller numbers of vertices.}

\begin{figure}[tb]
  \begin{center}
  \includegraphics[scale=0.35]{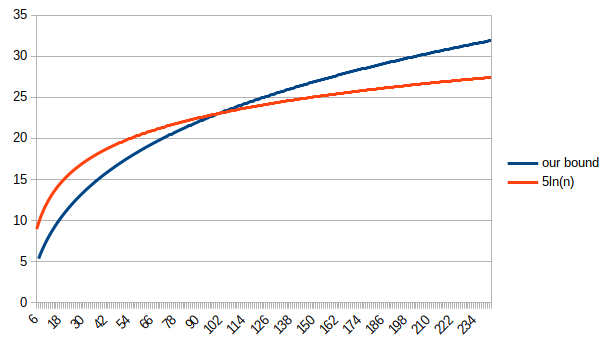}
  \includegraphics[scale=0.35]{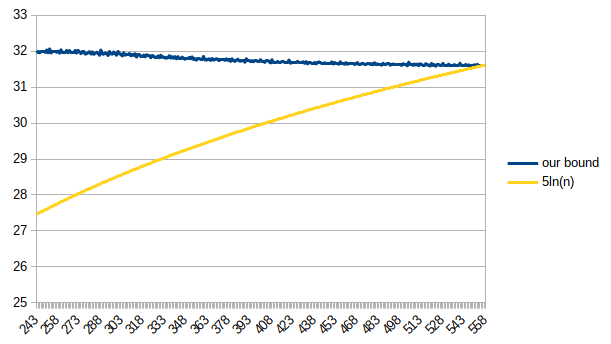}
\end{center}
\caption{Computer evaluated bound given by \cref{eq:logsqForComputation} for $6\le n\le 242$ in the \emph{left}. In the \emph{right} chart, for $243\le n\le 558$, we provide the upper bound of $5\ln(n)+5$.}\label{fig:logsqBound}
\end{figure}

For $n\le 242$, we use bound provided by \cref{thm:logsqBound}.
In fact, for computer computation, we used a very slightly sharper bound which appears in the proof as \cref{eq:logsqForComputation}.
In this range of parameters, $\E[F(n)] < 5\ln(n)+5$; see \cref{fig:logsqBound} the left chart.

For $243\le n\le \ComputThr$, we used partial estimates from the proof in order to minimize the accumulation of overestimation in our analysis.
As in the proof of \cref{thm:Completegraphslargevalues} we express $\E[F]\le S_1+S_2+S_3+S_4$.
Recall that some estimates use induction and, hence, in such cases, we used computer-calculated upper bounds. 
That is, we use bounds for $n'<n$ that we already computed (denoted as $\beta(n')$).
We now describe what we used in our program to upper bound those quantities.
A similar applies to the value of $1/2<\Oconst<1$, which is a split point between the cases.
In principle, $\nu$ can be different for each $n$.
However, to reduce running time, we only considered a couple of values around the value $\nu$ that have performed the best for $n-1$.
For $S_1$, we used a simple estimate in \cref{eq:S1}.
For $S_2$, we used estimate given by \cref{eq:S2}. 
For $S_3$, we used estimate given by \cref{eq:small}, where $M_k$ was estimated by \cref{eq:f(i) bounded by maximum} and $E_k^2=\E[\PF_k^2]$ as estimated by \cref{eq:Eksq}. 
For $S_4$, we used \cref{eq:S4}, where $Pr[O_k\ge \oconst n]$ is estimated in \cref{lem:Ok}.
There, we do one more optimization.
We find a minimum value of sum in \cref{eq:OkSum} by checking all admissible $x$ in the following equation:
\[
  h(x)\df
    e^{-2\frac{(n\oconst k-x))^2}{(n-k) k^2}}+
  \frac{\beta(n-k)}{x}
  .
\]
Observe that in \cref{lem:Ok} this equation originates in \cref{eq:Hoeff} and \cref{eq:Markov}, where the $x$ is fixed to be $n\ln^\mu(n)$.
As a final upper bound of $Pr[O_k\ge \oconst n]$, we take the smaller from $h(x)$ and $\frac{\beta(n-k)}{\oconst n}$.

The upper bound given by this part of computation is $5\ln(n)+5$ for $243\le n\le 558$ and for $559\le n\le \ComputThr$ even $5\ln(n)$, see \cref{fig:logsqBound} the right chart and \cref{fig:logBound} for details.
Note that the sudden increase of our bound in \cref{fig:logBound} (at value $34475$) is caused by a weaker optimization of value $x$ which is fixed to a three possible values as well as $\mu$ was fixed to $0.94$.
That leads to a faster\footnote{This optimalization is implemented in the file \texttt{Num\_bounds\_faster.sage}.} computation compensated by a weaker bound.
All together, the analysis and computations described in this section prove \cref{prop:smallValues}.

\begin{figure}
  \begin{center}
  \includegraphics[scale=0.45]{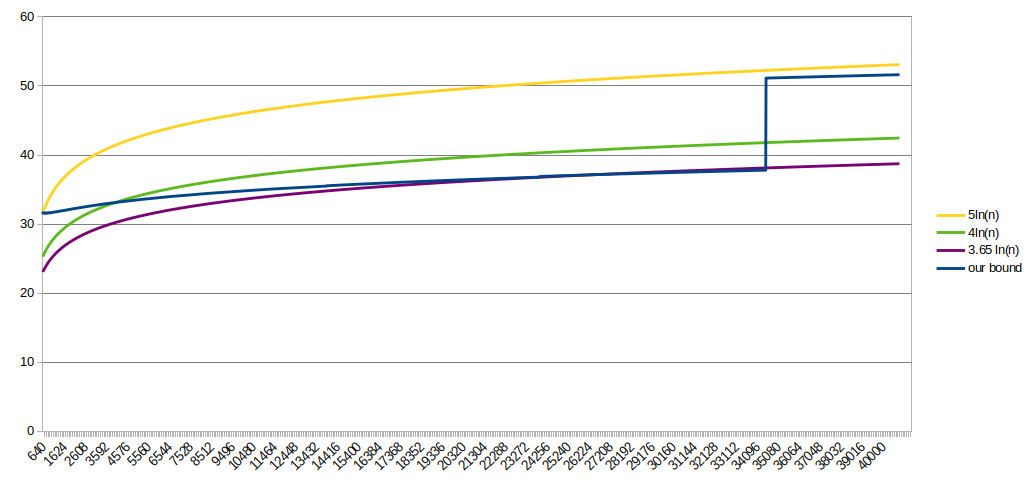}
\end{center}
\caption{Computer evaluated upper bound of $5\ln(n)$ for $559\le n \le 40748$.}
\label{fig:logBound}
\end{figure}
}

\section{Lower bound for complete graphs}\label{sec:lb}

In this section, we provide a counterpart to \cref{thm:completeGraphsAllValues}---a logarithmic lower bound on the expected number of faces~\cref{thm:lb}.

\loglb*

\begin{proof}
  We partition the set of possible faces according to their length and we only count those that are easy to count:
  Let $ F'_k$ be the number of faces having $k$ distinct vertices and $k$~edges on their boundary.
  There are $\tfrac 1k n(n-1)\cdots(n-k+1)$ possibilities for such a face. Each of them becomes a face of a random embedding with probability $(n-2)^{-k}$. Together, we get
  (using Bernoulli's inequality):
  \[
    \E[F'_k] = \frac 1k \prod_{i=0}^{k-1} \frac{n-i}{n-2}
             \ge \frac 1k \prod_{i=0}^{k-1} \Bigl(1 - \frac in\Bigr)
             \ge \frac 1k \Bigl( 1 - \sum_{i=0}^{k-1}\frac in\Bigr)
             \ge \frac 1k \Bigl( 1 - \frac {\binom k2}n \Bigr)
             = \frac{1}{k} - \frac{k-1}{2n}.
   \]
  
   \begingroup
\allowdisplaybreaks
   Let $m \df \lfloor \sqrt{2n}\rfloor$. Then $F \ge F_3' + F_4' + \cdots + F_m'$, and
   \begin{align*}
      \E[F] &\ge \sum_{k=3}^m \E[F'_k] \ge \sum_{k=3}^m \left(\tfrac{1}{k} - \tfrac{k-1}{2n}\right) = 
      H_m - \tfrac{3}{2} - \tfrac{1}{2n}(2+3+\cdots+(m-1)) \ge H_m - 2 \\ &= H_{\lfloor \sqrt{2n}\rfloor} - 2 \ge \ln(\sqrt{2n}) + \Bigl( \ln(\lfloor \sqrt{2n}\rfloor) - \ln(\sqrt{2n})\Bigr)- 2 + \gamma \\
      &\ge \frac{1}{2} \ln(n) + \frac12 {\ln(2)} + \ln(1/2) -2 + \gamma > \frac{1}{2} \ln(n) - 2  
   \end{align*}
   \endgroup
   We have used estimate $H_m \ge \ln(m) + \gamma$ (implied by \cref{thm:Hnconv}) and $\lfloor \sqrt{2n}\rfloor/\sqrt{2n} \ge 1/2$. 
\end{proof}

\section{The $\ln^2(n)$ upper bound on $G(n,p)$} %
\label{sec:random}

In this section, we prove our bound for random embeddings of the random graph $G(n,p)$ when $p$ is not too small. 
We will use a random process very similar to \hyperref[def:rpA]{Random process A}, except that there will be an extra choice at each step of whether we include each edge or not.
More precisely, after processing each edge $(i,j)$ we include it with probability $p$.
This gives a uniformly random embedding of some subgraph of $K_n$, which is distributed according to the distribution of $G(n,p)$.

\bigskip

\noindent
{\bf Random process C}\phantomsection\label{def:rpC}

\begin{enumerate}
\item Order the vertices of the graph $v_{n}, \dots, v_1$ arbitrarily, we process the vertices in this order. 
\label{step:special1}
\item Consider vertex $v_k$ for $k\in[n]$.\label{step:vertex1}
  Label the darts of~$D_k$ as $\{d_1, \dots, d_{n-1}\}$ arbitrarily. We define $R_k$ as this order, that is 
  $R_k(d_i) = d_{i+1}$ (except $R_k(d_{n-1}) = d_{1}$).
        Let $C_k \df \{n,n-1,\dots,k+1, u,u,\dots,u\}$ where there are $k-1$ copies of the symbol $u$ that represent that the dart choosing this option remains unpaired.
        This is the set of choices of where the darts may lead at the end of this step.
    \begin{enumerate}
      \item {\bf Random choice 1:}\phantomsection\label{def:rpC:rc1} For each copy of $u$ in $C_k$, mark it for deletion, by replacing it with the symbol $\xcancel{u}$, with probability $1-p$.
      \item 
        Process darts in $D_k$ in order $d_1, d_2, \dots, d_{n-1}$.  If there is at least one copy of $u$ in $C_k$, give $d_1$ the label $u$, remove one copy of $u$ from $C_k$, and proceed processing $d_2$.  Otherwise, if $k>1$ then give $d_1$ the label $\xcancel{u}$, remove one copy of $\xcancel{u}$ from $C_k$, and proceed processing $d_2$.  If $k=1$, start by processing $d_1$.
        
       Suppose we are processing dart $d_\ell$. 
      \item                   
{\bf Random choice 2a:}\phantomsection\label{def:rpC:rc2a}
       Pick a symbol from the set $C_k$ uniformly at random to label $d_\ell$ with, then remove this choice from $C_k$. 
                   \begin{itemize}
                     \item Case 1: \emph{The choice was some  $i \geq k+1$}. 
{\bf Random choice 2b:}\phantomsection\label{def:rpC:rc2b} 
                       Then pick an unpaired dart $d'$ uniformly at random from the unpaired darts at $v_i$.
                       If the dart $d'$ has label $u$, then add the transposition $(d',d_\ell)$ to the permutation $L$. 
                       If the dart $d'$ has label $\xcancel{u}$, then remove $d_\ell$ from~$D_k$, $d'$ from~$D_i$, and redefine $R_k$, $R_i$ appropriately.
                       \item Case 2: \emph{The choice was some $u$ or $\xcancel{u}$.}  Then leave dart $d_\ell$ unpaired.
                       \end{itemize}
                   If $\ell < n-1$, then proceed with the next dart in the order.
    \end{enumerate}
                   If $k\ge2$, then proceed with the next vertex in the order.\hfill $\lrcorner$

\end{enumerate}

\bigskip

\begin{observation}\label{obs:rand_proces_C}
  \hyperref[def:rpC]{Random process C} outputs a combinatorial map $(D,R,L)$ with the following properties:
 \begin{itemize}
     \item The underlying graph $G$ has distribution $G(n,p)$. 
     \item For each fixed graph $G_0$, when we restrict on the instances when $G=G_0$, the map $(D,R,L)$ 
       is a random embedding of~$G_0$ (as defined in \cref{sec:REintro}).
 \end{itemize}
\end{observation}

\medskip

Note that the underlying graph of the embedding outputted by this process may not be connected.  To allow for that case, we define the \emph{number of faces} for maps that are not connected as the sum of the number of faces in each connected component minus the number of connected components plus one. This corresponds to the fact that one can always pick an arbitrary face $f_1$ in one connected component $H_1$ and an arbitrary face $f_2$ in another connected component $H_2$ and insert the whole embedding of $H_2$ inside $f_1$ using $f_2$ as a boundary. Note that each isolated vertex (in fact any tree) always contributes zero towards the number of faces.

For a graph embedding $M$, let $c(M)$ denote the number of connected components of the underlying graph.  We define $F(n,p)$ as the random variable $F(M) - c(M) + 1$, for a random map $M$ obtained by \hyperref[def:rpC]{Random process C}.  We are now ready to prove the main result of this section.

\randomGraphs*

\smallskip

\begin{proof}
  Suppose we are at some step of \hyperref[def:rpC]{Random process C}. Then we have a partially constructed graph embedding, which we denote by the partial map $M$.  In the proof of Theorem \ref{thm:logsqBound}, we added a completed face at a step of \hyperref[def:rpA]{Random Process A} if and only if we paired the active dart $d_\ell$ with an unpaired dart in a partial walk starting or ending with $d_\ell$.  However, with \hyperref[def:rpC]{Random process C}, we could add a temporary face at this step, then at some later steps remove all the unpaired darts in this temporary face.  This would make this temporary face into a completed face in the final embedding.

In order to analyse these ways of adding completed faces, we define a new random variable.  Let $F_\ell^k$ be the random variable for the number of temporary faces created after processing $d_\ell$ at step $k$ that satisfy one of the following conditions:
\begin{itemize}
    \item  They are a completed face.
    \item  They are a $j$-open face for some $j>0$, where each of the $j$ unpaired darts in this temporary face has label $\xcancel{u}$.  %
\end{itemize} 

\begin{claim}
 We have $\E[F(n,p)] \leq \sum_{k=1}^n \sum_{\ell=1}^{n-1} \E[F_\ell^k].$
\end{claim}

\noindent\textit{Proof of claim.~}
After each dart $d_\ell$ is paired we add either one or two new temporary faces to $M$.  If we make a completed face, then this will not be affected by the rest of the random process, and so will appear in the final embedding.  If we make a $j$-open face for $j>0$, then this temporary face becomes a completed face in the final embedding if and only if every unpaired dart contained in it is removed at a later step of the random process.  This happens if and only if every unpaired dart in the temporary face is given the label $\xcancel{u}$.  This proves the claim.  \hfill$\diamondsuit$

\medskip

We continue by estimating $\E[F_\ell^k]$.  Firstly, for $\ell = n-1$ and $k=1$, we give the worst possible bound of $\E[F_{n-1}^1] \leq 2$.  Now suppose $\ell<n-1$ or $k>1$.  Suppose we are at the step of \hyperref[def:rpC]{Random process C} where we are processing dart $d_\ell$ at vertex $v_k$.  If $k>1$, recall that the dart $d_1$ has the label $u$ or $\xcancel{u}$.  Let $M$ be the partial map at the start of this step.  Further, suppose that we pair $d_\ell$ with some dart $d$ at this step.  We have two cases, see Figure \ref{fig:gnprandomprocess} for an example of this analysis.

\begin{figure}
    \centering
    \includegraphics[scale=0.79]{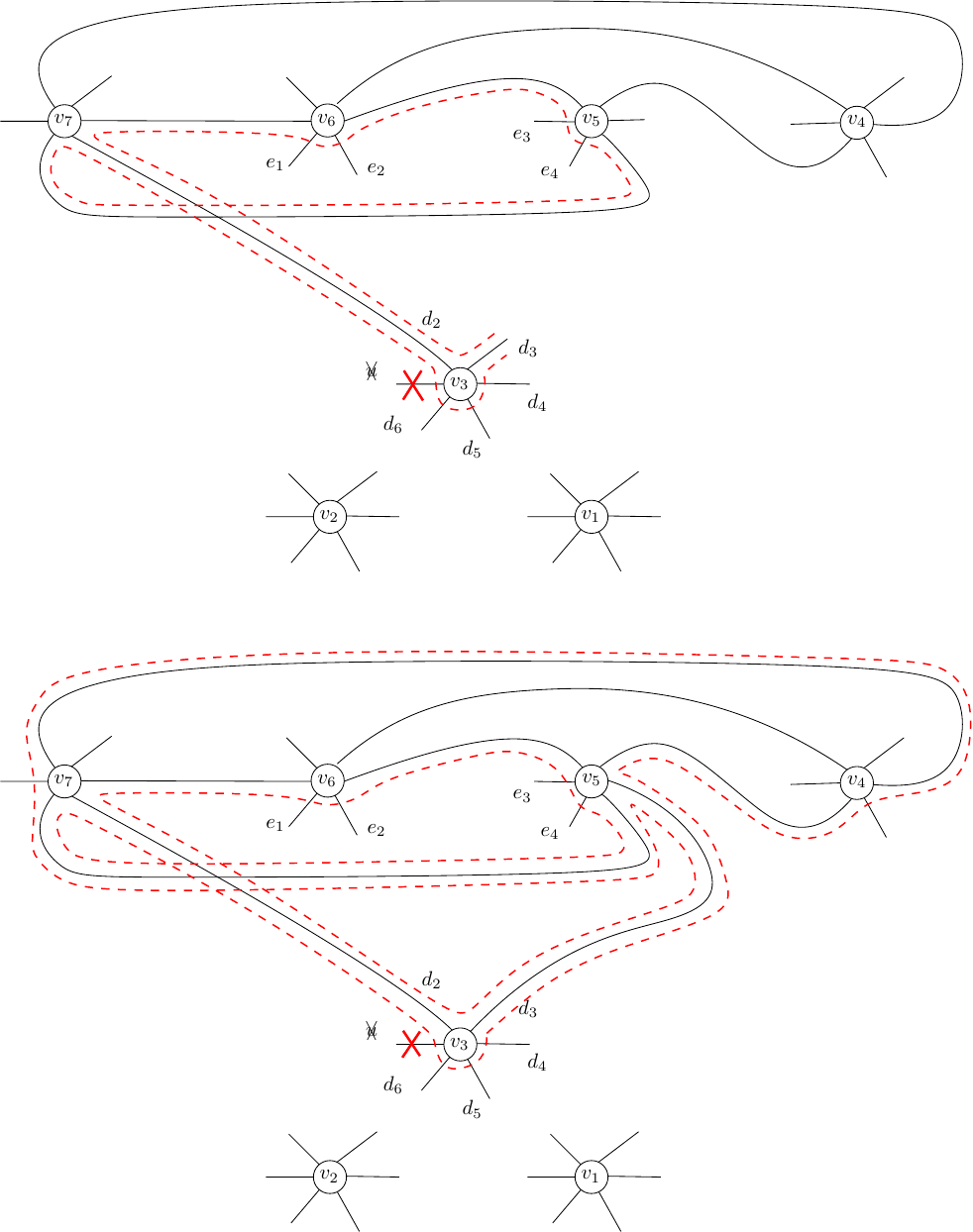}
    \caption{
      An example of a step of \hyperref[def:rpC]{Random Process C}, where every one of the darts $u$ has been marked for deletion.  Here we are processing the dart $d_3$ at vertex $v_3$.  At this step our set of choices is $C_3 = \{4,5,6,\xcancel{u}\}$.  The temporary facial walk starting at $d_3$ is traced out in red dotted line.  It turns out that this temporary facial walk ends at $d_3$.  If we choose $5 \in C_3$, then there are three choices of dart at $v_5$ to pair with to make an edge.  If we pair into $e_3$, then we make a new $2$-open face.  If we pair into $e_4$ we make a new $3$-open face.
      On the lower picture, as we paired $d_3$ to a dart of a different face and we know that $F_{k,\ell}=0$ in this case.
      Potentially, we could have created $F^\sigma_k$, which is traced in red dotted line.
      This will become $F^\sigma_k$ if later all $d_4$, $d_5$, and $d_6$ will either be unpaired (and so also marked for deletion) or they will be paired with a dart in $\Vup$ which is marked for deletion already.
    }
    \label{fig:gnprandomprocess}
\end{figure}

Case 1: The darts $d_\ell$ and $d$ are in different temporary faces in $M$.  Then pairing $d_\ell$ and $d$ joins these two temporary faces together to make a new temporary face.  If $\ell<n-1$ then the dart $d_{\ell+1}$ is unpaired, and so this new temporary face contains the unlabelled dart $d_{\ell+1}$.  This means this temporary face contains a dart not labelled $\xcancel{u}$, and so $F_\ell^k=0$.  If $\ell=n-1$ and $k<n$, then $R_k(d_{n-1}) = d_1$ has the label $u$ or $\xcancel{u}$.  Therefore, we may create an additional temporary face containing only darts labelled $\xcancel{u}$ at this step only if $d_1$ has the label $\xcancel{u}$.

Case 2: The darts $d_\ell$ and $d$ are in the same temporary face in $M$.  Then pairing $d_\ell$ and $d$ splits this temporary face into two new temporary faces.  More precisely, suppose a walk around the temporary face containing $d_\ell$ visits unpaired darts $d_\ell, *, e_{j-1} , e_{j-2} , \dots , e_1$ in this order where $*$ denotes some unpaired darts around vertex $v_k$.  Then if we choose $d=e_{i+1}$, we make an $i$-open face and another temporary face that contains $d_{\ell+1}$.  As in the previous case, if $\ell<n-1$ then the dart $d_{\ell+1}$ is unlabelled.  This means this new temporary face containing $d_{\ell+1}$ does not contribute to $F_\ell^k$.  If $\ell = n-1$ then we make an additional temporary face containing only darts labelled $\xcancel{u}$ only if $d_1$ has the label $\xcancel{u}$.  

Notice that Case 1 adds $0$ faces if $\ell < n-1$ and at most $1$ face if $\ell = n-1$, while Case 2 adds at most $1$ face if $\ell < n-1$ and at most $2$ faces if $\ell = n-1$.  Moreover these faces are added in Case 1 and 2 in the exact same circumstances.  Therefore, since we are looking for an upper bound for $\E[F_\ell^k]$, we may disregard Case 1 and just analyse Case 2.

The $i$-open face created in Case 2 contributes to $F_\ell^k$ if and only if:
\begin{itemize}
    \item The pairing dart $d$ has the label $u$, and so the edge $(d_\ell, d)$ is included in the embedding.  This happens with probability $p$.
    \item Each of $e_1, e_2, \dots, e_{i}$ has the label $\xcancel{u}$.  Each of these darts has label $\xcancel{u}$ independently with probability $1-p$ due to \hyperref[def:rpC:rc1]{Random Choice 1}, so the probability of them all having label $\xcancel{u}$ is $(1-p)^{i}$.
\end{itemize}
Therefore the probability of this happening is $p(1-p)^{i}$.

Notice that there is at most one choice of dart $d=e_{i+1}$ to pair with which makes an $i$-open face for each $i$.  Suppose that dart $e_{i+1}$ is at vertex $v_t$.  Then dart $d=e_{i+1}$ is chosen if and only if the label $t$ is chosen from $C_k$ in \hyperref[def:rpC:rc2a]{Random Choice 2a}, and the dart $e_{i+1}$ is chosen from the $k$ unpaired darts at $v_t$ in \hyperref[def:rpC:rc2b]{Random Choice 2b}.  There is therefore a probability of at most $\tfrac{1}{(n-\ell)k}$ that we make an $i$-open face for each $i$.  The probability that this $i$-open face becomes a completed face in the final embedding is $p(1-p)^{i}$.  Therefore for $\ell < n-1$ we have
\begin{align} \label{eqn:Efnkbound}
     \E[F_\ell^k] < \sum_{i \geq 0} \frac{p}{k(n-\ell)} (1-p)^{i} = \frac{1}{k(n-\ell)}.
 \end{align}

If $\ell = n-1$ and $k>1$, by the same formula as \eqref{eqn:Efnkbound} there is a probability of at most $\tfrac{1}{k}$ of adding a face not containing $d_{1}$ that becomes a completed face in the final embedding.  Since $\ell=n-1$ we may make an additional temporary face containing dart $d_1$ and only darts labelled $\xcancel{u}$ at this step only if $d_1$ has the label $\xcancel{u}$.  Recall that in \hyperref[def:rpC]{Random process C} we set the label of $d_1$ as $u$ if possible.  It is only not possible if in \hyperref[def:rpC:rc1]{Random Choice 1} we replace each $u$ with $\xcancel{u}$.  There are $k-1>0$ copies of $u$ and each one is independently set as $\xcancel{u}$ with probability $1-p$, so this happens with probability $(1-p)^{k-1}$.

 Therefore, for $\ell = n-1$ and $k>1$ we have
\begin{align*} 
     \E[F_{n-1}^k] \leq \frac{1}{k} + (1-p)^{k-1}.
 \end{align*}

    Summing over all values of $\ell$ for $k>1$ we get
    \[
      \sum_{\ell=1}^{n-1} \E[F_\ell^k] \leq (1-p)^{k-1} + \frac{1}{k} \sum_{\ell=1}^{n-1} \frac{1}{n-\ell} = \frac{H_{n-1}}{k} + (1-p)^{k-1}.
    \]
    For $k=1$, we obtain
    \[
        \sum_{\ell=1}^{n-1} \E[F_\ell^1] \leq 2 + \sum_{\ell=1}^{n-2} \frac{1}{n-\ell} = H_{n-1} + 1.
      \]
    Then summing over all $k$ we obtain
    \begin{align*}
      \E[F(n,p)] &\leq \sum_{k=1}^{n} \left( \frac{H_{n-1}}{k} + (1-p)^{k-1} \right)  \\
      &< H_{n-1} H_n + \frac{1}{p}. \qedhere
    \end{align*}
\end{proof}

This gives a corollary on the expected number of faces of almost all graphs.

\begin{corollary} \label{cor:almostallpolylog}
For any graph~$G$ put $X(G) := \E[F(G)]$, so $X(G)$ is the expected number of 
faces of~$G$ in a random embedding. When we let $G$ be random, specifically 
$G \sim G(n,p)$ we have the following bound
\[
  Pr[X(G) \ge t (H_n^2+\frac{1}{p}) ] \le \frac{1}{t}.
   \]
\end{corollary}

In particular, for constant $p$, most graphs have expected number of faces just a bit above $(\ln n)^2$. 

\begin{proof}
  We use Observation~\ref{obs:rand_proces_C}, Theorem~\ref{thm:randomgraphs} and Markov's inequality. 
  Finally, we note that $\E[X(G)] = \E[F(n,p)]$. 
\end{proof}

\section{The {$\Theta(\log(n))$} bounds for graphs with fixed degree sequence}
\label{sec:randomBnd}

For small values of $p$, we first refer to a result of Chmutov and Pittel \cite{ChPi16}. The authors consider a random surface obtained by gluing together polygonal disks. Taking the dual embeddings of this problem shows this is equivalent to studying random embeddings of random graphs with a fixed degree sequence, where we allow loops and multiple edges. In this case, a corollary to their main result gives that the expected number of faces is asymptotic to $\ln(n) + O(1)$.
Their method of proof uses representation theory. In particular, they use representation theory of the symmetric group and recent character bounds of Larsen and Shalev \cite{larsen2008characters}. We start by giving a combinatorial proof that the expected number of faces in this model is $\Theta(\log(n))$. We then extend this result to random simple graphs with a fixed degree sequence. This second result is not equivalent to a conjugacy class product in the symmetric group. Therefore standard representation theoretic techniques do not apply, but we can still make use of our combinatorial reformulation. 

We do not use a random process. Instead, we count two different things and combine them to give estimates on the expected number of faces. Firstly, we count all the different possible faces which could appear in a random embedding on a fixed degree sequence. Then for each possible face, we estimate the number of embeddings that contain it. When we study random multigraphs, these numbers can be estimated directly. When we restrict to simple graphs, we will appeal to a result of Bollob\'as and McKay \cite{BM86}. This will help us estimate the fraction of faces and embeddings which are simple.

\subsection{Random multigraphs}\label{sec:rand_multi}

We are interested in random graphs with a fixed degree sequence generated using
the configuration model (see \cite{Wo99} for an in-depth description of this
model). Fix an arbitrary sequence of integers 
${\mathbf d} = (t_{1},t_{2},\ldots,t_{n})$ satisfying $2 \leq t_{1} \leq t_{2} \leq \cdots \leq t_{n}$ 
and $\sum_{i=1}^n t_{i} \equiv 0 \pmod{2}$. Whereas the second condition on~${\mathbf d}$ is satisfied 
by the degree sequence of all graphs, the first
condition eliminates vertices of degree $0$ or $1$ that do not affect the number
of faces in a random embedding. We also fix the integer $m = m_{\mathbf d} \df \frac{1}{2}\sum_{i=1}^n t_i$ 
corresponding to the number of edges in a graph with degree sequence~${\mathbf d}$. 

In this subsection, we prove the following result.

\thmmulti*

Given a set $D$ of $2m$ darts and a partition $\lambda \vdash 2m$, we write $C_\lambda$ for the conjugacy class in $Sym(D)$ comprised of all permutations with cycle type $\lambda$. Notice that we can think of ${\mathbf d}$ as a partition of $2m$, so that a random map with degree sequence ${\mathbf d}$ is defined by a map $M=(D,R,L)$ satisfying $D = \{1,\ldots,2m\}$, $R \in C_{\mathbf d}$, and $L \in C_{2^m}$. 

Recall that the expected number of faces in a random map with degree sequence ${\mathbf d}$ is just
the expected number of cycles in a product $R\circ L$ of a pair of random permutations 
$R \in C_{\mathbf d}$ and  $L \in C_{2^{m}}$. Because we are picking $L$ uniformly from 
the conjugacy class $C_{2^m}$, 
the distribution of the number of cycles in $R\circ L$ does not depend on the particular permutation 
$R \in C_{\mathbf d}$. We may therefore fix $R  \in C_{\mathbf d}$ while letting $L$ range
over all possibilities in $C_{2^m}$.
Formally, instead of the set of all maps $\{(D,R,L) \mid L \in C_{2^{m}}, R \in C_{\mathbf d}\}$
we will, in whole Section~\ref{sec:randomBnd}, study 
\[\mathcal{M}_{\mathbf d} = \mathcal{M}^R_{\mathbf d} \df \{(D,R,L) \mid L \in C_{2^{m}}\}, \] 
the \emph{set of all maps with degree sequence $\mathbf d$ up to isomorphism of maps}. 

We define the set of \emph{possible faces} (of maps with the fixed rotation system $R\in C_{\mathbf d}$) to be 
\[\Phi = \{f \mid f \text{ is a cycle of } R\circ L \text{ for some } L \in C_{2^m} \}. \]
In what follows, we use two different measures for the size of a face in $\Phi$.

\begin{Definition}[Face length]\label{def:faceLength}
    Given a possible face $f \in \Phi$ we define
   $u(f)$ as the \emph{unique length} of the face and is defined as the number of different edges in the face.
\end{Definition}

For example, suppose a face has length $k$, visits $k-2$ edges once and visits one edge twice by traveling on either side of this edge. Then this face has unique length $u(f) = k-1$ as it visits $k-1$ different edges.  We will enumerate faces using their unique length.

The natural setting for our analysis is multigraphs (allowing both loops and multiple edges), as restricting the question to simple graphs means restricting $L$ to subsets of their conjugacy classes.
Moreover, by allowing parallel edges and loops we get a very simple formula for the number of maps containing a given element of $\Phi$.
First, we look what happens when we fix a particular face $f\in \Phi$ with $u(f)=k$.

\begin{lemma}\label{lem:completion}
  Each possible face $f\in\Phi$ with $u(f)=k$ appears as a face in $\vert C_{2^{m - k}} \vert$ embeddings
  of some graph with degree sequence $\mathbf d$. 
\end{lemma}

\begin{proof}
  Recall that we have fixed a permutation $R \in C_{\mathbf d}$ referring to the rotation systems of the darts at the vertices. We are therefore counting the number of 
  permutations $L \in C_{2^{m}}$, such that $R\circ L$ contains the given face $f$; 
  equivalently, the $k$ different edges of $f$ must all appear in $L$.
    
    Now, the key observation is that the remaining darts can be joined in any way in order to make an embedding containing this face. This means we have free choice for an edge permutation on the remaining darts. Since there are $k$ unique edges in $f$, we, therefore, have free choice for the other $m-k$ edges. There are $\vert C_{2^{m - k}} \vert$ possible edge schemes on this number of darts, giving the result.
\end{proof}

Notice that the number of embeddings in the above lemma only depends on the unique length of the
face and not the structure of the face. This will be used in Lemma~\ref{lem:hk}. 

For a permutation $\tau \in S_n$, define $c(\tau)$ as the number of cycles in this permutation. %
Then the expected number of faces in a random element of $\mathcal{M}_{\mathbf d}$ (denoted as $\E(F_{\mathbf d})$) is given by a simple counting over all possible embeddings:
\begin{align}
\label{eq:exp:basic}
    \E(F_{\mathbf d}) = \frac{1}{|C_{2^{m}}|} \sum_{L \in C_{2^{m}}} c(R \circ L). 
\end{align}

We define $h_k$ as the number of faces $f \in \Phi$ such that $u(f) = k$. \phantomsection \label{def:hk}
Using this notion, we can express the expected number of faces as follows.

\begin{lemma}\label{lem:hk}
  Let $\mathbf d$ be a degree sequence.
  Let $\E(F_{\mathbf d})$ denote the expected number of faces of a random map $M\in\mathcal M_{\mathbf d}$. 
  Then
\[
  \E(F_{\mathbf d}) = \sum_{k=1}^m \frac{h_k}{(2m-1)(2m-3)(2m-5)\dots(2m-2k+1)},
\] 
      where $m$ denotes the number of edges in any $M$ with degree sequence $\mathbf d$. 
\end{lemma}

\begin{proof}
  We start with \cref{eq:exp:basic}, which we can rearrange by summing over all possible faces $f \in \Phi$ instead of $L \in C_{2^{m}}$ to get 
\[ 
  \E(F_{\mathbf d}) = \frac{1}{|C_{2^{m}}|} \sum_{f \in \Phi} \left|\{L \in C_{2^m} \mid f\in R\circ L\}\right|.
\]
  We can then rearrange the previous formula in terms of \hyperref[def:hk]{$h_k$} using \cref{lem:completion} for any possible unique length of $k$ to obtain
\begin{equation}\label{eq:framework}
  \E(F_{\mathbf d}) = \frac{1}{\vert C_{2^{m}} \vert} \sum_{k=1}^{m} h_k \vert C_{2^{m - k}} \vert. 
\end{equation}

To finish the proof, we calculate the fraction of the two sizes of conjugacy classes. 
It is straightforward to see that $\vert C_{2^j}\vert = \frac{(2j)!}{j! 2^j}$.
Hence, we have
  \begin{align}
    \frac{\vert C_{2^{m - k}} \vert}{\vert C_{2^{m}} \vert} 
        &= \frac{(2m-2k)! \, (m)! \, 2^{m}}{(2m)! \, (m - k)! \, 2^{m - k}} 
        = \frac{2^k m^{\underline{k}}}{(2m)^{\underline{2k}}}  \nonumber  \\ 
        &= \frac{1}{(2m-1)(2m-3)(2m-5)\dots(2m-2k+1)}.\label{eq:c2mk}
\end{align}

\end{proof}

Recall, that all darts $d \in f$ are ``outgoing''; that is, when we follow the
boundary of~$f$, we next apply $L$ to get to a different vertex and then~$R$.
Let $L(f)$ be the set $\{L(d) : d \in f\}$ be the corresponding set of
``incoming'' darts of~$f$. 
Note that $f \cap L(f)$ is empty if $f$ visits each edge at most once. 
  We say that a face $f$ together with one marked dart $d\in f$ %
  is a \emph{rooted face} and $d$ is called its \emph{root}.
  Let $g_k$ denote the \phantomsection \label{def:gk} number of rooted faces $f$ of unique length $k$ such that $f \in \Phi$.
We will calculate $g_k$ then use the following simple relation between \hyperref[def:hk]{$h_k$} and $g_k$:

\begin{observation}\label{obs:ftog}
  For each $k$, $1\le k\le m$, we have 
    $\frac{1}{2k}g_k \leq h_k \leq \frac{1}{k}g_k$.
\end{observation}

\begin{proof}
  Let $f$ be a face with $u(f)=k$.
  Consider an edge $e$ in~$f$, let $d$, $d'$ be the two darts of~$e$. 
  If $e$ appears only once on $f$ then exactly one of  %
  $d$, $d'$ is element of~$f$. Hence, only one of them can be the root.
  If $e$ appears twice on $f$ then both $d$ and~$d'$ can serve as the root.
  Hence, $k h_k \leq g_k \leq 2 k h_k$. 
\end{proof}

In the following lemmas, we show quite tight upper and lower bounds on $g_k$ that will be close to $(2m-1)(2m-3)(2m-5)\dots(2m-2k+1)$.
We will compute how many options there are to construct a rooted face with $k$ unique edges by fixing $L$ step-by-step.
We will look at the darts of one face $f$ of unique length $k$ in the order given by $R$ and~$L$, starting with the root of~$f$ denoted as~$d_1$.
More precisely, we say that darts $d_1, d_2, \ldots, d_{2k}$ form a \emph{rooted unique sequence} for some rooted face~$f$ with $u(f)=k$ 
if they are the sequence of darts in order of appearance in $f \cup L(f)$ 
starting with root $d_1$ and $d_2=L(d_1)$ excluding any repeats (obtained by traversing an edge the second time).

A part of a rooted unique sequence $d_1,d_2,\ldots,d_{2i}$ for $1\le i\le k<u(f)$ can be viewed as a
partial face starting with $d_1$ and leading to $d_{2i}$. %

To make our description of constructed faces more succinct, we define the following notation. 
We define $\Uforw(d)$ as the first unpaired dart in the sequence $R(d)$, $RLR(d)$, $(RL)^2R(d)$, \dots\ 
and similarly 
$\Uback(d)$ as the first unpaired dart in the sequence $R^{-1}(d)$, $R^{-1}LR^{-1}(d)$, $(R^{-1}L)^2R^{-1}(d)$, \dots\ 
See Figure~\ref{fig:Uforwback}. 

Note that the definition of $\Uback$ and $\Uforw$ depends on $L$, thus is being updated during the proof. 
Also note that $\Uback(d)$ is defined when $d$ is unpaired or $d = d_j$ for $j$ even, 
while $\Uforw(d)$ is defined when $d$ is unpaired or $d = d_j$ for $j$ odd. 

\begin{figure}
    \includegraphics[scale=1]{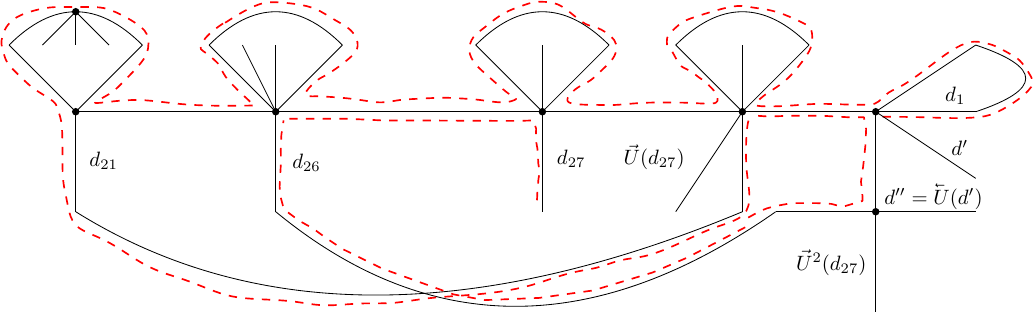}
    \caption{An illustration of the definition of~$\Uforw$, $\protect\Uback$.} 
    \label{fig:Uforwback}
\end{figure}

\begin{figure}
     \centering
     \begin{subfigure}[b]{\textwidth}
         \begin{minipage}{0.42\textwidth}
         \includegraphics[width=\textwidth]{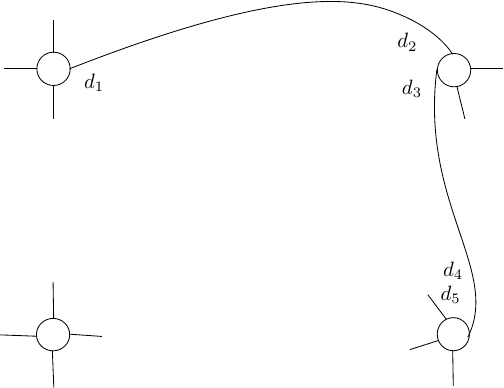}
       \end{minipage}\unskip\hspace{3em}
       \begin{minipage}{0.45\textwidth}
         \caption{There are $16$ total darts, so there are $16$ choices for $d_1$.  There are then $15$ choices for $d_2$.  Once $d_2$ is chosen, $d_3$ must be the next dart in the rotation system at that vertex.  There are $13$ choices for $d_4$.}
         \label{fig:ub1}
       \end{minipage}
     \end{subfigure}
     \\\bigskip
     \begin{subfigure}[b]{\textwidth}
         \begin{minipage}{0.42\textwidth}
         \includegraphics[width=\textwidth]{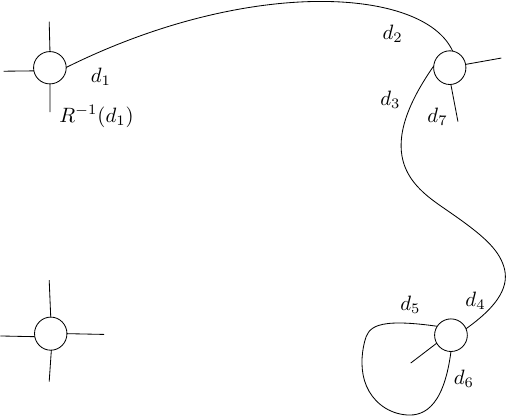}
       \end{minipage}\unskip\hspace{3em}
       \begin{minipage}{0.42\textwidth}
         \caption{There are 11 choices for $d_6$, and once this is chosen $d_7 = \Uforw(d_6)$ 
           is determined as shown in the diagram. Now suppose that $k=4$. 
           Then there is precisely one choice of $d_8$ which completes this walk into a face, and
         that choice is $R^{-1}(d_1)$.}
         \label{fig:ub2}
       \end{minipage}
     \end{subfigure}
     \\\bigskip
     \begin{subfigure}[b]{\textwidth}
         \centering
         \begin{minipage}{0.42\textwidth}
         \includegraphics[width=\textwidth]{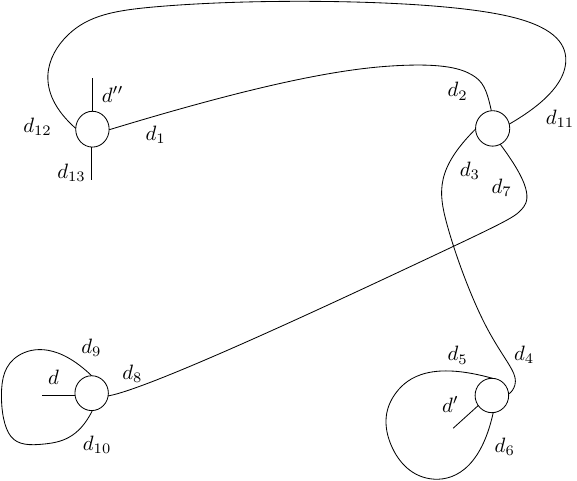}
       \end{minipage}\unskip\hspace{3em}
       \begin{minipage}{0.45\textwidth}
         \caption{We give an example of a sequence $d_1,\allowbreak d_2,\dots,d_{13}$ for $k=7$.  In this case, any choice of $d_{14}$ will close this walk into a face.  This is because $d$, $d'$ and $d''$ are exactly the $1$-open darts in this partial map.}
        \phantomsection\label{fig:ub3}
       \end{minipage}
     \end{subfigure}
    \caption{An illustration of the argument in Lemma \ref{lem:ubound}.}
    \label{fig:gkupper}
\end{figure}

\begin{lemma}[Upper-bound on $g_k$]\label{lem:ubound}
  We have $g_1 = 2m$.
  For $2\le k \le m$,
  \[g_k \leq 2m(2m-1)(2m-3)(2m-5)\dots(2m-2k+3).\]
\end{lemma}

\begin{proof}
  If $k=1$ then there are $2m$ choices for $d_1$.
  Then, in order to close a face with $d_2$, we have only one choice.
  Now, suppose that $k\ge 2$. Again, we have~$2m$ choices for~$d_1$.
  Then there are at most $2m-1$ choices for~$d_2$. However, one of them 
  closes a face of length~$1$, thus we have only $2m-2$ choices. 
  As $R$ is fixed, $d_3$ is determined. 

Now suppose that the sequence of darts is currently $d_1,d_2, \dots, d_{2i-1}$,
i.e., the first $2i-1$ darts of a rooted unique sequence are fixed, for $1\le i\le k-1$.
We then have at most $2m-(2i-1)$ choices for $d_{2i}$, as it cannot be any of the previous darts in the facial walk. 
Moreover one of those choices is $\Uback(d_1)$ which closes a face of length~$i$, thus 
we have only~$2m-2i$ choices. 
We follow the facial walk from $d_{2i}$ to $R(d_{2i})$ and possibly further (if this 
dart was already visited): we put $d_{2i+1} \coloneqq \Uforw(d_{2i})$. 

When we are about to choose the final dart~$d_{2k}$, we split into two cases.

\smallskip
\noindent \textbf{Case 1:} $d_{2k-1} \neq \Uback(d_1)$.~~\phantomsection \label{gkub:c1}
In this case, there is at most one choice of $d_{2k}$ which closes this facial walk into a face of unique length $k$.
This choice is setting $d_{2k} := \Uback(d_1)$.

\smallskip
Observe that in \hyperref[gkub:c1]{Case 1}, the last unique edge $d_{2k-1},d_{2k}$ always appears on $f$ only once. 
However, this does not need to be always the case; see \cref{fig:ub3} for such an example.
There, in particular, it is not true that we have at most one choice when choosing the last edge. 

\smallskip
  \noindent\textbf{Case 2:} $d_{2k-1} = \Uback(d_1)$.~~\phantomsection \label{gkub:c2}
Let $L$ be the permutation consisting of 2-cycles on edges defined by our choices of $d_1,d_2,\ldots,d_{2k-1}$.
A dart $d\notin\{d_1,d_2,\ldots,d_{2k-1}\}$ is called \emph{1-open} if $d\neq d_{2k-1}$ is the only unpaired dart on a 1-open temporary face in $R\circ L$.
  Observe that by choosing $d_{2k}$ we can close the face if and only if $d_{2k}$ is 1-open dart.
We therefore need an upper bound on the number of $1$-open darts.

To proceed with our calculations, let $N_i$ be the number of partial faces 
$f = (d_1,d_2, \dots, d_{2i-1})$ that are not yet closed (that is all darts in it are distinct). 
Further, let $S_i$ be the number of pairs $(f,d)$ where $f$ is a partial face
as above and $d$ is a $1$-open dart for~$f$. 

In Case~1 we have at most one way to close the possible face, this gives us at most~$N_k$ faces. 
In Case~2 we have at most one way for each 1-open dart, this gives us at most~$S_k$ faces. 
(We are overestimating, as $S_k$ also counts partial faces that lead to Case~1 and $N_k$ faces leading to Case~2.) 
Thus, we only need to estimate $N_k+S_k$. To this end, we first observe that $N_1 = 2m$ and $S_1 = 0$. 
We also have the following recurrence 
\begin{align}
    N_{i+1} &\le N_i \cdot (2m-2i) \label{eq:Ni} \\
    S_{i+1} &\le S_i \cdot (2m-2i-1) + N_i. \label{eq:Si} 
\end{align}
As explained above, we can extend $f=(d_1, \dots, d_{2i-1})$ by $d_{2i}$ that is distinct from the previous $2i-1$ darts 
and from $\Uback(d_1)$; this gives us equation~\eqref{eq:Ni} ($d_{2i+1}$ is defined uniquely). 
For equation~\eqref{eq:Si} we observe that when we extend $f$ to $f'$ as just described, the pairs 
$(f,d)$ can be extended to $(f',d)$, unless $d_{2i}=d$. This explains the first term in the equation. 
The other one counts pairs $(f',d)$ where $d$ is $1$-open in~$f'$ but not in~$f$. 
There is exactly one such pair for each~$f$, namely putting $d_{2i} = \Uforw^2(d_{2i-1})$, which will create 
a new $1$-open dart $d = \Uforw(d_{2i-1})$. 

Adding \eqref{eq:Ni} and~\eqref{eq:Si} we get $N_{i+1} + S_{i+1} \le (N_i+S_i) \cdot (2m-2i+1)$. 
Consequently, $N_k+S_k \le 2m(2m-1)\dots(2m-3)$, finishing the proof. 
\end{proof}

  \newcommand{\dbad}{d'}
  \newcommand{\dbadprev}{d''}
  \newcommand{\nullElem}{\emptyset}

\begin{lemma}[Lower-bound on $g_k$]\label{lem:lbound}
  $g_1 = 2m$.
  For $2\le k \le m$,
  \[g_k \geq 2m(2m-4)(2m-6)(2m-8)\dots(2m-2k).\]
\end{lemma}

We refer to Figures \ref{fig:gklower1} and \ref{fig:gklower2} for an illustration of arguments used in the proof of \cref{lem:lbound}.

\begin{proof}
There are $2m$ darts in total, and therefore, $2m$ choices for $d_1$.
Let $\dbad \df R^{-1}(d_1)$ and $\dbadprev \df \Uback(\dbad)$. %
We will estimate the number of choices of $d_{2i}$ for $i = 1, \dots, k$ (darts 
$d_{2i+1}$ are forced by the given rotations and partially constructed $L$). To keep the estimates manageable, 
we restrict our choices somewhat:  we will keep the following invariants true during our process of selecting the face: 
\begin{itemize}
    \item $\dbad \ne d_j$ for $j < 2k$ (as we want $\dbad = d_{2k}$)
    \item $\dbadprev \ne d_{2i}$ (as this would force $\dbad=d_{2i+1}$) 
    \item there are no 1-open faces %
       (otherwise, when $d_{2i-1}=\dbadprev$ then choosing 1-open dart as $d_{2i}$ would lead to $d_{2i+1}=\dbad$) 
\end{itemize}
However, we will not prevent $d_{2i+1}$ to be chosen as $\dbadprev$. 
If this happens, we will redefine $\dbadprev$ so that $\dbad$ is still available to be picked as $d_{2k}$, 
that is we again put $\dbadprev \df \Uback(\dbad)$ (which changes meaning, as the 
previously-defined $\dbadprev$ is no longer unpaired). 
In addition to the above, when selecting $d_{2i}$ we will forbid $d_{2i} = d^{o}\df \Uforw^2(d_{2i-1})$ 
which is the only choice that can create $1$-open face when paired with $d_{2i-1}$.
Consult \cref{fig:lb1} for the illustration.

Next, we estimate the number of choices along the way. The choices $d_1,\dbad,\dbadprev,d^o$ are not allowed for $d_2$.  
We therefore have at least~$2m-4$ choices for $d_2$.
Now suppose that the sequence of darts is currently $d_1,d_2, \dots, d_{2i}$, i.e., the first $2i$ darts of a rooted unique sequence are fixed, for $1\le i\le k-2$.
We also suppose that $\dbad$ and $\dbadprev$ are not among $d_1,d_2, \dots, d_{2i}$ and no $1$-open darts were created.
We define $d_{2i+1} \df \Uforw(d_{2i})$ as the next dart in the rooted unique sequence.
This may force $d_{2i+1}=\dbadprev$. %
We then have at least $2m-(2i+4)$ choices for $d_{2i+2}$, as it cannot be any of the previous darts in the facial walk 
and we do not allow $\dbad$, $\dbadprev$ or $d^o$ as choices.
(We have one more choice in case $d_{2i+1}=\dbadprev$, but we don't use this in our estimate.) 
Consult Figures \ref{fig:lb2}, \ref{fig:lb3}, and \ref{fig:lb4} for the illustration.

\begin{figure}
     \centering
     \begin{subfigure}[b]{\textwidth}
         \begin{minipage}{0.45\textwidth}
         \includegraphics[width=\textwidth]{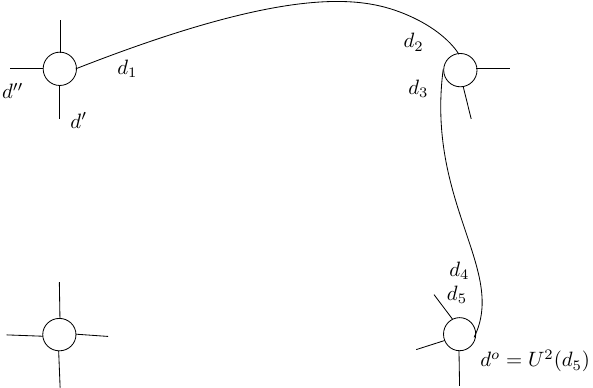}
       \end{minipage}\unskip\hspace{3em}
       \begin{minipage}{0.45\textwidth}
         \caption{
           From the start of the process, $d'$ and $d''$ are set as shown in the
           top picture. We disregard the choices to pair $d_i$ with $d'$, $d''$,
           and $d^o=\Uforw^2(d_5)$ (this option would create 1-open dart $\Uforw(d_5)$) as
           described.
       }
       \label{fig:lb1}
       \end{minipage}
     \end{subfigure}
     \\\bigskip
     \begin{subfigure}[b]{\textwidth}
         \begin{minipage}{0.45\textwidth}
         \includegraphics[width=\textwidth]{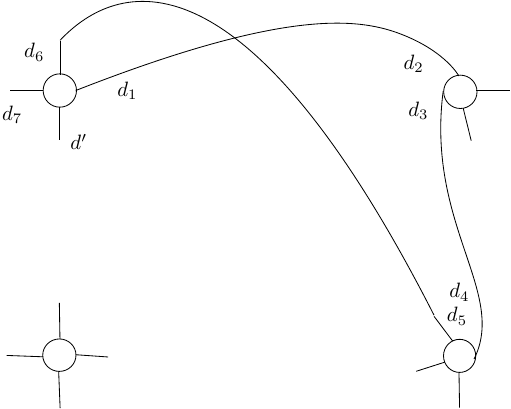}
       \end{minipage}\unskip\hspace{3em}
       \begin{minipage}{0.45\textwidth}
         \caption{Darts $d'$ and $d''$ stay the same until we choose $d_6$ to be $R^{-1}(d'')$. Then $d''$ became $d_7$, 
         effectively giving use one more choice for~$d_8$.}
         \label{fig:lb2}
       \end{minipage}
     \end{subfigure}
     \\\bigskip
     \begin{subfigure}[b]{\textwidth}
         \centering
         \begin{minipage}{0.45\textwidth}
         \includegraphics[width=\textwidth]{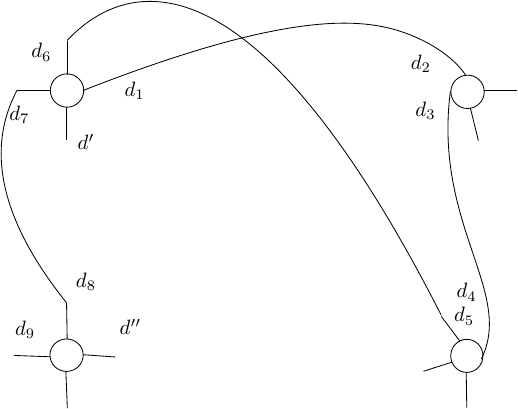}
       \end{minipage}\unskip\hspace{3em}
       \begin{minipage}{0.45\textwidth}
         \caption{%
         Once we choose $d_8$, we let $d''=\Uback(d')$. This happens unless degree of vertex where $d_8$ was chosen has degree two.}
         \label{fig:lb3}
       \end{minipage}
     \end{subfigure}
    \caption{An illustration of the argument in \cref{lem:lbound}. (Part I.)}
    \label{fig:gklower1}
\end{figure}
\begin{figure}
     \begin{subfigure}[b]{\textwidth}
         \centering
         \begin{minipage}{0.45\textwidth}
         \includegraphics[width=\textwidth]{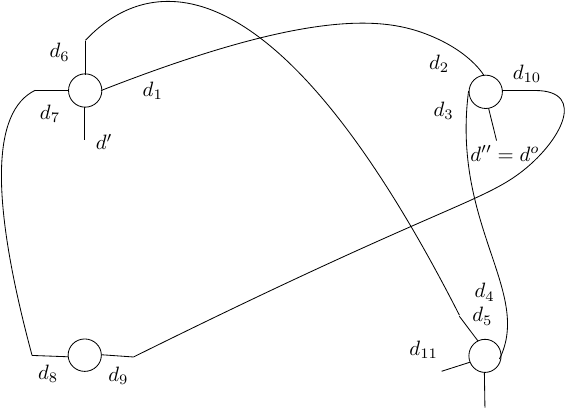}
       \end{minipage}\unskip\hspace{3em}
       \begin{minipage}{0.45\textwidth}
         \caption{In the case a vertex incident with $d_8$ has degree two then $d''$ stays $d_7$.
         Later on, when $d_{10}$ is chosen, then $d''$ is set as $\Uback(d')$. Moreover, observe that $d''$ is a forbidden choice due to another reason: the same dart is also $d^o=\Uforw^2(d_{11})$.}
         \label{fig:lb4}
       \end{minipage}
     \end{subfigure}
     \caption{An illustration of the argument in \cref{lem:lbound} (Part II.).}
     \label{fig:gklower2}
\end{figure}

After our facial walk passes through $k-1$ distinct edges we put, as before, $d_{2k-1} \df \Uforw(d_{2k-2})$
and finish the face by putting $d_{2k} \df \dbad$. 
Thanks to our invariants this choice is valid and closes the face after exactly $k$ unique edges were determined.
\end{proof}

In the estimate above, besides the obvious loss of not counting $\dbadprev$, we do not count the option that the last edge appears twice in the face. 

We are now ready to put these lemmas together to get the final result which is \cref{thm:multigraphs} with specified constants.

\begin{theorem} \label{thm:multi:main}
    Let $m$ denote the number of edges in the multigraph with degree sequence $\bf d$. Then
    \[ \frac{1}{2}\left(H_m-1\right) \leq \E(F_{\mathbf d}) \leq 2 \, H_m + 2.\]
\end{theorem}

\begin{proof}
  In both estimates we use \cref{lem:hk} as a base for computation of
  $\E(F_{\mathbf d})$ and \cref{obs:ftog} that compares $h_k$ with $g_k$.

  For the lower bound, we start of with estimate on $g_k$ given by \cref{lem:lbound} and conclude by the following computation:
  \begin{align*}
  \frac{1}{2}\left(H_m-1\right)  &\leq \sum_{k=1}^m \frac{1}{2k} \frac{m-k}{m} \leq \sum_{k=1}^m \frac{1}{2k} \frac{2m(2m-2k)}{(2m-1)(2m-3)} \\
  &\leq \sum_{k=1}^m \frac{1}{2k} \frac{(2m)(2m-4)(2m-6)(2m-8)\dots(2m-2k)}{(2m-1)(2m-3)(2m-5)\dots(2m-2k+1)}  \\
   &\leq \sum_{k=1}^m \frac{1}{2k} \frac{g_k}{(2m-1)(2m-3)(2m-5)\dots(2m-2k+1)} \\
   &\leq \sum_{k=1}^m \frac{h_k}{(2m-1)(2m-3)(2m-5)\dots(2m-2k+1)} = \E(F_{\mathbf d}).
  \end{align*}
    
  For the upper-bound we use the estimate on $g_k$ given by \cref{lem:ubound} and we conclude that: %
  \begingroup
\allowdisplaybreaks
  \begin{align*} 
  \E(F_{\mathbf d}) &= \sum_{k=1}^m \frac{h_k}{(2m-1)(2m-3)(2m-5)\dots(2m-2k+1)} \\
   &\leq \sum_{k=1}^m \frac{1}{k} \frac{g_k}{(2m-1)(2m-3)(2m-5)\dots(2m-2k+1)} \\
   &\leq \sum_{k=1}^m \frac{1}{k} \frac{(2m)(2m-1)(2m-3)(2m-5)\dots(2m-2k+3)}{(2m-1)(2m-3)(2m-5)\dots(2m-2k+1)} 
   = \sum_{k=1}^m \frac{1}{k} \frac{2m}{2m-2k+1} \\
    &< 2 + \sum_{k=1}^{m-1} \frac{m}{k(m -k)} 
    = 2 + \sum_{k=1}^{m-1} \left(\frac{1}{k}+\frac{1}{m-k}\right) \le 2 + 2H_m.\qedhere
  \end{align*}
  \endgroup
\end{proof}

\cref{thm:multi:main} is a direct analogue of \cref{thm:randomgraphsmallp} for multigraphs with loops; in next 
section we revisit the question for simple graphs. 

\begin{corollary}
  Let $G$ be a random multigraph with degree sequence $\mathbf d$.
  Then the probability that the number of faces in a random embedding of $G$ is
  greater than  $c (\ln(m_d)+2)$ is less than $\frac{2}{c}$.
\end{corollary}
\begin{proof}
    Observe that picking a random multigraph with degree sequence $\mathbf d$ then randomly embedding it gives a uniform at random chosen element from $\mathcal{M}_{\mathbf d}$.  Therefore, the result follows from \cref{thm:multi:main} and Markov's inequality.
\end{proof}

\subsection{Random simple graphs}\label{sec:rand_simple}
Let us fix some notation to be used throughout this section.
Given a degree sequence ${\mathbf d} = (t_{1},t_{2}, \dots, t_{n})$,
let $m_{\mathbf d} \df \frac 12 \sum_i t_{i}$ and $\lambda_{\mathbf d} \df \tfrac{1}{2m_{\mathbf d}}\sum_{i=1}^n \binom{t_{i}}{2}$.
We omit the subscript when ${\mathbf d}$ is clear from the context. 
If $t_{i} \leq d$ for all $i$, we refer to ${\mathbf d}$ as a \emph{$d$-bounded degree sequence}. 
In this section, we prove the following theorem.

\thmsimple*

As in Section~\ref{sec:rand_multi}, we may and will fix a rotation system $R \in C_{\mathbf d}$.
Let $\mathcal{M}_{ \mathbf d}^s$ denote the collection of simple maps with the fixed rotation $R$.
 Let $\Phi^s(k)$ denote the collection of possible faces of unique length $k$ in $\mathcal{M}^s_{\mathbf d}$.
 Moreover, let $G(n,{\mathbf d})$ and $G^s(n,{\mathbf d})$  denote, respectively, the collection of multigraphs and the collection of simple graphs on $n$ vertices with degree sequence~${\bf{}d}$. Bender and Canfield~\cite{BC78} showed
 that a random multigraph with degree sequence ${\mathbf d}$ is simple with probability 
 $(1+o(1))e^{-\lambda_{\mathbf d}-\lambda_{\mathbf d}^2}$. 
 In particular, this tells us that
\begin{equation}\label{eq:numsimple} |G^s(n,{\mathbf d})| = (1+o(1))e^{-\lambda_{\mathbf d}-\lambda_{\mathbf d}^2}|G(n,{\mathbf d})|.
\end{equation}
We continue by using a theorem of Bollob\'as and McKay to determine the number of maps containing a given $f \in \Phi^s(k)$.
Index the vertices in our model by $\{v_1,v_2,\ldots,v_n\}$ so that vertex~$v_i$ has degree~$t_{i}$. 
We say that $v_iv_j \in E(f)$ if a dart incident to $v_i$ is paired with a dart incident to $v_j$ in the face $f$. For each $f \in \Phi^s$ we define 
\[ \mu_f (\mathbf{d}) \df \frac{1}{2m} \sum_{v_iv_j \in E(f)} t_{i}t_{j}.\]
The following is a special case of Theorem 1 from \cite{BM86} which we will reformulate as an analog of \cref{lem:completion} for simple graphs; see \cref{cor:bollobas} below. 
\begin{theorem}[Bollob\'as and McKay {\cite[proof of Theorem 1]{BM86}}, reformulated]\label{thm:BollobasMckay}
  For each $d\ge 2$ and for each $\varepsilon \in (0,\tfrac{1}{d})$ there is a function $\delta(n)=o(1)$. 
  Let ${\mathbf d}$ be a $d$-bounded degree sequence of length $n$ such that $m=m_{\mathbf d} \ge (1+\varepsilon) n$.
  Let $f$ be a face on degree sequence $f_1,\dots,f_n$  (i.e., degrees of~$v_i$ in~$f$ is~$f_i$). 
  Let $t_{i}' \coloneqq t_{i} - f_i$ for $i = 1, \dots, n$ and let $\mathbf{d'} = (t_{1}', \dots, t_{n}')$.  Then if we pick a map uniformly at random from those in $\mathcal{M}_{\mathbf d}$ which contain $f$, the probability that this map is simple is
  \begin{align}
    (1+\delta(n))e^{-\lambda_{\mathbf{d'}} - \lambda_{\mathbf{d'}}^2-\mu_f(\mathbf{d'})}. \label{eq:BM}
  \end{align}
\end{theorem}

Let us note that the statement of \Cref{thm:BollobasMckay} follows directly from Equation (2) within the proof of Theorem 1 in~\cite{BM86}.

We want to obtain a bound for the number of maps containing a face with unique length $k$, so we give the following simple corollary.

\begin{corollary} \label{cor:bollobas}
  For each $d\ge 2$ and for each $\varepsilon \in (0,\tfrac{1}{d})$ there is a function $\delta(n)=o(1)$. 
  Let ${\mathbf d}$ be a $d$-bounded degree sequence of length $n$ such that $m=m_{\mathbf d} \ge (1+\varepsilon) n$.
  If $k \le m/2$ then for any $f \in \Phi^s(k)$ the number of simple maps with degree sequence $\mathbf d$ that 
  contain~$f$ is at most $|C_{2^{m-k}}|$ and at least
  \[(1 + \delta(n))e^{-\binom{d}{2}-\binom{d}{2}^2-\frac{d^2}{2}}|C_{2^{m-k}}|.\]
\end{corollary}
\begin{proof}
  Let $f$ be a face on degree sequence $f_1,\dots,f_n$, let $t_{i}' = t_{i} - f_i$ for $i = 1, \dots, n$ and let 
  $\mathbf{d'} = t_{1}', \dots, t_{n}'$. 

  The number of (not necessarily simple) maps on degree sequence $\mathbf{d'}$
  is $|C_{2^{m-k}}|$ by \Cref{lem:completion}, proving the upper bound.

  For the lower bound, by \cref{thm:BollobasMckay} the probability of a map in
  $\mathcal{M}_{\mathbf d}$ containing $f$ being simple is
  $(1+\delta(n))e^{-\lambda_{\mathbf{d'}} - \lambda_{\mathbf{d'}}^2-\mu_f(\mathbf{d'})}$.  
  Since ${\mathbf d}$ is $d$-bounded, we have 
  $\lambda_{\mathbf{d'}} \leq \frac{1}{2m-2k} \sum_{i=1}^n \binom{d}{2} \leq \frac{n}{m} \binom{d}{2} \le \binom d2$.  
  Similarly,
  \[\mu_f(\mathbf{d'}) \leq \frac{1}{2m-2k} \sum_{v_iv_j \in E(f)} d^2 = d^2 \frac{k}{m} \le \frac{d^2}2. \qedhere\] 
\end{proof}

Recall that in the previous section we defined \hyperref[def:hk]{$h_k$} as the number of faces of unique length $k$, and \hyperref[def:gk]{$g_k$} as the number of rooted faces of unique length $k$.  We define a simple face as a face which has no loops or parallel edges in it.  Then we define $h_k^s$ as the number of simple faces of unique length $k$, and $g_k^s$ as the number of rooted simple faces of unique length $k$.
It is easy to observe that the analog of \cref{obs:ftog} holds even for the number of faces of simple graphs.

\begin{observation}\label{obs:ftog:simple}
  For each $k$, $1\le k\le m$, we have 
    $\frac{1}{2k}g^s_k \leq h^s_k \leq \frac{1}{k}g^s_k$.
\end{observation}

Next we prove a variant of \cref{lem:lbound} for simple graphs. 
Observe that $h_1^s=h_2^s=0$ as $t_i\ge 2$ for all $i$.

\begin{lemma} \label{lem:gklowersimple}
  Let $3\le k \le m - \tfrac{d^2}{2}$. Then
  \[h_k^s \geq \frac{1}{2k} 2m(2m-d^2)(2m-d^2-2)\dots(2m-d^2-2k+4).\]
\end{lemma}

\begin{figure}
     \centering
     \begin{subfigure}[b]{\textwidth}
         \begin{minipage}{0.45\textwidth}
         \includegraphics[width=\textwidth]{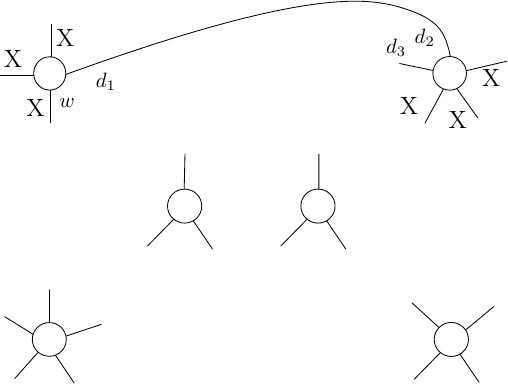}
       \end{minipage}\unskip\hspace{3em}
       \begin{minipage}{0.45\textwidth}
         \caption{
         Here there are $24$ choices for $d_1$.  There are only $20$ choices for $d_2$.  This is because all of the other choices at the starting vertex (denoted as $w$) are disallowed, since our graph is simple and cannot have multiple edges.  The choice of $d_3$ is then determined, and the disallowed choices for $d_4$ are marked with an $X$.}
         \label{fig:lb1simple}
       \end{minipage}
     \end{subfigure}
     \\\bigskip
     \begin{subfigure}[b]{\textwidth}
         \centering
         \begin{minipage}{0.45\textwidth}
         \includegraphics[width=\textwidth]{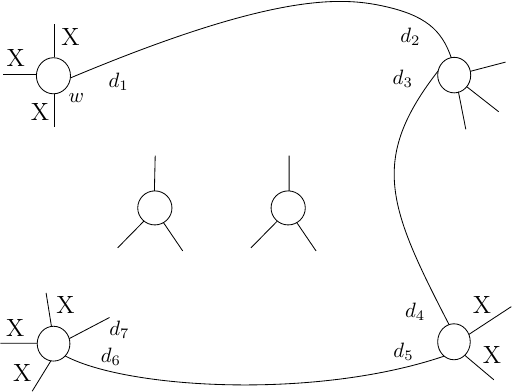}
       \end{minipage}\unskip\hspace{3em}
       \begin{minipage}{0.45\textwidth}
         \caption{In this example, there are only $10$ choices for $d_8$.  The other three darts at the same vertex as $d_7$ aren't allowed as the graph cannot have loops.  The darts at the bottom right vertex are also dissallowed, as we cannot have multiple edges.  Finally, all darts around vertex $w$ are dissallowed as well.}
         \label{fig:lb2simple}
       \end{minipage}
     \end{subfigure}
    \caption{An illustration of the proof of Lemma \ref{lem:gklowersimple}.}
\end{figure}

We follow a very similar proof as in \cref{lem:lbound}, see Figures \ref{fig:lb1simple} and \ref{fig:lb2simple} for an example of the process.  The difference is that at each step when picking $d_{2i}$ we also disallow any choices which add a parallel edge or loop.

\begin{proof}[Proof of \cref{lem:gklowersimple}]
We will count rooted simple faces.  Then, using \cref{obs:ftog:simple} we will obtain $h_k^s \geq g_k^s/2k$.
      
There are $2m$ choices for $d_1$.  Since we are not allowing any loops in the face, we cannot choose any other darts at vertex $w$ incident with $d_1$.
This means that we have at least $2m-d > 2m-d^2$ total choices for $d_2$.  
As before, $R$ being fixed means $d_3$ is determined.

Now suppose that the sequence of darts is currently $d_1,d_2, \dots, d_{2i}$.
As before, we have to put $d_{2i+1} := \Uforw(d_{2i})$. 
There are several different choices of $d_{2i+2}$ which we disallow:
\begin{itemize}
    \item Any of the choices $d_1,d_2,\dots,d_{2i+1}$.
    \item Any dart incident with vertex $w$.
    \item Any choice which adds a loop or multiple edge to the face.
    \item Any choice which adds a 1-open face.
\end{itemize}
We upper bound the total number of disallowed choices.
Suppose $d_{2i+1}$ is at vertex $v$, then there are at most $d-2$ unpaired darts present at $v$.  Pairing into any unpaired dart at $v$ will create a loop. 
Pairing into any dart at a vertex $u$ for which there is already an edge between~$v$ and~$u$ will add a multiple edge.
There are at most $d-1$ possible edges incident with $v$ and for each adjacent $u$ at most $d-2$ available darts
(that is darts not forbidden by the other rules). 
Indeed, if $u=w$ then no darts are allowed (and counted separately) and if $u$ is any other vertex already adjacent to $v$ then at least two darts incident with $u$ (darts $d_j,d_{j'}$ for some $j \ne j'\le 2i$) are not available, 
that is already counted as forbidden.
Additionally, we disallow $d-1$ darts at $w$.
Therefore, there are at most $d^2-2d$ choices that add a loop or multiple edge. 
There is at most one choice which adds a 1-open face, by the same reasoning as in the proof of Lemma \ref{lem:lbound}.
In total, we have at most $(2i+1) +d^2-2d + d - 1 + 1 \leq d^2 + 2i - (d-1)$ disallowed choices for $d_{2i+2}$. %
Therefore, we have at least $2m-d^2-2i$ choices for $d_{2i+2}$.  %
    
     We need a little extra analysis for the final step, as we must ensure that when completing the face we do not add a loop or a multiple edge.  After our facial walk passes through $k-2$ distinct edges,
     we have the sequence $d_1,d_2,\dots,d_{2k-3}$ and must make a choice for $d_{2k-2}$.  At this step, we disallow all the $d^2 + 2(k-2)-(d-1)$ choices as in the previous case. 
      When choosing $d_{2k-2}$ we disallow 
      any darts incident with the only vertex incident with $w$.
      This is at most $d-2$ darts.
      Therefore, we have at least 
      \[
      2m-(d^2+2(k-2)-(d-1))-(d-2) > 2m-d^2-2(k-2)>0
    \]
      choices at this step.

      Now at the final step, we choose $d_{2k} = R^{-1}(d_1)$.  Our disallowed choices mean that this choice is always possible and that the edge $(d_{2k-1}, d_{2k})$ will not add a loop or multiple edge. 
\end{proof}

\begin{proof}[Proof of \cref{thm:randomgraphsmallp}]
  Select a uniformly random $M \in \mathcal{M}^s_{\mathbf d}$. For each $f \in \Phi^s$, 
  let $X_f$ denote the indicator random variable for the event ``$f$ appears in $M$''.
  Using \cref{cor:bollobas} and \cref{eq:numsimple} we get
\begin{align}
  \E[F_{\mathbf d}^s] &= \sum_{f \in \Phi^s} \E[X_{f}] \label{eq:xf}\\
    &\leq \sum_{k=3}^m  h_k^s \frac{|C_{2^{m-k}}|}{|G^s(n,\mathbf{d})|} \nonumber\\
    &= \frac{1}{(1+o(1)) e^{-\lambda_{\mathbf d}-\lambda_{\mathbf d}^2}}
       \sum_{k=3}^m  h_k^s \frac{|C_{2^{m-k}}|}{|G(n,\mathbf{d})|}. \nonumber
\end{align}

Using the trivial bound $h_k^s \leq h_k$, \cref{eq:framework}, and \cref{thm:multi:main} we obtain the upper bound.
Recall that as $d$ is a constant we have that $\lambda_{\mathbf d}$ is also a constant which yields the resulting bound.
\begin{align*}
    \E[F_{\mathbf d}^s] &\leq \frac{1}{(1+o(1))e^{-\lambda_{\mathbf d}-\lambda_{\mathbf d}^2}} \sum_{k=3}^m  h_k \frac{|C_{2^{m-k}}|}{|C_{2^m}|}  \\
    &= \frac{1}{(1+o(1))e^{-\lambda_{\mathbf d}-\lambda_{\mathbf d}^2}} \E[F_{\mathbf d}] = O(\log(n)).
\end{align*}
   
For the lower bound, recall \cref{eq:c2mk} from \cref{lem:hk} that 
   \[\frac{|C_{2^{m-k}}|}{|G(n,\mathbf{d})|}=  \frac{|C_{2^{m-k}}|}{|C_{2^{m}}|} = \frac{1}{(2m-1)(2m-3)(2m-5)\dots(2m-2k+1)}.\]

   Combining this with Lemma \ref{lem:gklowersimple} we obtain the following for $2 \le k \le m/2$. 
    \begin{align*}
        \frac{2k \, h_k^s|C_{2^{m-k}}|}{|C_{2^m}|} 
          &\geq \frac{2m(2m-d^2)(2m-d^2-2)\dots(2m-d^2 - 2k + 4)}{(2m-1)(2m-3)\dots(2m-2k+1)} \\
          &= \frac{2m}{2m-1} \prod_{i = 0}^{k-2} \Bigl(1 - \frac{d^2-3}{2m-2i-3} \Bigr)  \\
          &\ge \Bigl(1 - \frac{d^2}{m}\Bigr)^k \ge e^{-2\frac{d^2}{m}k} \ge e^{-d^2}
    \end{align*}
    On the last line we used that $1-x \ge e^{-2x}$ for $x \in [0,1/2]$. 
    
    Putting this together with \cref{cor:bollobas} (recall that in what follows $o(1)$ depends on $\varepsilon$ and $d$) and \cref{eq:numsimple} as in \cref{eq:xf} we get the required result: 
    \begin{align*}
  \E[F_{\mathbf d}^s] &= \sum_f \E[X_f] \geq 
    \sum_{k=3}^{\floor{m/2}}  h_k^s \frac{(1 + o(1))e^{-\binom{d}{2}-\binom{d}{2}^2-\frac{d^2}{2}}|C_{2^{m-k}}|}
                                 {|G^s(n,\mathbf{d})|} \\
       &\ge \frac{(1 + o(1))e^{-\binom{d}{2}-\binom{d}{2}^2-\frac{d^2}{2}}}
                 {(1+o(1))e^{-\lambda_{\mathbf d}-\lambda_{\mathbf d}^2}} 
            \sum_{k=3}^{\floor{m/2}} \frac{h_k^s |C_{2^{m-k}}|}{|G(n,{\mathbf d})|} \\ 
       &\geq \frac{(1 + o(1))e^{-\binom{d}{2}-\binom{d}{2}^2-\frac{d^2}{2}}}
                  {2(1+o(1))e^{-\lambda_{\mathbf d}-\lambda_{\mathbf d}^2}}
            e^{-d^2} \left(H_{\floor{m/2}}-1\right) 
            = \Omega(\log(n)). %
    \end{align*}
    As before, the last equality follows as $d$ and hence $\lambda_{\mathbf d}$ are constants.
\end{proof}

\section{Open Problems}\label{sec:open}

We showed in \cref{cor:almostallpolylog} that almost all dense graphs (when setting $p\ge \tfrac{1}{\text{polylog }n}$) have a polylogarithmic average number of faces.  Then, in \cref{sec:randomBnd} we showed that random sparse graphs have a logarithmic average number of faces. The same Markov's inequality argument as in \cref{cor:almostallpolylog} gives that most sparse graphs have a logarithmic average number of faces. This leads us to the conjecture that this property holds for almost all graphs, without any density condition on the edges.

\begin{Conjecture}
    For any $p(n): \mathbb{N} \rightarrow [0,1]$, almost all graphs in $G(n,p)$ satisfy $\E[F] = O(\log(n))$.
\end{Conjecture}

This conjecture would follow from the stronger statement of Conjecture \ref{con:random}.  Conjecture \ref{con:random} can also be stated in terms of the closely related model of random graphs with $n$ vertices and $M$ edges.

\begin{Conjecture}
  The expected number of faces in a random embedding of a random graph ${G\in G(n,M)}$ is \[(1+ o(1))\ln(M).\]
\end{Conjecture}

A main result of the paper was that the complete graph does have a logarithmic number of expected faces.  A large family of examples of graphs on $n$ vertices with $\E[F] = \Theta(n)$ are given in \cite{CHMMS22}.  However all of these examples have maximum degree $O(1)$ with respect to the number of vertices.  We were unable to find any examples of dense graphs with such a large number of average faces, which leads us to the next conjecture.

\begin{Conjecture}
    Let $G$ be a graph on $n$ vertices with minimum vertex degree $\Omega(n)$.  Then $G$ satisfies $\E[F] = \Theta(\log(n))$.
\end{Conjecture}

Theorem \ref{thm:logBound} confirms this conjecture for the complete graph.  The multiplicative constant in our bound is not optimal, we restate the conjecture given in the introduction which suggests a possible optimal constant.

\begin{Conjecture}[{\cite[page~289]{Mauk1996}}]
    The expected number of faces in a random embedding of the complete graph $K_n$ is $2 \ln n+O(1)$.
\end{Conjecture}

Another natural line of enquiry would be to extend these results to non-orientable surfaces.  One natural way to define a random embedding of a graph on a non-orientable surface is to randomly choose a rotation system, and randomly choose a signature for all the edges in the graph, with probability $1/2$ of being either sign.  From data, we expect a similar result to hold for non-orientable random embeddings of $K_n$ under this definition.

\begin{Conjecture}
    The expected number of faces in a non-orientable random embedding of the complete graph $K_n$ is at most $\ln(n) + O(1)$.
\end{Conjecture}

 We think that in general, a similar property should hold for random embeddings of all graphs.

\begin{Conjecture}
    Let $F^-$ be the random variable for the average number of faces in a non-orientable random embedding of some graph $G$.  Then $\E[F^-] \leq \E[F]$.
\end{Conjecture}

It is an easy exercise to check this conjecture's validity on some toy models.  In particular, the chain of triangles joined by cut edges considered in \cite{loth2022expected} satisfies this property.  Also, an analysis of \hyperref[def:rpA]{Random Process A} gives the upper bound of $\E[F^-] \leq \tfrac{1}{2}\E[F] + 1$ for the dipole, which is the graph with $2$ vertices joined by $m$ edges.  Computer data ran on some more general graphs gives evidence for some small values of $n$.

Lastly, it would be of interest to understand higher moments of $F$.
This is widely open even for a complete graph.
In this paper, we only obtain an upper bound (with respect to $k$) for the second moment of the number of potential faces on $n-k$ vertices in $K_n$\lv{; recall \Cref{lem:power_of_expectation} for details}. 

\bibliographystyle{plainurl}
\bibliography{bibliography}

\lv{
\appendix
\section{Appendix}\label{sec:A}
\appendixText
}

\end{document}